\documentclass[12pt]{amsart}
\usepackage[utf8]{inputenc}
\input xy
\xyoption{all}

\usepackage{amsfonts, amsthm, amsmath, amssymb,bm}
\usepackage{amssymb,amscd}
\usepackage{mathtools}
\usepackage[numbers]{natbib}
\usepackage{esint}
\usepackage{mathabx}
\usepackage{xcolor}
\usepackage{enumerate}
\usepackage{cases}

\usepackage{ stmaryrd }
\newcommand{\Aa}{\mathbb{A}}

\newcommand{\Rr}{\mathbb{R}}
\newcommand{\Cc}{\mathbb{C}}
\newcommand{\Qq}{\mathbb{Q}}

\newcommand{\bs}{\backslash}
\newcommand{\mc}{\mathcal}
\newcommand{\mf}{\mathfrak}
\newcommand{\ol}{\overline}
\newcommand{\on}{\operatorname}
\newcommand{\wc}{\widecheck}

\renewcommand{\d}{\,\mathrm{d}}
\newcommand{\GL}{\mathrm{GL}}
\newcommand{\PGL}{\mathrm{PGL}}

\newcommand\IP[1]{\langle #1\rangle} 

\newtheorem{thm}{Theorem}

\newtheorem{prop}{Proposition}[section]
\newtheorem{cor}[thm]{Corollary}
\newtheorem{rmk}{Remark}[section]

\renewcommand{\o}{\mathfrak{o}}
\newcommand{\M}{\mathcal{M}}
\renewcommand{\P}{\mathcal{P}}

\newcommand{\R}{\mathbb{R}}
\newcommand{\A}{\mathbb{A}}

\newcommand{\C}{\mathbb{C}}

\newcommand{\Z}{\mathbb{Z}}

\newcommand{\B}{\mathcal{B}}
\newcommand{\D}{\mathcal{D}}

\newcommand{\p}{\mathfrak{p}}
\newcommand{\q}{\mathfrak{q}}

\newcommand{\loc}{\mathrm{loc}}
\newcommand{\Eis}{\mathrm{Eis}} 
\newcommand{\gen}{\mathrm{gen}}

\newcommand{\diag}{\mathrm{diag}}
\newcommand{\aut}{\mathrm{aut}}
\newcommand{\vol}{\mathrm{vol}}
\newcommand{\Ad}{\mathrm{Ad}}

\newtheorem{lemma}{Lemma}[section]

\begin{document}

\title[reciprocity and non-vanishing]{Spectral reciprocity for $\mathrm{GL}(n)$ and simultaneous non-vanishing of central $L$-values}

\author{Subhajit Jana}
\address{Queen Mary University of London, Mile End Rd, London E1 4NS, United Kingdom}
\email{s.jana@qmul.ac.uk}

\author{Ramon Nunes}
\address{Departamento de Matem\'atica, Universidade Federal do Cear\'a, Campus do Pici, Bloco 914, 60440-900 Fortaleza-CE, Brasil}
\email{ramon@mat.ufc.br}

\date{}

\begin{abstract}
Let $F$ be a totally real number field and $n\ge 3$. Let $\Pi$ and $\pi$ be cuspidal automorphic representations for $\PGL_{n+1}(F)$ and $\PGL_{n-1}(F)$, respectively, that are unramified and tempered at all finite places. We prove simultaneous non-vanishing of the Rankin--Selberg $L$-values $L(1/2,\Pi\otimes\widetilde{\sigma})$ and $L(1/2,\sigma\otimes\widetilde{\pi})$ for certain sequences of $\sigma$ varying over cuspidal automorphic representations for $\mathrm{PGL}_n(F)$ with conductor tending to infinity in the level aspect and bearing certain local conditions. Along the way, we also prove a reciprocity formula for the average of the product of Rankin--Selberg $L$-functions $L(1/2,\Pi\otimes\widetilde{\sigma})L(1/2,\sigma\otimes\widetilde{\pi})$ over a conductor aspect family of $\sigma$.
\end{abstract}

\maketitle

\tableofcontents

\section{Introduction}

\subsection{Motivation: non-vanishing of the central $L$-values} 

Fix a set of complex numbers $s:=\{s_1,\dots,s_k\}$ with $\Re(s_i)=0$ and a natural number $n\ge 2$. It is a folklore conjecture that there are infinitely many cuspidal automorphic representations $\sigma$ for $\GL(n)$ over a number field $F$ with a fixed central character such that the $L$-values $L(1/2+s_i,\sigma)$ do not vanish for $1\le i\le k$. 

Naturally, for a given $n$ there is an upper bound of $k$ in terms of $n$ for which the current technology is potent enough to prove the above conjecture. This is because the product $L$-function $\prod_iL(1/2+s_i,\sigma)$, whose degree is $kn$, has growing complexity as $k$ grows. Thus, finding large $k$ in terms of $n$ for which we can verify this conjecture becomes a natural and challenging problem.

In this paper one of the main theorems (Theorem \ref{asymptotics}, Corollary \ref{non-vaninshing}), \emph{informally}, yields that $k$ can be taken as large as $2n$ for general $n$. More precisely: the product $\prod_i L(1/2+s_i,\sigma)$ can also be viewed as the $\GL(n)\times\GL(k)$ Rankin--Selberg $L$-function $L(1/2,\sigma\otimes E_s)$ where $E_s$ is the minimal parabolic Eisenstein series $\boxplus_{i=1}^k|.|^{s_i}$. At this point it becomes natural to ask about the non-vanishing of the Rankin-Selberg $L$-function $L(1/2,\Sigma\otimes\sigma)$ where $\Sigma$ is \emph{any} automorphic representation for $\GL(k)$. This includes the case where $\Sigma$ is an Eisenstein series and the one where $\Sigma$ is a cuspidal representation. Presumably, the problem is, arithmetically, the most difficult when $\Sigma$ is cuspidal (so that $L(1/2,\Sigma\otimes\sigma)$ is not necessarily factorable) and the least difficult when $\Sigma$ is a minimal Eisenstein series (so that $L(1/2,\Sigma\otimes\sigma)$ is totally factorable).

In this paper, we answer the question affirmatively when $\Sigma$ is the $\GL(2n)$ Eisenstein series $\Pi\,\boxplus\,\pi$ where $\Pi$ and $\pi$ are certain cuspidal representations for $\GL(n+1)$ and $\GL(n-1)$, respectively. In terms of factorization, this is the second most difficult case, right after the cuspidal case. In particular, this answers the question affirmatively when $\Sigma$ is a cuspidal representation of $\GL(n+1)$ or $\GL(n-1)$. 

We stress that even the non-vanishing of $L(1/2,\Sigma\otimes \sigma)$ when $\Sigma$ is a cuspidal automorphic representation for $\GL(n+1)$ is new for $n>2$.
For comparison, Tsuzuki \cite{Tsuzuki20212.07811} proves non-vanishing for $\Sigma$ being a minimal Eisenstein series of $\GL(n-1)$.  Finally, in low degree, slightly better results are known. For $n=2$, Blomer--Li--Miller \cite{BLM2019spectral} show non-vanishing for $k=4$ but with $\Sigma$ cuspidal and for $n=1$, Luo \cite{Luo2005nonvanishing} shows that one may take $k=3$ for characters of suitably factorable moduli.

\subsection{Reciprocity formul{\ae}}

One of the main inputs going into the proof of simultaneous non-vanishing is a \emph{reciprocity formula}. A reciprocity formula, in our context, describes an identity between moments of $L$-values attached to two apparently different families of automorphic representations. The reciprocity formul{\ae}, apart from being aesthetically pleasing, often allow one to understand averages of $L$-values without (or only mildly) dealing with the geometric side of a trace formula. Instead, the formul{\ae} arrange it so that the complexity of the families of $L$-functions drops from one side to the other, much akin to what happens in the Poisson and Voronoi summation formul{\ae}.

To demonstrate the complications of studying the geometric side of a trace formula in high-rank groups, it suffices to look at the Kuznetsov trace formula: while in $\GL(2)$, there is a good amount of literature dedicated to sums of Kloosterman sums and many applications thereof, for $\GL(3)$ the study of these objects is no easy matter, as can be perceived, for example, from the works of Blomer and Buttcane, \emph{e.g.} \cite{BlBu2019global}. Some applications can be seen in \cite{Blomer2013app,BBM2017app,BlBu2020sub}. For larger $n$ we refer to \cite{Blomer1906.07459, GSW2021orthogonality} as examples of application of Kuznetsov formula where the geometric side needed careful analysis.

One of our main results, Theorem \ref{reciprocity}, describes a higher-rank reciprocity formula which proves a relation between moments of Rankin--Selberg $L$-values for $\GL(2n)\times\GL(n)$ over two different families of representations.

\subsubsection{A recapitulation of previous formul{\ae}}
\subsubsection*{Different groups: Motohashi and generalizations}

Motohashi's formula is the first kind of spectral identity to appear in the literature. It is also the most studied one and has several interesting applications. The story begins with the work of Motohashi \cite{Mot1993explicit}, where he proved a formula of the shape
$$\int_{\Rr}|\zeta(1/2+it)|^4h(t)\d t = \sum_{f}L(1/2,f)^3\wc{h}(t_j)+(\ldots),$$
where the sum on the right-hand side runs over cusp forms of $\PGL(2)$. We have swept under the rug via the symbol $(\ldots)$ the contribution of the Eisenstein series and some degenerate terms. Moreover, the transformation $h\rightsquigarrow
\wc{h}$ is explicitly given via a series of integral transformation. This allowed Motohashi to give an asymptotic formula for the fourth-power moment of the Riemann zeta function.
Certain special cases of the reverse transform $\wc{h}\rightsquigarrow h$ are known. In particular, this was used by Ivi{\'c} \cite{ivic2001hecke} in order to obtain Weyl-type subconvex bound for $L(1/2,f)$ in the spectral aspect.
Around the same time, Conrey and Iwaniec \cite{CI2000cubic} studied a similar problem in the level aspect, with forms twisted by quadratic Dirichlet characters. Even though they did not phrase their results in such terms, a reciprocity formula can be inferred from their proof. An \textit{almost} exact formula (with an error term) was later found by Petrow \cite{Petrow2015twisted}. More recently, Petrow and Young (\emph{cf}.\ \cite{PY2020Weyl} and \cite{PY1908.10346}) have vastly generalized Conrey and Iwaniec's work to general Dirichlet characters. Here, again there is no effort in producing exact formul{\ae} but one can notice that an underlying version of Motohashi's formula is used (in both directions!). A bit later, and with a different application in mind, Blomer \textit{et al}.\ \cite{BHKM2020Motohashi} have yet another form of the identity. Finally, a period theoretic interpretation was given by Reznikov \cite{Reznikov2008RS}; see also \cite{MV2010subconvexity}. Explicit versions were later obtained by Nelson \cite{Nelson1911.06310} and Wu \cite{wu2022motohashi}. There also are cuspidal versions of Motohashi's formula where the cube of $L(1/2,f)$ gets replaced by a Rankin-Selberg $L$-function and the left-hand side gets replaced by a mixed-moment of $L$-functions of degrees $3$ and $1$. In this setting, asymptotic formul\ae for the right-hand side were initially obtained by Li \cite{li2011bounds} in the $t$-aspect) and Blomer \cite{blomer2012subconvexity} in the level aspect. Although the latter hinted on a spectral identity, the first explicit formula in the context was given by Kwan \cite{kwan2024spectral}, via a period-theoretic approach.

\subsubsection*{Same group: $\GL(4)\times \GL(2) \rightsquigarrow \GL(4)\times \GL(2)$} In this case, we talk about a formula where in both sides we average over (possibly different) families of automorphic representations of the same group, namely $\GL(2)$. Moreover the degree of the $L$-functions is also the same, namely $8$. We divide them, however, in three different cases according to how the $L$-function factorizes and this corresponds to the number of partitions of $4$ as a sum of two non-negative integers.

\begin{enumerate}

\item $\mathbf{4=2+2}$.
This is probably the most natural situation. It amounts to studying the fourth moment of $L(1/2,\sigma)$ or the second moment of $L(1/2,\sigma\otimes \sigma_0)$ over a family of $\sigma$, where $\sigma_0$ is a fixed automorphic form of $\GL(2)$.  If one is to be rigorous, the latter does not exactly fit the description above as we may have to deal with certain periods which are only related to $L$-functions up to taking the square of its absolute value (\emph{cf}.\ \cite{Ichino2008Trilinear}). In the fourth-moment case, this dates back to unpublished work of Kuznetsov and Motohashi; see \cite{motohashi2003functional}.
In the Rankin-Selberg case, a reciprocity formula is implicit in \cite{MV2010subconvexity}, and an explicit version can be found in \cite{AK2018level}. Later, a period approach to the equality was developed by Zacharias in \cite{Zach2019Periods} and \cite{Zach2019Periods2}. The latter also deals with regularization and contains a Weyl-type subconvex bound for $L(1/2,\sigma)$, where $\sigma$ is a $\PGL(2)$ cuspidal representation of prime level tending to infinity.

We may think the above example as the $n=2$ case for a general reciprocity formula $\GL(2n)\times\GL(n)\rightsquigarrow\GL(2n)\times\GL(n)$. For $n>2$, one can see a certain glimpse of a reciprocity formula in the works of Blomer \cite{Blomer2012Period} and the first named author \cite{Jana2020RS}. Here one has to be even more flexible when employing the term ``reciprocity'' as on one side of the formula the periods are neither known nor expected to be related to $L$-functions. Nevertheless these results fit in the overall strategy where one can bypass the need of a delicate study of geometric side (in the sense of the trace formula).

\item $\mathbf{4=3+1}$.
Here we are concerned with averages of
$$L(1/2,\Pi\otimes \widetilde{\sigma})L(1/2,\sigma),$$
where $\Pi$ is a fixed representation of $\GL(3)$. In the particular case where $\Pi$ is the symmetric square of some representation of $\GL(2)$, these averages have been studied in connection to quantum unique ergodicity and the $L^4$-norm problem; see for instance \cite{HK2020random} and the references therein.  Here, the first instances of spectral fomrul{\ae} appeared in the works of Blomer and Khan; see \cite{BlKh2019reciprocity} and \cite{BlKh2019uniform}. However, some of the ideas involved in proving such formul\ae can be traced back to works o Li \cite{li2009central} and Khan \cite{Khan2012simultaneous}.
In both \cite{BlKh2019reciprocity} and \cite{BlKh2019uniform}, the main applications are when $\Pi$ is an Eisenstein series, where one may obtain strong exponents in the subconvexity problem for $\GL(2)$. In \cite{Nunes2020reciprocity}, the second named author proved a version of the main result in \cite{BlKh2019reciprocity} valid for number fields via a period theoretic approach. This is not a complete generalization of \cite{BlKh2019reciprocity} for two reasons: it does not allow for general weight functions and the fixed $\GL(3)$ representation needs to be \textit{cuspidal}.

\item $\mathbf{4=4}$.
Let $\Pi$ be a fixed automorphic representation of $\GL(4)$, then we are interested in studying the first moment of $L(1/2,\Pi\otimes \sigma)$ as $\sigma$ varies. A reciprocity formula was found by Blomer, Li and Miller \cite{BLM2019spectral}. This case is the hardest one as it involves the least factorable $L$-function. In order to grasp the difficulty, it is worth mentioning that the authors of \cite{BLM2019spectral} do not have applications to subconvexity. Their only application is to non-vanishing and it is important to point out that their analysis is finer than usual as they only win by a logarithmic power, instead of the more usual polynomial saving.

\end{enumerate}

In this work, we give a generalization of the work in \cite{Nunes2020reciprocity} (hence of \cite{BlKh2019reciprocity,BlKh2019uniform}) to a spectral sum over $\GL(n)$ with $L$-functions of degree $2n^2$, such as in \cite{Jana2020RS} but in a more unbalanced situation.

\subsection{Main results}

Let $F$ be a number field and $S$ be a finite set of places of $F$ containing all the archimedean places. Let $\Pi$ and $\pi$ be cuspidal automorphic representations for $\GL(n+1)$ and $\GL(n-1)$ over $F$, respectively, with trivial central characters and unramified outside $S$. Let $\Phi\in \Pi$ and $\phi\in \pi$ be cusp forms. Also, let $\mathbf{s}=(s_1,s_2)$ be a pair of complex numbers and $\Lambda$ denote completed $L$-functions. Our main object of study is the following average of $L$-functions.
\begin{equation}\label{eq:main-moment}
    \mathcal{M}(\mathbf{s},\Phi,\phi) := \int_{\gen}\frac{\Lambda(s_1,\Pi\otimes\widetilde{\sigma})\Lambda(s_2,\sigma\otimes\widetilde{\pi})}{\mathcal{L}(\sigma)}H_S(\sigma;\Phi,\phi,\mathbf{s})\d \sigma,
\end{equation}
where the integral is taken with respect to an automorphic Plancherel measure over the \emph{generic} automorphic spectrum of $\GL(n)$ with trivial central character (\emph{cf}.\ \S \ref{sec:spectral-decomposition} for a discussion on the measure $\d \sigma$). The factor $H_S(\sigma):=H_S(\sigma;\Phi,\phi,\mathbf{s})$ is a certain weight function depending on the $S$-adic components of $\Phi,\phi$, and $\sigma$. We refer to \S\ref{prelim-local-factor} for the definition. The factor $\mathcal{L}(\sigma)$ is a certain harmonic weight which appears in \emph{e.g.} the Kuznetsov formula; see \eqref{harmonic-weights} for the definition. For example, if $\sigma$ is cuspidal then $\mathcal{L}(\sigma)$ is proportional to $L(1,\sigma,\Ad)$.

Finally, for any pair of complex numbers $\mathbf{s}=(s_1,s_2)$, we shall write
\begin{equation}\label{s-prime}
\wc{\mathbf{s}}=(\wc{s}_1,\wc{s}_2):=\left(\frac{1+(n-1)s_2-s_1}{n},\frac{(n+1)s_1+s_2-1}{n}\right).
\end{equation}
We also write $\Pi(w^\ast)\Phi$ as $\wc{\Phi}$ where
\begin{equation}\label{w-ast}
 w^\ast:=\begin{pmatrix}
\mathrm{I}_{n-2}&&\\&&1\\&1&
\end{pmatrix}
\end{equation}
which is a Weyl element in $\GL(n)$.
Our first main result is the following reciprocity formula.

\begin{thm}\label{reciprocity}
Let $\mathbf{s}\in \Cc^2$ such that
$$\frac12\leq \Re(s_1),\Re(s_2),\Re(\wc{s}_1),\Re(\wc{s_2})< 1.$$
Then, we have the equality
$$\mathcal{M}(\mathbf{s},\Phi,\phi)=\mathcal{N}(\mathbf{s},\Phi,\phi)+\mathcal{M}(\wc{\mathbf{s}},\widecheck{\Phi},\phi)$$
where
$$\mc{N}(\mathbf{s},\Phi,\phi)=\mc{R}(\wc{\mathbf{s}},\wc{\Phi},\phi)+\mc{D}(\wc{\mathbf{s}},\wc{\Phi},\phi)-\mc{R}(\mathbf{s},\Phi,\phi)-\mc{D}(\mathbf{s},\Phi,\phi).$$
and $\mc{R}(\mathbf{s},\Phi,\phi)$ and $\mc{D}(\mathbf{s},\Phi,\phi)$ are given by \eqref{def-R} and \eqref{degenerate-term-def}, respectively.
\end{thm}

As in \cite{Nunes2020reciprocity}, we choose the vectors $\Phi$ and $\phi$ in a such a way that the left-hand side picks up forms of conductors up to a certain height and the right-hand side can be asymptotically computed.

For an automorphic representation $\sigma$ we define $\sigma_{\mathrm{f}}$ to be the finite part $\otimes_{v<\infty}\sigma_v$. For a prime ideal $\p_0$ in $F$, we also define $\sigma^{(\p_0)}_{\mathrm{f}}:=\otimes_{\substack{v<\infty\\ v\neq \p_0}}\sigma_{v}$, the finite part of $\sigma$ outside $\p_0$ and correspondingly, we write
$$\mf{c}{(\sigma^{(\p_0)}_{\mathrm{f}})}:=\prod_{\p_0\neq\p<\infty}\p^{c(\sigma_{\p})},\quad C{(\sigma^{(\p_0)}_{\mathrm{f}})}:=\prod_{\p_0\neq\p<\infty}C(\sigma_{\p});$$
see \S\ref{newvectors} for the definitions. Note that these products converge absolutely as $c(\sigma_{\p})=0$ and $C(\sigma_{\p})=1$ for almost all $v$.

\begin{thm}\label{asymptotics}
Let $n\ge 3$ and $F$ be a totally real number field.
Let $\Pi$ and $\pi$ be cuspidal automorphic representations for $\GL_{n+1}(F)$ and $\GL_{n-1}(F)$, respectively,  with trivial central characters such that both $\Pi$ and $\pi$ are unramified and \emph{tempered} at every non-archimedean place. 
Further assume that $\p_0$ is a prime not dividing the discriminant of $F$ and $\tau$ is a supercuspidal representation of $\mathrm{GL}_n(F_{\mathfrak{p}_0})$ with trivial central character. 

Then for every $\delta>0$ there exists $\eta>0$ so that for any integral ideal $\q$ of the finite adeles such that $\q$, $\p_0$, and the discriminant of $F$ are pairwise coprime, we have
\begin{multline*}
\sum_{\substack{\sigma\text{ cuspidal; }\sigma_{\mathfrak{p}_0}=\tau;\\ \mf{c}(\sigma^{(\p_0)}_{\mathrm{f}})\mid\q,\,C(\sigma^{(\p_0)}_{\mathrm{f}})\ge N(\q)^{1/2-\delta}}}\frac{L(1/2,\Pi\otimes\widetilde{\sigma})L(1/2,\sigma\otimes\widetilde{\pi})}{L(1,\sigma,\operatorname{Ad})}H_{\q}(\sigma)h_{\infty}(\sigma) \\
= \frac{\mathcal{D}_\infty(F,n)}{\varepsilon_{\p_0}(1,\tau\otimes\widetilde{\tau}) }\frac{L^{\p_0}(1,\Pi\otimes\widetilde{\pi})L^{\p_0}(n/2,\widetilde{\Pi})}{L^{\p_0}(1+n/2,\widetilde{\pi})} + O_{\Pi,\pi}(N(\q)^{-\eta}).
\end{multline*}
Here $H_{\q}(\sigma)$ and $h_\infty(\sigma)$ are certain test functions at the $\q$-adic places and $\infty$-adic places defined in \eqref{non-arch-weight} and $\eqref{arch-weight}$, respectively, $\varepsilon_{\p_0}$ denotes the $\p_0$-adic epsilon factor. On the other hand, $\D_\infty(F,n)$ is a constant which depends on the archimedean component of the test vectors, in particular on $h_\infty$ and $n$ and the number field $F$. Moreover, we construct $h_\infty$ in a way so that $\D_\infty(F,n)\asymp 1$.
\end{thm}

The analogous theorem for $n=2$ is done in \cite{Nunes2020reciprocity}. Our proof also works for $n=2$, but a certain care is needed to compute the degenerate terms.
\\

In particular, we have the following corollary.

\begin{cor}\label{non-vaninshing}
Let $F$, $\q$, $\p_0$, $\tau$, $\Pi$, and $\pi$ be as above. Assume $N(\q)$ is sufficiently large. Then there exists at least one cuspidal representation $\sigma$ with trivial central character, such that $\sigma_{\p_0}=\tau$ and $C(\sigma^{(\p_0)}_{\mathrm{f}})\mid\q$ with $C(\sigma^{(\p_0)}_{\mathrm{f}})\ge N(\q)^{1/2-\delta}$ so that
both $L(1/2,\Pi\otimes\widetilde{\sigma})$ and $L(1/2,\sigma\otimes\widetilde{\pi})$ do not vanish.
\end{cor}

\begin{rmk}
    An important distinction between Corollary \ref{non-vaninshing} and \cite[Corollary 1.3]{Nunes2020reciprocity} is the need for an auxiliary prime at which the varying representations are supercuspidal. This difference is necessary for $n\geq 3$ as we cannot easily bound the contribution of the continuous spectrum, so we \emph{artificially} introduce this local condition to annihilate the continuous spectrum. The same kind of restriction also appears in \cite{Tsuzuki20212.07811}.

    On the other hand, the temperedness assumption in Corollary \ref{non-vaninshing} is rather of a technical nature. From the proof one will see that one also may allow $\Pi$ and $\pi$ to be $\theta$-tempered (see \S\ref{sec:local-prelim} for definition) for some very small $\theta$.
\end{rmk}

\vspace{0.1cm}

\subsection{Sketch of the proof: strong Gelfand formations}

Before we proceed to proving our results, we would like to present another point of view on spectral reciprocity formul{\ae}, which is due to Reznikov \cite{Reznikov2008RS}, which also serves as a high-level sketch of our proof. The key concept here is that of a \textit{Strong Gelfand Formation}. First, we say that a pair of reductive groups $(G,H)$ over a number field is a strong Gelfand pair if for every place $v$ and every pair of irreducible (admissible) representations $\Pi$ and $\sigma$  of $G(F_v)$ and $H(F_v)$, respectively, we have that the space of $H(F_v)$-invariant maps from $\Pi$ to $\sigma$ is at most one-dimensional. Now let $G$, $H_1$, $H_2$ and $J$ be reductive groups with natural embeddings as below.
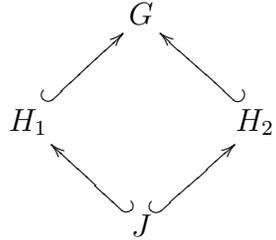
\begin{equation*}
\xymatrix{&G&\\ H_1\ar@{^{(}->}[ur]& &H_2
\ar@{_{(}->}[ul]\\
&J\ar@{^{(}->}[ur]\ar@{_{(}->}[ul]&}
\end{equation*}
Then we say that $(G,H_1,H_2,J)$ is a \textit{strong Gelfand formation} if the pairs $(G,H_i)$ and $(H_i,J)$ are strong Gelfand pairs, for $i=1,2$.

It is well-known that if $(G,H)$ is a Gelfand pair, $\Pi$ and $\sigma$ are automorphic representations of $G$ and $H$, respectively, and $\Phi\in \Pi$ and $\xi \in \sigma$ are automorphic forms, then the period
$$
\int_{[H]}\Phi(h)\ol{\xi}(h)\d h,
$$
where $[H]:=H(F)\bs H(\Aa)$, can often be linked to an $L$-function, \emph{e.g.} via Rankin--Selberg method or the Ichino--Ikeda formula. For the purpose of the sketch we assume that such quotients are compact. Therefore, in order to obtain a reciprocity formula one might consider $\Phi\in \Pi$ and $\phi \in \pi$ automorphic forms of $G$ and $J$, respectively, and consider the period $\int_{[J]}\Phi(j)\ol{\phi}(j)\d j$. Spectrally expanding the vector $\Phi$ in the spaces of automorphic forms in $H_1$ and $H_2$, one should get
$$\sum_{\xi_1}\left(\int_{[H_1]}\Phi\ol{\xi_1}\right)\left(\int_{[J]}\xi_1\ol{\phi}\right)
=
\sum_{\xi_2}\left(\int_{[H_2]}\Phi\ol{\xi_2}\right)\left(\int_{[J]}\xi_2\ol{\phi}\right),$$
where $\xi_i$ runs through an orthonormal basis of automorphic forms for $H_i$. 

The reciprocity formula which we prove here can be seen as a special case of the above discussion where take $G=\GL(n+1)$, $H_1=\GL(n)$, $J=\GL(n-1)$, and $H_2=w^\ast\GL(n)w^{\ast-1}$, with $w^\ast$ as in \eqref{w-ast}. For these groups, inclusion is given by $h\hookrightarrow\begin{pmatrix}h&\\&1\end{pmatrix}$
which is an embedding of $\GL(n)$ inside $\GL(n+1)$.

\subsection{A closely related work}

While this article was at the final stage of preparation, a preprint by Miao \cite{Miao2110.11529} has been posted on arXiv, obtaining a similar reciprocity formula to our Theorem \ref{reciprocity}. The author starts, as we do, from the identity in Proposition \ref{abstract-reciprocity}. The author also predicts an application of the reciprocity formula to non-vanishing. We decided nevertheless to keep our own proof of the reciprocity formula as we could not match our degenerate term to any term in \cite{Miao2110.11529} and that term is the source of the main term in Theorem \ref{asymptotics}.

\subsection{What's next?}

\subsubsection{Global}

\begin{enumerate}
    \item The question of generalizing the above results to $\Pi$ and $\pi$ that are general (not necessarily cuspidal) automorphic representations is a quite natural one. The difficulty comes from the fact that, in this case the period $\mc{P}(\mathbf{s},\Phi,\phi)$, defined in \eqref{P-period} and from whose properties we deduce the reciprocity in Theorem \ref{reciprocity}, does not converge. A remedy to this lack of convergence should follow from a suitable notion of regularization. The works of Ichino--Yamana \cite{IchinoYamana2016} and Zydor \cite{zydor2022periods} could be of help in finding out this suitable notion. It is important to remark that we do not require simply the conclusions from these papers. Instead we need to understand how the various truncation operators interact when averaged over the spectrum of $\GL(n)$.

    \item Finally, after staring at the list of different reciprocity formul{\ae} for $\GL(2)$ one might wonder whether we could have a general reciprocity formul{\ae} for the average of
    $$L(1/2,\Pi\otimes \widetilde{\sigma})L(1/2,\sigma\otimes \widetilde{\pi}),$$
    where, for given $n-1\ge r\ge 2$,  $\Pi$ and $\pi$ are automorphic representations of $\GL(n+r)$ and $\GL(n-r)$, respectively. Unfortunately, it is not clear to us how to generalize our method in any straightforward way. The main reason is the more complicated nature of Rankin--Selberg zeta integrals for $\GL(n)\times \GL(m)$ when $n-m>1$, which involves extra integration over unipotent groups. Similarly, one may wonder whether one may replace $\Pi\boxplus\pi$ by a $\GL(2n)$ cuspidal representation to obtain a reciprocity formula for $\GL(2n)\times\GL(n)$, thus generalizing the $n=2$ case from \cite{BLM2019spectral}.
\end{enumerate}

\subsubsection{Local}

\begin{enumerate}
\item Recall the local weight factors $H_S(\sigma)$ and corresponding weight factor $\wc{H}_S(\sigma):=H_S(\sigma;\wc{\Phi},\phi,\wc{\mathbf{s}})$ in the dual side of the reciprocity formula in Theorem \ref{reciprocity}. There is an implicit integral transform which relates $H_S$ and $\wc{ H}_S$ governed by the Weyl element in \eqref{w-ast}. For various analytic question in automorphic forms, one requires nice properties of these test functions \emph{e.g.} non-negativity and richness; see \cite{Mot1993explicit,Nelson1911.06310}.

\item For $n=2$ Blomer--Khan \cite{BlKh2019reciprocity} explicitly determined the integral transform relating the local weight functions in both sides of the reciprocity formula via classical methods \emph{e.g.} Kuznetsov and Voronoi summation formul\ae. It is natural to wonder whether one can deduce the same in our higher-rank case via the method of integral representations. The test functions are determined by the local Whittaker functions chosen as test vectors in \S \ref{prelim-local-factor}. We think one can produce the transform via the analysis of a certain local Bessel distribution, invoking, in particular, the local $\GL(n+1)\times\GL(n-1)$ functional equation which is governed by the Weyl element \eqref{w-ast}. However, we were unsuccessful in this endeavor. The reason, partly, seems to be that we are dealing with a period on $\GL(n+1)\times\GL(n-1)$ which does not resemble the $\GL(n+1)\times\GL(n-1)$ zeta integral.

\item On the other hand, via the theory of Kirillov models one will be able to show that the family of local weights is quite abundant, in the sense of \cite{Nelson1911.06310}.

\item More importantly, we would like to construct test vectors which produce a non-negative local weight. This seems to be difficult due to the \emph{unbalanced} (that is, correlation of $L$-functions of degrees $n(n+1)$ and $n(n-1)$) nature of the moment considered in Theorem \ref{reciprocity}), unlike the balanced moment (that is, correlation of $L$-functions of degree $2n$) as considered in \cite{Jana2020RS} where non-negativity of the local weight functions was immediate.

\item Finally, we may wonder if we can obtain a similar result as in Theorem \ref{asymptotics} in the archimedean aspect as well. One of the main ingredient to prove Theorem \ref{asymptotics} is the classical non-archimedean newvector theory of \cite{JPSS1981conducteur} which allows us to pick up the family considered in Theorem \ref{asymptotics}. In principle we can, at least at the real places, via the analogous archimedean newvector theory as in \cite{JaNe2019anv}. However, we do not pursue that in this article.
\end{enumerate}

\subsection{Structure of the paper}

We fix notations and conventions that are used throughout the text in \S \ref{sec:general-notations}. Preliminaries on automorphic forms, Whittaker models and integral representations of $L$-functions in the local and global settings are recalled in \S \ref{sec:local-prelim} and \S \ref{sec:global-prelim}, respectively.

In \S \ref{sec-global}, we use spectral theory and the theory of integral representations of $L$-functions to relate the average of $L$-functions $\mc{M}(\mathbf{s},\Phi,\phi)$ to a period of $\Phi$ and $\phi$ over $\GL(n-1)$ up to a degenerate term coming from regularizing the zeta integral on the smaller group. It becomes clear that the reciprocity formula follows from a certain identity of periods, also deduced in the same section.

Everything up until this point works for general vectors $\Phi$ and $\phi$. In \S \ref{sec:choice-of-vectors}, we describe our choices of local vectors used in the proof of Theorem \ref{asymptotics}. For these choices, we study in \S \ref{local-original} and \S\ref{local-dual} the local factors appearing on the original and dual sides, respectively. We use these vectors to pick up the family of representations considered in Theorem \ref{asymptotics}. This is one of the most technical parts of the paper.

In \S \ref{sec:residue-term}, we show a meromorphic continuation of the term $\mc{M}(\textbf{s},\Phi,\phi)$ to a neighborhood of $\mathbf{s}=\left(\frac12,\frac12\right)$. This adds an extra term $\mc{R}(\textbf{s},\Phi,\phi)$, called the residue term (since it appears after an application of the residue theorem), and we estimate its contribution when the vectors are as in \S\ref{sec:choice-of-vectors}. 

In \S \ref{sec:degenerate-term}, we show that the degenerate term can be factored into a product of local integrals and study these local factors, first the unramified computation and later with the choices from \S \ref{sec:choice-of-vectors}. This is the source of the main term in Theorem \ref{asymptotics}. The estimates to the residue and degenerate terms are only given for the term in the dual side since they vanish in the original side, which we show in \S \ref{sec:proofs}. This is the point where we make use of the auxiliary prime $\p_0$. The proofs of the main results are also given in \S\ref{sec:proofs}.

\section*{Acknowledgments}

We would like to thank Paul Nelson for various discussion and feedback. We would like to thank Farrell Brumley for providing us reference \cite{FLO2012representations} which became quite helpful in many parts of the paper. We would also like to thank Valentin Blomer for his interest in this work and encouragement. Part of the work has been carried out at the Max-Planck-Institut f\"ur Mathematik (MPIM) in Bonn. We are grateful towards the MPIM for its excellent working environment. Finally, we generously thank the anonymous referee for their careful reading of the manuscript and numerous helpful suggestions. Last but not the least, the first author thanks Shreyasi Datta for her constant support since the beginning of the project.

\section{General Notations}\label{sec:general-notations}

The letter $F$ denotes a local of global number field. In the beginning of each section we define $F$ explicitly. For any place $v$ of a global field $F$ we denote by $F_v$ the localization of $F$ at $v$. Similarly, for any global object $\mathfrak{G}$ (\emph{e.g.} an $L$-function) we denote the $v$-adic component (if defined) of $\mathfrak{G}$ by $\mathfrak{G}_v$ (\emph{e.g.} a $v$-adic $L$-factor). If clear from the context, we suppress the subscript $v$ from the notation $\mathfrak{G}_v$. We denote the adele ring of $F$ by $\A$.

For any $n\ge 1$ by $G_n$ we denote the algebraic group $\GL(n)$. We embed $G_n\hookrightarrow G_{n+1}$ in the upper-left corner. We denote the subgroup of upper triangular unipotent matrices in $G_n$ by $N_n$. Also, by $Z_n$ we denote the center of $G_n$. For any ring $R$, we have $Z_n(R)\cong R^\times \cong G_1(R)$, so often we identify $Z_n(R)$ with $G_1(R)$. We denote the long Weyl element of $G_n$ by $w_n$ which is given by the matrix whose anti-diagonal elements are $1$ and all other elements are $0$.

We fix Haar measures on $G_n(R)$, $N_n(R)$, and $G_1(R)$ which we denote by $\d g$, $\d n$, and $\d^\times z$, respectively. We also fix $G_n(R)$-invariant quotient measures on $Z_n(R)\bs{G}_n(R)$ and $N_n(R)\backslash G_n(R)$ which, abusing notations, we denote by $\d g$. Again, if clear from the context, we suppress the index $n$ from the notations.

Let $A_n$ be the group of diagonal matrices in $G_n$ which is isomorphic to $G_1(R)^n$. If $F$ is a local field then we denote the standard maximal compact subgroup of $G_n(F)$ by $K_n$. We have the Iwasawa decomposition $G_n(F)=N_n(F)A_n(F)K_n$. Using Iwasawa parametrization we write
$$\d g =\delta^{-1}(a) \d n\d^\times a \d k,$$
where $\delta$ is the modular character of the group $N_n(R)A_n(R)$ given by
$$\delta(a) = \prod_{j=1}^n |a_j|^{n-2j+1},\quad a=\mathrm{diag}(a_1,\dots,a_n)$$
and $\d^\times a= \prod_{i=1}^n \d^\times a_i$.
Also, here $\d k$ denotes the probability Haar measure on $K_n$.

We follow an $\epsilon$-convention, as usual in analytic number theory, which allows us to change the values of $\epsilon$ (which is typically very small) from line to line. We also adopt the usual Vinogradov notations $\ll$ and $\gg$. Moreover, we write $A \asymp B$ to mean $|B|\ll |A|\ll |B|$. Our convention is that the implied constants in the Vinogradov notations are allowed to depend on the global field and the ambient group.

\section{Local Preliminaries}\label{sec:local-prelim}

In this section we work over a general local field $F$ of characteristic zero, archimedean or non-archimedean, without mentioning the field explicitly. If $F$ is non-archimedean we let $v$ be its valuation, $\o$ its ring of integers, $\p$ the maximal ideal in $\o$, and $N(\p)$ the order of its residue field.

For the group $G_n$ we mostly write $G$, unless there is a source of confusion. We adopt he same convention for subgroups of $G$.

The letter $\pi$ will denote an irreducible admissible representation of $G$.

\subsection{Measure normalizations}
Let $|\cdot|_F$ denote an absolute value on $F$. In particular, $|\cdot|_\C=|\cdot|^2_\R$. When there is no confusion we will drop the subscript $F$. On $F$ we fix a translation invariant measure $\d x$ so that if $F$ is non-archimedean then $\vol(\o,\d x)=1$. We fix $\d^\times x:=\zeta_F(1)\frac{\mathrm{d}x}{|x|}$ to be the Haar measure on $F^\times$ where $\zeta_F$ is the zeta function attached to $F$.

\subsection{Additive character}\label{sec:additive-character}

We fix an additive character $\psi_0$ of $F$. If $F$ is non-archimedean then we assume that $\psi_0$ is unramified, \emph{i.e.} trivial on $\o$.  We define an additive character of $N$ by
\begin{equation*}
    \psi(n(x)):=\psi_0\left(\sum_{i=1}^{n-1}x_{i,i+1}\right), \quad n(x):=(x_{i,j})_{i,j}\in N.
\end{equation*}
We denote restriction of $\psi$ to smaller unipotent subgroups also by the same letter.

\subsection{Gamma factors and analytic conductors}\label{sec:gamma-conductor}

For every irreducible representation $\pi$ we attach local $\gamma$, $L$, $\varepsilon$ factors which are related by
$$\gamma(s,\pi):=\varepsilon(s,\pi)\frac{L(1-s,\widetilde{\pi})}{L(s,\pi)},$$
where $\widetilde{\pi}$ is the contragredient of $\pi$. We refer to \cite[\S3]{Cogdell2007functions} for the description of the local factors.

We also attach a local \emph{analytic conductor} $C(\pi)$ to $\pi$. If $F$ is archimedean one defines $C(\pi)$ via the Langlands parameters of $\pi$; see \cite[eq. (31)]{IS2020perspective}. If $F$ is non-archimedean then one defines $C(\pi)$ via the invariance property under the Hecke congruence subgroups, as in \cite{JPSS1981conducteur}; see \S\ref{newvectors}. We record that if $F$ is non-archimedean and $\pi$ is unramified then $C(\pi)=1$.

If $\pi_i$ are representations of $G_{n_i}$ for $i=1,2$. We also attach an analytic conductor $C(\pi_1\otimes\pi_2)$ of the Rankin--Selberg product $\pi_1\otimes\pi_2$. Then one has
\begin{equation}\label{conductor-inequality}
    C(\pi_1\otimes\pi_2)\ll C(\pi_1)^{n_2}C(\pi_2)^{n_1}
\end{equation}
where the implied constant is absolute; see \cite[Appendix A]{HB2019ZeroFree}. 

One can analytically relate the local $\gamma$ factor and the analytic conductor of a representation (see \cite{JaNe2019anv} for some discussion on this). We have the asymptotic expansion 
$$\gamma(1/2+s,\pi) = \gamma(1/2,\pi)C(\pi)^{-s} + O_\pi(s),$$
as $s\to 0$. More precisely,
\begin{equation}\label{gamma-factor-bound}
    \gamma(1/2+s,\pi)\asymp C(\pi\otimes|\det|^{\Im(s)})^{-\Re(s)},
\end{equation}
as long as $s$ is away from the poles and zeros of the $\gamma$ factor.

\subsection{Sobolev norms}
We follow \cite[\S2.3.3]{MV2010subconvexity} to define a Sobolev norm on unitary representations. First, we define a Laplacian $\mathfrak{D}$ on $C^\infty(G)$. 

If $F$ is archimedean then we define
$$\mathfrak{D}:= 1 -\sum_{X}X^2,$$
where $\{X\}$ is an orthonormal basis of the Lie algebra of $G$ with respect to the standard Killing form.

Let $F$ be non-archimedean. Let $K[m]$ denote the principal congruence subgroup of level $m$ and $e[m]$ denote the orthogonal projector on the orthogonal complement of $K[m-1]$-invariant vectors inside the space of $K[m]$ invariant vectors. We define the Laplacian on $G$ by
$$\mathfrak{D} : =\sum_{m=0}^\infty N(\p)^m e[m].$$
Note that $\sum_{m=0}^\infty e[m]$ is the identity operator. We also note that $\mathfrak{D}$ is invertible and a large enough power of $\mathfrak{D}^{-1}$ is of trace class.

Finally, we define the \emph{order $d$ Sobolev norm} of $v\in\pi$ by
$$S_d(v):=\|\mathfrak{D}^d v\|_\pi.$$
We refer to \cite[\S2.4]{MV2010subconvexity} for useful properties of the Sobolev norm.

\subsection{Whittaker and Kirillov models}\label{whittaker-kirillov-model}

For the details of this subsection we refer to \cite{BZ1976rep} for non-archimedean case and \cite[\S 3]{Jacquet2010distinction}
for archimedean case.

We recall the notion of {genericity} for an irreducible representation $\pi$ of $G$. We call $\pi$ to be \emph{generic} if
$$\mathrm{Hom}_G\left(\pi,\mathrm{Ind}_N^G \psi)\right)\neq \{0\},$$
where 
$$\mathrm{Ind}_N^G \psi:=\{W\in C^\infty(G)\text{ with moderate growth}\mid W(ng)=\psi(n)W(g), n\in N, g\in G\}.$$
We also know that if $\pi$ is generic then the above $\mathrm{Hom}$-space is one-dimensional. We always identify $\pi$ with its image under a non-zero element of the $\mathrm{Hom}$-space, which we call the \emph{Whittaker model} of $\pi$ under $\psi$. 

The theory of \emph{Kirillov models} asserts that the restriction
$$W\mapsto \left\{ g\mapsto W\left[\begin{pmatrix}g&\\&1\end{pmatrix}\right]\right\}$$
is injective. Furthermore, for any $\phi\in C^\infty_c(N_{n-1}\backslash G_{n-1},\psi)$ there is a unique $W_\phi\in \pi$ such that
$$W_\phi\left[\begin{pmatrix}g&\\&1\end{pmatrix}\right]=\phi(g),$$
and the map $\phi\mapsto W_\phi$ is continuous.

If $\pi$ is generic and unitary then we put a unitary inner-product on $\pi$ by setting
$$\langle W_1,W_2\rangle_0 := \int_{N_{n-1}\backslash G_{n-1}}W_1\left[\begin{pmatrix}g&\\&1\end{pmatrix}\right]\overline{W_2\left[\begin{pmatrix}g&\\&1\end{pmatrix}\right]}dg,$$
for any two $W_1,W_2\in \pi$. 

We work with a specific normalization of the inner product in this paper. We fix
\begin{equation}\label{inner-product-normalization}
     \langle W_1,W_2\rangle:=
\begin{cases}
    \frac{\zeta_F(n)}{L(1,\pi\otimes\widetilde{\pi})}\langle W_1,W_2\rangle_0 &\text{ if $F$ is non-archimedean}\\
    \langle W_1,W_2\rangle_0 &\text{ if $F$ is archimedean}
\end{cases}
\end{equation}
In the non-archimedean case, the inner-product is built so that if $\pi$ is unramified and $W\in \pi$ is spherical, then $\|W\|^2=|W(1)|^2$. Note that such normalization is purely for cosmetic purposes so that the harmonic weights in \eqref{eq:main-moment} remain independent of the set of ramified places $S$; see Lemma \ref{harmonic-weight-computation}. On the other hand, in the archimedean case such normalization is not necessary as we are concerned about the ``finite part'' of the $L$-functions, as in Theorem \ref{asymptotics}, as opposed to completed $L$-functions.
\\

We recall the Langlands classification of the unitary representations of $G$. Let $P$ be the standard parabolic subgroup of $G$ attached to the partition $n=\sum_{i=1}^k n_i$. Let $\pi_i$ be any \emph{essentially square-integrable} representation of $G_{n_i}$. For any \(k\)-tuple $(s_1,\dots,s_k)\in\C^k$ we consider the unitarily normalized induction $\mathrm{Ind}_P^G\bigotimes_{i=1}^k\pi_i\otimes|\det|^{s_i}$. Then any unitary representation $\pi$ of $G$ is the unique irreducible constituent of such an induction and is denoted by $\bigboxplus_{i=1}^k\pi_i\otimes|\det|^{s_i}$. We define $\pi$ to be \emph{$\theta$-tempered} (\emph{resp}.\ \emph{tempered}) if $\max_{i=1}^k|\Re(s_i)|\le\theta$ (\emph{resp}.\ $0$).

In this paper, we will always assume that $\theta<1/2$ whenever it appears. Note that it is known that the local components of a unitary automorphic representation are always $\theta$-tempered for some $\theta<1/2$.\\

We need the following bound for Whittaker functions. Although the bound may very well be available in the literature (in a scattered way), we were unable to find a proper reference to the result in the $\theta$-tempered case that works for general local fields. We state the result now and prove it later after developing the necessary tools.

\begin{lemma}\label{local-whittaker-bound}
Let $W\in \pi$ be $\theta$-tempered. If $g=ak$ with $a=\mathrm{diag}(a_1,\dots,a_n)$ then for any large $N>0$ and small $\eta>0$
$$W(g)\ll_{N,\eta} |\det(a/a_n)|^{-\theta}\delta^{1/2-\eta}(a)\prod_{i=1}^{n-1}\min(1,|a_i/a_{i+1}|^{-N})S_d(W),$$
for some $d>0$ depending only on $N$ and the group.
\end{lemma}

This result is proved for a real place in \cite[Lemma 5.2]{JaNe2019anv} in the tempered case and in \cite[Lemma 7.2]{Jana2020RS} in the $\theta$-tempered case. 

\subsection{Whittaker--Plancherel formula}
We record the relevant formulation of the Whittaker--Plancherel formula. We refer to \cite[Chapter 5]{Wallach1992realreductive} for the archimedean case and \cite[Theorem 2.3.2]{beuzartplessis2008.05036} for the non-archimedean case. 

Let $\widehat{G}$ be the unitary dual of $G$ \emph{i.e.} $\widehat{G}$ is the set of all isomorphism classes of unitary irreducible representations of $G$. we equip $\widehat{G}$ with a local Plancherel measure $\d\mu^\loc$ which is compatible with the Haar measure $\d g$ on $G$ in the sense of Harish-Chandra:
\begin{equation}\label{harish-chandra-plancherel}
f(1)=\int_{\widehat{G}}\mathrm{Trace}(\pi(f))\d\mu^\loc(\pi),\quad f\in C_c^\infty(G).
\end{equation}
It is known that $\d\mu^\loc$ is supported only on the irreducible generic tempered representations of $G$. Let $\xi\in C^\infty(N\backslash G,\psi)$ be an element in the Harish-Chandra Schwartz space (adapted to Whittaker models), in the sense of \cite[\S 4]{Wallach1992realreductive}. Then we have the following absolutely convergent spectral decomposition:
\begin{equation}\label{whittaker-plancherel-theorem}
     \xi(g)=\int_{\widehat{G}}\sum_{W\in\B(\pi)}W(g)\int_{N\backslash G}\xi(h)\overline{W(h)}\d h\d\mu^\loc(\pi),
\end{equation}
where $\B(\pi)$ is an orthonormal basis of $\pi$. The above sum does not depend on the choice of the orthonormal basis.

We record a useful lemma.

\begin{lemma}\label{trace-class-property}
For each $d_1,d_2>0$ there is an $L>0$ such that 
$$\int_{\widehat{G}}C(\pi)^{d_1}\sum_{v\in\B(\pi)}S_{d_2}(v)S_{-L}(v)\d\mu^\loc(\pi)<\infty.$$
Here $\B(\pi)$ is an orthonormal basis of $\pi$ consisting of eigenvectors of $\mathfrak{D}$.
\end{lemma}

\begin{proof}
If $F$ is archimedean the proof can be done as in \cite[Lemma 3.3]{JaNe2019anv}, verbatim.

Let $F$ be non-archimedean. We denote by $c(\pi)$ the conductor exponent of $\pi$ (see \S\ref{newvectors} for the definition). Using \cite[2.6.3 Lemma]{MV2010subconvexity}
we see that the integral in the lemma is bounded by
$$\int_{\widehat{G}}N(\p)^{d_1c(\pi)}\sum_{m=c(\pi)}^\infty N(\p)^{m(d_2-L)}N(\p)^{md_3}\d\mu^\loc(\pi),$$
for some $d_3>0$. For any $L'>0$ there is a large $L>0$ such that the inner sum is $O(N(\p)^{-L'c(\pi)})$. Thus it is enough to prove that
$$\int_{\widehat{G}}N(\p)^{-Ac(\pi)}\d\mu^\loc(\pi)=\sum_{\ell=0}^\infty N(\p)^{-\ell A}\mu^\loc\left(\{\pi\in\widehat{G}\mid c(\pi)=\ell\}\right)<\infty,$$
for large enough $A>0$. Thus it suffices to show that there is a fixed $B\ge 0$ so that
$$\mu^\loc\left(\{\pi\in\widehat{G}\mid c(\pi)\le \ell\}\right)\ll N(\p)^{\ell B}.$$
Let $f:=(\vol(K_0(\mathfrak{p}^\ell)))^{-1}\mathbf{1}_{K_0(\mathfrak{p}^\ell)}=f\ast f$ where $\ast$ denotes convolution. Note that for $c(\pi)\le \ell$, from newvector theory (see \S\ref{newvectors}), we obtain that
$$\mathrm{trace}(\pi(f)) =\sum_{v\in\B(\pi)}\|\pi(f)v\|^2\ge \|\pi(f)v_0\|^2=\|v_0\|^2 = 1$$
where $v_0\in\pi$ is a unit newvector. Now applying Harish-Chandra Plancherel formula, as in \eqref{harish-chandra-plancherel}, we obtain
$$N(\p)^{\ell(n-1)}\gg (\vol(K_0(\mathfrak{p}^\ell)))^{-1}\ge \mu^\loc\{\pi\in\widehat{G}\mid c(\pi)\le \ell\},$$
as required.
\end{proof}

\subsection{Local zeta integral and functional equation}
For this subsection we refer to \cite[\S3]{Cogdell2007functions} for a detailed discussion.

Let $\Pi$ and $\pi$ be irreducible generic representations of $G_{n+1}$ and $G_n$, respectively, realized in the Whittaker models with respect to the same additive character. For $\Re(s)$ sufficiently large, we define the local zeta integral of $V\in \Pi$ and $W\in \pi$ by
\begin{equation}\label{local-zeta-integral-def}
    \Psi(s,V,\ol{W}):=\int_{N_n\backslash G_n}V\left[\begin{pmatrix}g&\\&1\end{pmatrix}\right]\ol{W(g)}|\det(g)|^{s-1/2}\d g.
\end{equation}
If $\Pi$ and $\pi$ are unitary then the above integral converges for $\Re(s)>1$. One can then meromorphically continue $\Psi$ to the whole complex plane.

Let $\omega_\pi$ be the central character of $\pi$. We have the local functional equation
\begin{equation}\label{local-functional-equation}
    \Psi(1-s,\widetilde{V},\ol{\widetilde{W}})=\omega_\pi(-1)^n\gamma(s,\Pi\otimes\bar{\pi})\Psi(s,V,\ol{W}),
\end{equation}
where $\widetilde{W}\in\widetilde{\pi}$ denotes the contragredient of $W$ defined by $\widetilde{W}(g):=W(w_ng^{-t})$ and similarly for $\widetilde{V}$; see \cite[Theorem 3.2]{Cogdell2007functions}.

\begin{lemma}\label{trivial-bound-zeta-integral}
Let $V\in \Pi$ and $W\in\pi$ such that both $\Pi$ and $\pi$ are varying over some families of representations. Also let $s\in \C$ be a regular point of the zeta integral $\Psi(.,V,\ol{W})$. Then for each $d>0$ there is a $d'>0$ such that
$$\Psi(1/2+s,V,\ol{W})\ll_{s} S_{d'}(V)S_{-d}(W),$$
where the dependency of the implicit constant on $s$ is at most polynomial.
\end{lemma}

\begin{proof}
We first show the standard fact that any Whittaker function decays rapidly along any positive root: for any $N:=(N_1,\dots,N_{n-1})$ and $a=\diag(a_1,\dots,a_n)$ then 
\begin{equation}\label{rapid-decay-arch}
    W(a)\ll_N\prod_{i=1}^{n-1}\min(1,|a_i/a_{i+1}|^{-N_i})S_{d}(W),
\end{equation}
for some $d>0$ depending only on $N$. The proof uses unipotent equivariance and smoothness of the Whittaker function. 

If $F$ is archimedean we may choose an element $Y$ (depending on $N_1,\dots,N_{n-1}$) in the Lie algebra of $G_n$ such that
$$\d\pi(Y)W(a) = \prod_{i=1}^{n-1}(a_i/a_{i+1})^{N_i}W(a).$$
The claim follows from the Sobolev inequality $\d\pi(Y)W(a) \ll S_d(W)$, for some $d>0$.

If $F$ is non-archimedean then the invariance of $W$ under some open-compact subgroup of $G_n$ implies that
\begin{equation}\label{rapid-decay-nonarch}
    W(a) = 0, \quad \text{if $|a_i| > c |a_{i+1}|$ for some $i$},
\end{equation}
where $c$ depends only on the level of the open-compact subgroup under which $W$ is invariant, \emph{i.e.} a certain Sobolev norm of $W$. Then the claim follows again from the Sobolev inequality.

Now let $\Re(s)$ be large enough so that we can write $\Psi(1/2+s,V,W)$ as the absolutely convergent integral
$$\int_{N_n\backslash G_n}V\left[\begin{pmatrix}g&\\&1\end{pmatrix}\right]\ol{W(g)}|\det(g)|^s\d g.$$
Let $F$ be archimedean. We integrate by parts with respect to $\mathfrak{D}$ sufficiently many, say $d$, times to obtain that the above equals
$$\int_{N_n\backslash G_n}\mathfrak{D}^d\left(V\left[\begin{pmatrix}g&\\&1\end{pmatrix}\right]|\det(g)|^s\right)\mathfrak{D}^{-d}\ol{W(g)}\d g.$$
It is straightforward to check that $\mathfrak{D}|\det(g)|^s= p(s)|\det(g)|^s$ for a certain polynomial $p(s)$. Let $\omega_\pi$ be the central character of $\pi$.
We use Iwasawa coordinates in $N_n\backslash G_n$ to write the above integral as
\begin{multline*}
    \sum_{j\le d} p_{j}(s)\int_{A_{n-1}\times K_{n}}\int_{F^\times}\mathfrak{D}^jV\left[\begin{pmatrix}z\begin{pmatrix}a&\\&1\end{pmatrix}k&\\&1\end{pmatrix}\right]|\det(a)|^s|z|^{ns}\omega_\pi(z)\\
    \ol{\mathfrak{D}^{-d}W\left[\begin{pmatrix}a&\\&1\end{pmatrix}k\right]}\d^\times z \frac{\d^\times a}{|\det(a)|\delta(a)}\d k,
\end{multline*}
for some polynomials $p_j(s)$, $j\ge 1$.

We make the change of variables $k\mapsto\begin{pmatrix}k'&\\&1\end{pmatrix}k$ and integrate over $k'\in K_{n-1}$ to obtain that the above is, up to an absolute constant, equal to a finite sum of terms of the form
\begin{multline*}
    \sum_{j\le d}p_{j}(s)\int_{K_{n}}\int_{N_{n-1}\backslash G_{n-1}}\int_{F^\times}\mathfrak{D}^jV\left[\begin{pmatrix}z\begin{pmatrix}h&\\&1\end{pmatrix}k&\\&1\end{pmatrix}\right]|\det(h)|^{s-1}|z|^{ns}\omega_\pi(z)\\
    \ol{\mathfrak{D}^{-d}W\left[\begin{pmatrix}h&\\&1\end{pmatrix}k\right]}\d^\times z \d h\d k.
\end{multline*}
We apply the Cauchy--Schwarz inequality on the $h$-integral and use the unitarity of $\pi$ to obtain that the absolute value of the above expression is bounded by
\begin{multline*}
S_{-d}(W)\int_{K_n}\left(\int_{N_{n-1}\backslash G_{n-1}}\left|\int_{F^\times}\mathfrak{D}^eV\left[\begin{pmatrix}z\begin{pmatrix}h&\\&1\end{pmatrix}k&\\&1\end{pmatrix}\right]|\det(h)|^{\Re(s)-1}\right.\right.\\
\left.\left.|z|^{n\Re(s)}\omega_\pi(z)\d^\times z\right|^2 \d h\right)^{1/2}\d k.
\end{multline*}
Using rapid decay estimate of $V$ from \eqref{rapid-decay-arch} we see that the above integral is absolutely convergent for $\Re(s)$ large enough, and it is bounded by $S_{d'}(V)$ for a certain $d'$.

Now let $F$ be non-archimedean and $\Re(s)$ be sufficiently large. First, we assume that $V$ and $W$ are $K$-type vectors, \emph{i.e.}, $\mathfrak{D}$-eigenvectors. Then if the level of $V$ is smaller than that of $W$, the zeta integral vanishes and the assertion follows.

Now let the level of $V$ be larger than that of $W$. We write the zeta integral in the Iwasawa coordinates as
$$\int_{K_{n}}\int_{A_n}\Pi\begin{pmatrix}k&\\&1\end{pmatrix}V\left[\begin{pmatrix}a&\\&1\end{pmatrix}\right]\pi(k)W(a)|\det(a)|^s\frac{\d^\times a}{\delta(a)}\d k.$$
Using rapid decay of $V$ as in \eqref{rapid-decay-nonarch} we may restrict the inner integral to $a_1\ll\dots\ll a_n$ where the the implied constants depend only on the level of $V$ and polynomially so. We use the Sobolev inequality to bound $\pi(k)W(a)\ll_{K_n} S_{d}(W)$ for some $d>0$. So the zeta integral becomes absolutely convergent for some large $\Re(s)$ and bounded by 
$$S_{d'}(V)S_{d}(W)\ll S_{d''}(V)S_{-d}(W),$$
for some $d'$ and $d''$ depending on $d$.
Thus the claim follows for $V,W$ being $\mathfrak{D}$-eigenvectors. The general claim now follows from \cite[\S2.4.4, S4d]{MV2010subconvexity}.

So far we have proved that for a general local field $F$ and $\Re(s)$ sufficiently positive the assertion in the lemma follows. Now if $\Re(s)$ is sufficiently negative then we use the local functional equation \eqref{local-functional-equation}, and the bounds in \eqref{gamma-factor-bound} and \eqref{conductor-inequality} to conclude that $\Psi(1/2+s,V,W)$ satisfies the claim in the lemma  (by absorbing the powers of the conductor into the Sobolev norms). We conclude our proof by an application of the Phragm\'en--Lindel\"of convexity principle.
\end{proof}

\begin{proof}[Proof of Lemma \ref{local-whittaker-bound}]
We follow the proofs in \cite[Lemma 7.2]{Jana2020RS} and \cite[Lemma 5.2]{JaNe2019anv}. Using \cite[\S2.4.1, S1b]{MV2010subconvexity} we reduce to the case of $k=1$.

We argue by induction on $n$. The $n=2$ case is in \cite[Proposition 3.2.3]{MV2010subconvexity}. We now prove the inductive step. 

Note that in the archimedean case there exists a differential operator $Y$ such that
$$d\pi(Y)W(a) = (a_{n-1}/a_n)W(a).$$
We define $W_1:=d\pi(Y^N)W$. Thus it is enough to show that
$$W_1(a)\ll_{N,\eta} |\det(a)|^{-\theta}\delta^{1/2-\eta}(a)\prod_{i=1}^{n-2}\min(1,|a_i/a_{i+1}|^{-N})S_d(W_1),$$
as $S_d(W_1)\ll S_{d'}(W)$ for some $d'>d$.

In the non-archimedean case, we reduce to showing the above by appealing to the invariance of $W$ under sum open-compact subgroup and unipotent equivariance.

Let $\omega_\pi$ be the central character of $\pi$ and $\widetilde{a}:=\diag(a_1,\dots,a_{n-1})$. Using the Whittaker--Plancherel formula \eqref{whittaker-plancherel-theorem} we write
\begin{multline*}
    |\det(a/a_n)|^sW_1(a)=\omega_\pi(a_n)W_1(a/a_n)|\det(\widetilde{a}/a_n)|^{s}=
    \\\omega_{\pi}(a_n)\int_{\widehat{G_{n-1}}}\sum_{W'\in\B(\pi')}W'(\widetilde{a}/a_n)\Psi(1/2+s, W_1,\overline{W'})\d\mu^\loc(\pi'),
\end{multline*}
which is valid for sufficiently large $\Re(s)$. 

The right-hand side is absolutely convergent, which can be seen by applying Lemma \ref{trivial-bound-zeta-integral} and Lemma \ref{trace-class-property}. Thus it is analytic in $s$ in some right half-plane. Note that the poles of the integrand in the right-hand side as a function of $s$ may at most come from the poles of the zeta integral as the Whittaker functions are analytic \cite{Jacquet04IntRep}. It is known that $\Psi(1/2+s,W_1,\overline{W'})$ is a holomorphic multiple of the local $L$-factor $L(1/2+s,\pi\otimes\overline{\pi'})$; see \emph{e.g.} \cite[Theorem 3.5]{Cogdell2007functions}. As $\pi$ is $\theta$-tempered and $\pi'$ is tempered, the $L$-factor is holomorphic for $\Re(s)>-1/2+\theta$. So we may analytically continue the integrand of the right-hand side until $\Re(s)=-1/2+\theta+\eta$ for any $\eta>0$.

We apply the inductive hypothesis on $W'$ (note that $\pi'$ is tempered), thus deducing that
$$W'(\widetilde{a}/a_n)\ll_{N,\eta} \delta^{1/2-\eta}(\widetilde{a})\prod_{i=1}^{n-2}\min(1,|a_i/a_{i+1}|^{-N})S_{d'}(W'),$$
for some $d'>0$. We use 
$$\delta(a)=|\det(\widetilde{a}/a_n)|\delta(\widetilde{a}/a_n)$$
and Lemma \ref{trivial-bound-zeta-integral}
to obtain that
\begin{multline*}
    W_1(a)\ll_{N,\eta} |\det(a/a_n)|^{-\theta}\delta^{1/2-\eta}(a)\prod_{i=1}^{n-2}\min(1,|a_i/a_{i+1}|^{-N})
    \\\int_{\widehat{G_{n-1}}}\sum_{W'\in\B(\pi')}S_{d'-L}(W')S_{d}(W_1)\d\mu^\loc(\pi'),
\end{multline*}
where $d$ depends on $L,N$. Taking $L$ large enough and applying Lemma \ref{trace-class-property} we see the that last integral is convergent and we conclude.
\end{proof}

\begin{rmk}
From the proof of Lemma \ref{local-whittaker-bound} it can be noted that the exponent $\theta$ of $|\det|$ in the bound of $W$, as in the statement of the lemma, can be modified to $-\widetilde{\theta}$ where $\widetilde{\theta}$ is the \emph{minimum} of the magnitudes of the real parts of the Langlands parameters of $\pi$.
\end{rmk}

\subsection{Newvectors}\label{newvectors}

Let $F$ be non-archimedean and $\pi$ be generic with trivial central character. Let $K_0(\p^j)$ be the Hecke-congruence subgroup of $G_n$, \emph{i.e.} consist of matrices in $G_n(\o)$ whose last rows modulo $\p^j$ are congruent to $(0,\dots,0,*)$. 

Let $c(\pi)$ be the minimal non-negative integer $j$ such that the $K_0(\p^j)$-fixed subspace $\pi^{K_0(\p^j)}$ is non-zero. It is a theorem by Casselman \cite{Casselman1973some} (for $n=2$) and Jacquet--Piatetski-Shapiro--Shalika \cite{JPSS1981conducteur} (for general $n$) that $\pi^{K_0(\p^{c(\pi)})}$ is one dimensional (also, see \cite{matringe2013essential,jacquet2012correction} where an error in \cite{JPSS1981conducteur} has been corrected).
Any non-zero vector in this fixed space is called a \emph{newvector}. Also, $c(\pi)$ and $C(\pi):=N(\p)^{c(\pi)}$ are called the \emph{conductor exponent} and \emph{(analytic) conductor} of $\pi$, respectively.

In this paper we denote the newvector $W\in \pi$ such that $W(1)=1$ by $W_\pi$. Newvectors often serve the purpose of test vectors for the Rankin--Selberg periods. In this paper we use two such instances, hence record them here.

Let $\sigma$ and $\Pi$ be any irreducible generic representations of $G_n$ and $G_{n+1}$, respectively. Further assume that at least one of $\Pi$ and $\sigma$ is unramified. We consider the vector $W_{\Pi}^{(c(\sigma))}$, given by \eqref{choice-test-vector-original}. From \cite[Theorem 1.1]{BKL2020test}, we have
\begin{equation}\label{bkl-test-vector}
    \Psi(s, W^{(c(\sigma))}_\Pi,\ol{W_\sigma}) = \frac{L(s,\Pi\otimes\bar{\sigma})}{[G_n(\mf{o}):K_0(\p^{c(\sigma)})]}.
\end{equation}
Note that, this generalizes classical test vector result in \cite{JPSS1981conducteur} which considers the case of unramified $\sigma$.

The $L$-functions attached to $\pi$ can be given by
$$L(s,\pi) = \prod_{i=1}^n(1-N(\p)^{-s}\alpha_i)^{-1},$$
for a certain $\alpha:=\{\alpha_i\}\in\C^n$. If $\pi$ is also unramified then $\alpha_i\neq 0$ and are called the \emph{Satake parameters} attached to $\pi$. 

The description of $W_\pi$ restricted to $A_n(F)$, which is due to Shintani \cite{Shintani1976explicit} for unramified $\pi$ and Miyauchi \cite{Miyauchi2014Whittaker} for general $\pi$, is as follows:

Let $m: =(m_1,\dots,m_n)\in \Z^n$ and
$$
a=\diag(y_1,y_2,\ldots, y_n),
$$
with $v(y_i)=m_i$. Then
\begin{equation}\label{shintani}
    W_\pi(a) = 
    \begin{cases}
    \delta^{1/2}(a)\lambda_\pi(m), &\text{ if } m_1\ge\dots\ge m_n\\
    0, &\text{ otherwise},
    \end{cases}
\end{equation}
where $\lambda_\pi(m)$ is the Schur polynomial with index $m$ and evaluated at $\alpha$, \emph{i.e.}
$$\lambda_\pi(m) : =\frac{\det([\alpha_j^{m_i + n-i}]_{1\le i,j\le n})}{\det([\alpha_j^{n-i}]_{1\le i,j\le n})}.$$
If $\pi$ is $\theta$-tempered then $\max_i\{|\alpha_i|\} \le N(\p)^{\theta}$. Consequently, it follows from the highest weight theory for $\mathrm{U}(n)$ that
\begin{equation}\label{spherical-whittaker-bound}
    \lambda_\pi(m) \ll N(\p)^{\theta\sum_i m_i},
\end{equation}
where the implied constant is at most a polynomial in $m$.

We also record that if $\pi$ is unitary and unramified then
\begin{equation}\label{L2-whittaker}
    \|W_\pi\|^2 = 1,
\end{equation}
which follows by directly calculating the $L^2$-norm using the description in \eqref{shintani}. On the other hand, if $\pi$ is ramified, then from the description of $W_\pi$ as in \cite[Theorem 4.1]{Miyauchi2014Whittaker} one computes
$$\langle W_\pi,W_\pi\rangle_0 = \prod_{i,j=1}^{n}(1-\beta_i\overline{\beta_j}N(\p)^{-1})^{-1}$$
for some $\beta_i\in\C$.
If $\pi$ is $\theta$-tempered for some $0\le \theta<1/2$ (\emph{e.g.}, appears as a local component of a generic standard automorphic representation) then there exists an absolute $\delta>0$ such that $|\beta_i|\le N(\p)^{1/2-\delta}$; see, \emph{e.g.}, \cite[eq.(2)]{Brumley2006effective}. Moreover, $\theta$-temperedness of $\pi$ ensures the existence of an absolute $\delta'>0$ so that
$$L(1,\pi\otimes\tilde{\pi})=\prod_{i=1}^{n^2}(1-\beta'_j N(\p)^{-1})^{-1}$$
for some $\beta'_i$ satisfying $|\beta'_i|\le N(\p)^{1-\delta'}$; see \emph{e.g.}, \cite[eq.(4)]{Brumley2006effective}. Thus we obtain
\begin{equation}\label{L2-ramified-whittaker}
    \|W_\pi\|^2 = \frac{\zeta_F(n)}{L(1,\pi\otimes\widetilde{\pi})}\langle W_\pi,W_\pi\rangle_0 \asymp 1
\end{equation}
where the implied constants only depend on $n,\delta,\delta'$ and in particular, \emph{not} on $p,\pi$.

\section{Global Preliminaries}\label{sec:global-prelim}

In this section, we let $F$ be a number field and $\A$ its ring of adeles. By $G_n$ we denote the algebraic group $\GL(n)$ over $F$. For any subgroup $H<G_n$ defined over $F$, we denote by $[H]$ the quotient $H(F)\backslash H(\A)$. We also define $[\widetilde{H}]$ to be the quotient $Z_H(\A)H(F)\backslash H(\A)$ where $Z_H$ is the center of $H$. We may give a $G_n(\A)$-invariant \emph{finite} measure on $[\widetilde{G_n}]$ which is compatible with the product measure on $G_n(\Aa)$ and we denote it by $\d g$. When there is no confusion we might suppress the index $n$ from the notation.

\subsection{Classification of automorphic spectrum}\label{subsec:global-classification}

We give a quick description of the \emph{standard automorphic representations} \emph{i.e.} those which appear in the spectral decomposition of $L^2([\widetilde{G}])$. We refer to \cite[\S 2.2.1]{MV2010subconvexity} and \cite{MW} for details.

Let $Y(G)$ be the set of pairs $(M ,\sigma)$ where $M$ is the Levi part of a standard parabolic subgroup $G$ and $\sigma$ is an isomorphism class of \emph{discrete series} of $M(\Aa)$. 
Here by discrete series we mean the automorphic forms on $[M]$ such that for all $\varphi\in \sigma$ the integral
$$\|\varphi\|^2_\sigma:=\int_{[\widetilde{M}]}|\varphi(x)|^2\d x$$ is finite. 

We let $X(G)$ be the quotient of $Y(G)$ by the equivalence relation defined as follows: $(M,\sigma) \sim (M',\sigma')$ if there exists a Weyl element $w$ such that $wMw^{-1} = M'$ and $w\sigma \cong \sigma'$. For every $\chi\in X(G)$ we define $\mathcal{I}(\chi)$ to be the unitarily normalized induction $\mathrm{Ind}_{N_M(\A)M(\A)}^{G(\A)}\sigma$ where $N_M$ is the unipotent radical attached to $M$. Langlands' classification asserts that any standard automorphic representation is isomorphic to the unique irreducible constituent $\widetilde{\chi}$ of the induction $\mathcal{I}(\chi)$.

Let us now start with a \emph{cuspidal data} $\chi=(M,\sigma)$, \emph{i.e.} $\sigma$ being a cuspidal automorphic representation of $M(\A)$, and proceed with the same construction as above to obtain $\widetilde{\chi}$. Another theorem of Langlands asserts that any \emph{generic} (see below) automorphic representation is isomorphic to such a $\widetilde{\chi}$.

Finally, we denote by $X(\widetilde{G})$ the subset of isomorphism classes of $\chi$ in $X(G)$ so that $\mathcal{I}(\chi)$ is $Z_G$-invariant.

\subsection{Spectral decomposition}\label{sec:spectral-decomposition}

We define a norm on $\mathcal{I}(\chi)$ by $$\|f\|^2_{\mathcal{I}(\chi)}:= \int_K \|f(k)\|^2_\sigma \d k,$$
where $K:=\prod_{v}K_v$ where $K_v$ is the standard maximal compact of $G_n(F_v)$.
Finally, we define an intertwiner (by averaging over $P(F)\backslash G(F)$ and analytic continuation)
$$\Eis:\mathcal{I}(\chi)\to C^\infty([G]).$$
Then for any element $\xi\in C^\infty([\widetilde{G}])$ with sufficient decay at the cusp (\emph{e.g.} a cusp form) we have the pointwise Plancherel decomposition
\begin{equation}\label{spectral-decomposition}
    \xi(x)=\int_{X(\widetilde{G})}\sum_{f\in\widetilde{\B}(\mathcal{I}(\chi))}\frac{\langle \xi, \Eis(f)\rangle_{L^2([\widetilde{G}])}}{\langle f,f\rangle_{\mathcal{I}(\chi)}} \Eis(f)(x) \d\mu^\aut(\chi).
\end{equation}
Here $\widetilde{\B}(\mathcal{I}(\chi))$ denotes an orthogonal basis of $\mathcal{I}(\chi)$ and $d\mu^\aut(\chi)$ denotes the automorphic Plancherel measure on $X(\widetilde{G})$ compatible with $\d g$. The right-hand side above does not depend on the choice of orthogonal basis $\widetilde{\B}$. Also, the right-hand side converges absolutely and uniformly on compacta. For more details we refer to \cite[\S 2.2.1]{MV2010subconvexity}.

Often, we use the shorthand for \eqref{spectral-decomposition}, writing
\begin{equation*}
    \xi(x) = \int_\aut\sum_{\varphi\in\widetilde{\B}(\pi)}\frac{\langle \xi, \varphi\rangle}{\|\varphi\|^2} \varphi(x) \d\pi.
\end{equation*}
In practice, we mostly vary $\pi$ over the generic spectrum only, in which case we replace $\int_\aut$ above by $\int_\gen$.

\subsection{Fourier expansion of automorphic forms}

Let $\psi_0:F\backslash\A\to \C^\times$ be an additive character. For concreteness, we chose $\psi_0$, as in \cite{MV2010subconvexity}, to be the additive character $e_{\Qq}\circ \on{tr}$, where $e_{\Qq}$ is the only additive character of $\Qq\bs \Aa_{\Qq}$ whose restriction to $\Rr$ is $x\mapsto \exp(2\pi ix)$ and $\on{tr}: \Aa\rightarrow \Aa_{\Qq}$ is the adelic extension of $\on{tr}:F\rightarrow \Qq$. 
We extend $\psi_0$ to a character $\psi$ of $N(\A)$ as in \S\ref{sec:additive-character}. We define the $\psi$-Whittaker space by
$$\mathcal{W}(\psi):=\{W\in C^\infty(G(\A))\text{ with moderate growth}\mid W(ng) = \psi(n)W(g), n\in N(\A), g\in G(\A)\}$$
on which $G(\A)$ acts by right translation.

For any automorphic representation $\pi$ we define an intertwiner $\pi\to \mathcal{W}(\psi)$ by
$$\pi\ni\varphi\mapsto W_\varphi:=\int_{[N]}\varphi(n.)\overline{\psi(n)}\d n.$$
We call $\pi$ to be \emph{generic} if the above intertwiner does not vanish identically. The theory of Whittaker model asserts that if $\pi$ is irreducible and generic then the above intertwiner is unique up to scalars and, in fact, defines a $G(\A)$-equivariant embedding. We call the image $\mathcal{W}(\pi,\psi)$ of $\pi$ under the above intertwiner the Whittaker model of $\pi$. For generic $\pi$ we identify $\pi$ with its Whittaker model.

Given an automorphic form $\varphi$ in a generic representation $\pi$ of $G_n(\A)$ we can write its Fourier expansion using $W_\varphi$. For example, if $\varphi$ is cuspidal then we write (see \cite[Theorem 1.1]{Cogdell2007functions})
\begin{equation}\label{Whitt-decomp}
    \varphi(g)=\sum_{\gamma\in N_n(F)\backslash P_n(F)} W_\varphi(\gamma g)=\sum_{\gamma'\in N_{n-1}(F)\bs G_{n-1}(F)}W_{\varphi}\left[\begin{pmatrix}
    \gamma'&\\&1
    \end{pmatrix}g\right].
\end{equation}
The above Fourier expansions converge absolutely and uniformly on compacta.
Here $P_n$ is the standard Mirabolic subgroup of $\GL(n)$ defined by $\GL({n-1})\rtimes U_n$ and $U_n$ is the unipotent radical of the parabolic in $\GL(n)$ attached to the partition $n=(n-1)+1$. In other words, $P_n$ is the stabilizer of $(0,\dots,0,1)$ of the right action of $G_n$ on the row vectors; thus consists of matrices in $\GL(n)$ with last row being $(0,\dots,0,1)$.

If $\varphi$ is non-cuspidal then the Fourier expansion of $\varphi$ is more complicated. We do not need full Fourier expansion for non-cuspidal automorphic form; interested readers may look at \cite[Proposition 4.2]{IchinoYamana2016}. However, we do need a partial Fourier expansion with respect to the unipotent subgroup $U_n$. From abelian Fourier theory we have
\begin{equation}\label{Un-foureier-expansion}
    (\varphi-\varphi_{U_n})(g) = \sum_{\gamma\in P_{n-1}(F)\bs G_{n-1}(F)} W^{U_n}_\varphi\left[\begin{pmatrix}
    \gamma&\\&1
    \end{pmatrix}g\right].
\end{equation}
Here $\varphi_{U_n}$ is the the constant term of $\varphi$ along $U_n$ defined by
$$\varphi_{U_n}(g):=\int_{[U_n]} \varphi(ug) \d u,$$
and $W^{U_n}_\varphi$ is a partial Whittaker function defined by
\begin{equation}\label{partial-whittaker-function}
    W^{U_n}_\varphi(g):=\int_{[U_n]} \varphi(ug)\overline{\psi(u)} \d u = \int_{[U_n]} (\varphi(ug)-\varphi_{U_n}(ug))\overline{\psi(u)} \d u,
\end{equation}
which follows as $\varphi_{U_n}$ is left $U_n$-invariant.

\subsection{Global Zeta Integral and $L$-functions}\label{global-zeta-integral}

We give a quick description of the global theory of $\GL(n+1)\times\GL(n)$ zeta integrals; for details see \cite[\S2]{Cogdell2007functions}. Let $\Pi$ and $\pi$ be generic representations of $G_{n+1}(\A)$ and $G_n(\A)$, respectively. Let $\Phi\in\Pi$ and $\varphi\in\pi$ be two automorphic forms with Whittaker functions $W_\Phi$ and $W_\varphi$, respectively. We define the global Hecke zeta integral of $\Phi$ and $\phi$ by
$$\Psi(s,W_\Phi,\overline{W_\varphi}):=\int_{N_n(\A)\bs G_n(\A)}W_\Phi\left[\begin{pmatrix}g&\\&1\end{pmatrix}\right]\overline{W_\varphi(g)}|\det(g)|^{s-1/2}\d g.$$
The above converges absolutely for sufficiently large $\Re(s)$. If $\Pi$ is cuspidal then the above integral is also equal to the absolutely convergent integral
$$\int_{[G_n]}\Phi\left[\begin{pmatrix}g&\\&1\end{pmatrix}\right]\overline{\varphi(g)}|\det(g)|^{s-1/2}\d g,$$
which can be seen after inserting the Fourier expansion of $\Phi$ and $\phi$, and unfolding.

If $\Phi$ and $\varphi$ are factorable vectors then the global zeta integral factors into local zeta integrals as
$$\Psi(s,W_\Phi,\ol{W_\varphi}) = \prod_{v}\Psi_v(s, W_{\Phi,v},\ol{W_{\varphi,v}}),$$
where the local zeta integral $\Psi_v$ is defined as in \eqref{local-zeta-integral-def}, and $W_\Phi=\prod_v W_{\Phi,v}$, similarly for $W_\varphi$. Once again the product converges absolutely for sufficiently large $\Re(s)$.

If $\Phi$, $\varphi$ and $\psi$ are unramified at the places outside of a finite set $S$, and further, if $W_{\Phi,v}=W_{\Pi_v}$ and $W_{\varphi,v}=W_{\pi_v}$ for $v\notin S$ then (see \cite[Theorem 3.3]{Cogdell2007functions})
\begin{equation}\label{euler-product}
\Psi(s,W_\Phi,\overline{W_\varphi}) =\Lambda(s,\Pi\otimes\ol{\pi})\prod_{v\in S}\frac{\Psi_v(s, W_{\Phi,v},\ol{W_{\varphi,v}})}{L_v(s,\Pi_v\otimes\ol{\pi_v})},
\end{equation}
where $L_v$ denotes the $v$-adic local $L$-factor and $\Lambda$ denotes the global \emph{completed} Rankin--Selberg $L$-function of $\Pi\otimes\bar{\pi}$. We refer to \cite[Theorem 4.2]{Cogdell2007functions} for meromorphic properties of the Rankin--Selberg $L$-functions.

We write $\Lambda^S$ for the partial $L$-function removing all $v$-adic Euler factors for $v\in S$. If $S=\{v\mid\infty\}$ then we write, as usual in analytic number theory, $L$ for $\Lambda^S$.

We attach a \emph{global analytic conductor} $C(\pi)$ to an automorphic representation $\pi$, analogously to its $L$-function $\Lambda(s,\pi):=\Lambda(s,\pi\otimes\mathbf{1})$. If $\pi$ is unramified at the places outside of a finite set $S$, then $C(\pi)=\prod_{v\in S}C(\pi_v)$ where $C(\pi_v)$ is the local conductor as defined in \S\ref{sec:gamma-conductor}.

We also have the \emph{convexity bound}
\begin{equation}\label{convexity-bound}
    L(1/2,\pi)\ll_\epsilon C(\pi)^{1/4+\epsilon},
\end{equation}
which follows from the functional equation of $\Lambda(s,\pi)$ and the Phragm\'en--Lindel\"of convexity principle.

If $\Phi$ is cuspidal then the integral 
$$\int_{[G_n]}\Phi\left[\begin{pmatrix}g&\\&1\end{pmatrix}\right]\overline{\varphi(g)}|\det(g)|^{s-1/2}\d g$$
converges absolutely for any $\varphi$ and any $s\in \C$. Moreover, if $\varphi$ is generic then the above equals $\Psi(s,W_\Phi,\overline{W_\varphi})$. This way we may analytically continue $\Psi(s, W_{\Phi},\overline{W_\varphi})$ for any cuspidal $\Phi$.

However, if $\Phi$ is not cuspidal then the above integral is not convergent. One needs to regularize the integral, as in \cite{IchinoYamana2016}, to give it a meaning and then the regularized integral will again be equal to the global zeta integral.

In this paper we do not need the general regularization scheme but only need to regularize the above integral when $\varphi$ is cuspidal. This ``naive'' regularization is comparatively easier and we describe it in \S \ref{subsec:regularization} below.

\subsection{Harmonic weights}\label{sec:harmonic-weights}

We describe the harmonic weights that appear in the Kuznetsov trace formula. The harmonic weights relate the unitary inner products on a generic unitary automorphic representation and that of its Whittaker model.

Let $\pi$ be a generic irreducible unitary automorphic representation such that $\pi$ is unramified outside a finite set of places $S$ which also contains the archimedean places. Let $\varphi_1,\varphi_2\in\pi$ so that $W_{\varphi_1,v}=W_{\varphi_2,v}=W_{\pi_v}$ for all $v\notin S$. By Schur's lemma there exists a positive constant $\mc{L}^S(\pi)$ such that for any two $\varphi_1,\varphi_2\in\pi$,
\begin{equation}\label{harmonic-weights}
\langle \varphi_1,\varphi_2\rangle_\pi = \mc{L}^S(\pi) \prod_{v\in S} \langle W_{\varphi_1,v},W_{\varphi_2, v}\rangle_{\pi_v}
\end{equation}
where $\langle,\rangle_{\pi_v}$ is as in \eqref{inner-product-normalization}. If $\pi$ is cuspidal then a standard Rankin--Selberg argument shows that $\mc{L}^S(\pi)$ is independent of $S$ and
\begin{equation}\label{harmonic-weight-cuspidal}
    \mc{L}^S(\pi)=\mc{L}(\pi)=c_{F,G}L(1,\pi,\Ad),
\end{equation}
where $c_{F,G}$ is a positive constant depending on the number field $F$ and the group $G$.
If $\pi$ is non-cuspidal then we have that $\mc{L}^S(\pi)$ is again independent of $S$ and proportional to the first non-zero Laurent coefficient of $L(s,\pi\otimes\widetilde{\pi})$ around $s=1$.

\begin{lemma}\label{harmonic-weight-computation}
Let $M$ be a Levi of $G$ attached to the partition $n=n_1+\dots+n_k$ and $\pi:=\otimes_{i=1}^k\pi_i$ be a cusp form on $M$ such that $\pi_i$'s are pairwise non-isomorphic. Let $\Pi:=\boxplus_{i=1}^k\pi_i$ be the Eisenstein representation attached to the cuspidal data $(M,\pi)$. Then 
$$\mathcal{L}^{S}(\Pi)=\mc{L}(\Pi)=d_{F,M}\lim_{s\to 1}\frac{L(s,\Pi\otimes\widetilde{\Pi})}{\zeta_F(s)^k},$$
where $\zeta_F(s)$ is the Dedekind zeta function attached to $F$ and $d_{F,M}$ is a positive constant depending only on $F$ and $M$.
\end{lemma}

If $\pi_i$ are not pairwise non-isomorphic then $\mathcal{L}(\Pi)$ will be proportional to $\lim_{s\to 1}(s-1)^{k'}L(s,\Pi\otimes\widetilde{\Pi})$ where $k'$ is the order of the pole of $L(s,\Pi\otimes\widetilde{\Pi})$ at $s=1$. However, we do not need that result here, so we do not prove it.

For $\pi$ as in Lemma \ref{harmonic-weight-computation} we define the global Casselman--Shalika factor 
$$\mathfrak{cs}(\pi)=\prod_{1\le i<j\le k}L(1,\pi_i\otimes\widetilde{\pi}_j).$$
Note that as $\pi_i$ are pairwise non-isomorphic the above quantity is well-defined.
Similarly, we define the partial factor $\mathfrak{cs}^S(\pi)$ and the local factor $\mathfrak{cs}(\pi_v)$. Note that, we can also write assertion of Lemma \ref{harmonic-weight-computation} as
$$\mathcal{L}(\Pi)=d'_{F,M}|\mathfrak{cs}(\pi)|^2\prod_{i=1}^kL(1,\pi_i,\mathrm{Ad}),$$
for some positive constant $d'_{F,M}$. 

This result is probably known to the experts. However, we were unable to find a reference which points us to the constant with the precision we need (\emph{e.g.} to define \eqref{def-R}).

\begin{proof}
Without loss of generality we let $S$ contain the ramified places of $F$ including the archimedean places as well.

Let $f\in\mathcal{I}(M,\pi)$ be any nonzero element and $\Eis(f)$ be the corresponding Eisenstein series. Then from \cite[Proposition 7.3.1]{Shahidi2010Eisenstein} we have
$$W_{\Eis(f)}=\prod_{v}W_{\Eis(f),v},
\quad
W_{\Eis(f),v}=W^{\mathrm{Jac}}_{f_v},$$
where $f_v\in\mathrm{Ind}_{M(F_v)}^{G(F_v)}(\pi_v)$ and $\pi_v$ is realized in its Whittaker model. 
Here, $W^{\mathrm{Jac}}_{f_v}$ is Jacquet's functional as defined in \cite[\S1.4]{FLO2012representations}. 

If $\phi\in\pi$ we assume that $W_{\phi,v}=W_{\pi_v}$ when $v$ is an unramified place. Then it follows from \cite[Proposition 7.1.4]{Shahidi2010Eisenstein} that $$W_{f_v}^{\mathrm{Jac}}(1)=\mathfrak{cs}(\pi_v)^{-1},$$
for $v\notin S$.
Hence, using \eqref{L2-whittaker} we have
\begin{equation}\label{l2-whittaker-eisenstein}
    \|W^{\mathrm{Jac}}_{f_v}\|^2_{\Pi_v}=\frac{1}{|\mathfrak{cs}(\pi_v)|^2}.
\end{equation}
for all $v\notin S$.

Moreover, we have a factorization \cite[\S11]{FLO2012representations}
$$W_{\Eis(\mathfrak{cs}^S(\pi)f)}=\mathfrak{cs}^S(\pi)W_{\Eis(f)}=\prod_{v\notin S}\mathfrak{cs}(\pi_v)W^{\mathrm{Jac}}_{f_v}\prod_{v\in S}W^{\mathrm{Jac}}_{f_v}.$$
We choose $\varphi_1=\varphi_2=\Eis(\mathfrak{cs}^S(\pi)f)$ in \eqref{harmonic-weights}.

Using the definition of harmonic weights in \eqref{harmonic-weights} we have
$$\|f(g)\|^2_\pi=\prod_{i=1}^k\mathcal{L}(\pi_i)\prod_{v\in S}\|f_v(g_v)\|_{\pi_v}^2.$$
From \cite[Proposition A.2]{FLO2012representations} and \eqref{inner-product-normalization} we have for non-archimedean $v$
$$L_v(1,\Pi_v\otimes\widetilde{\Pi}_v)\|W^{\mathrm{Jac}}_{f_v}\|^2_{\Pi_v} = \prod_iL_v(1,\pi_{i,v}\otimes\widetilde{\pi}_{i,v})\int_{P(F_v)\bs G(F_v)}\|f_v(g_v)\|^2_{\pi_v}\d g_v.$$
In other words, for non-archimedean $v$
$$\int_{P(F_v)\bs G(F_v)}\|f_v(g_v)\|^2_{\pi_v}\d g_v=|\mathfrak{cs}(\pi_v)|^2\|W^{\mathrm{Jac}}_{f_v}\|^2_{\Pi_v}.$$
Thus using \eqref{l2-whittaker-eisenstein} we obtain that the right-hand side is $1$ for all $v\notin S$.

On the other hand, for archimedean $v$ similarly
using \cite[Proposition A.2]{FLO2012representations} and \eqref{inner-product-normalization} we get
$$\int_{P(F_v)\bs G(F_v)}\|f_v(g_v)\|^2_{\pi_v}\d g_v=\frac{\zeta_v(n)}{\prod_i\zeta_v(n_i)}\|W^{\mathrm{Jac}}_{f_v}\|^2_{\Pi_v}.$$
Thus we have
$$\|f\|^2_{\mathcal{I}(M,\pi)}=\int_{P(\A)\bs G(\A)}\|f(g)\|^2_\pi\d g\\=\prod_{i=1}^k\mathcal{L}(\pi_i)\prod_{\substack{v\in S\\v\text{ non-archimedean}}}{\|W^{\mathrm{Jac}}_{f_v}\|^2_{\Pi_v}|\mathfrak{cs}(\pi_v)|^2}\prod_{\substack{v\in S\\v\text{ archimedean}}}{\|W^{\mathrm{Jac}}_{f_v}\|^2_{\Pi_v}}.$$
Replacing $f$ by $\mathfrak{cs}^S(\pi)f$ we obtain
\begin{equation*}
    \|\mathfrak{cs}^S(\pi)f\|^2_{\mathcal{I}(M,\pi)}=\prod_{i=1}^k\mathcal{L}(\pi_i)|\mathfrak{cs}(\pi)|^2\prod_{v\in S}{\|W^{\mathrm{Jac}}_{f_v}\|^2_{\Pi_v}}.
\end{equation*}
Finally, using description of $\mathcal{L}(\pi_i)$ for cuspidal $\pi_i$ as in \eqref{harmonic-weight-cuspidal} we conclude.
\end{proof}

\subsection{Regularization}\label{subsec:regularization}

To prove the reciprocity formula in Theorem \ref{reciprocity} we need to express the global $L$-function of $\sigma\otimes\pi$ as a period integral on $\GL(n)\times\GL(n-1)$ while $\sigma$ is not necessarily cuspidal. In this case, the usual $\GL(n)\times\GL(n-1)$ zeta integral may not be absolutely convergent. This is why we need to regularize the global zeta integral to relate it with the Rankin--Selberg $L$-functions.

\begin{prop}\label{simple-truncation-cuspidal}
Let $\varphi$ be an generic automorphic form on $G_n(\A)$ and $\phi$ be a cusp form on $G_{n-1}(\A)$. Then for $s\in\C$ with $\Re(s)$ sufficiently large
$$\int_{[G_{n-1}]}\left(\varphi\left[\begin{pmatrix}
g&\\
&1
\end{pmatrix}\right]
-
\varphi_{U_{n}}\left[\begin{pmatrix}
g&\\
&1
\end{pmatrix}\right]
\ol{\phi(g)}\right)|\det(g)|^{s-1/2}\d g$$
is absolutely convergent and equals $\Psi(s,W_\varphi,\ol{W_\phi})$.
\end{prop}

To prove this proposition we need some preparatory results on the decay properties of $W^{U_n}_\varphi$ which will be used to deal with several convergence issues.

\begin{lemma}\label{approx-by-const-term}
Let $\varphi$ be any vector in a generic automorphic representation $\sigma$ of $G_n(\A)$ and $\varphi_{U_n}$ be its constant term along $U_n$. Let $X$ be any element in the universal enveloping algebra of $\prod_{v\mid\infty} G_n(F_v)$. Then for all large $N$, and for $z\in [Z_{n-1}]$ and $g\in G_{n}(\A)$ we have
$$\d\sigma(X)\left(\varphi-
\varphi_{U_{n}}\right)\left[\begin{pmatrix}
z&\\
&1
\end{pmatrix}g\right]
\ll_{X,N,g}|z|^{-N},$$
where the dependency on $g$ is at most polynomial in the coordinates of the toric part of $g$ according to an Iwasawa decomposition.
\end{lemma}

\begin{proof}
This is a special case/reformulation of \cite[Lemma I.2.10]{MW}.
\end{proof}

\begin{lemma}\label{bound-partial-whittaker}
Let $z\in [Z_{n-1}]$, $g\in G_{n-1}(\A)$ and $\gamma\in P_{n-1}(F)\bs G_{n-1}(F)$. Then for all large $M,N$
$$W^{U_n}_\varphi \left[\begin{pmatrix}
z\gamma g&\\
&1
\end{pmatrix}\right] \ll_{M,N,g} (1+|z|)^{-N} |\gamma|_F^{-M}$$
where $|\cdot|_F$ denotes any fixed norm of $F^{n-1}$ and the dependency on $g$ is at most polynomial as in the previous lemma.
\end{lemma}

\begin{proof}
We use the formulation in \eqref{partial-whittaker-function} to write
$$W^{U_n}_\varphi \left[\begin{pmatrix}
z\gamma g&\\
&1
\end{pmatrix}\right] = \int_{F^{n-1}\bs\A^{n-1}} (\varphi-\varphi_{U_n})\left[\begin{pmatrix}
\rm{I}_{n-1}&x\\
&1
\end{pmatrix}\begin{pmatrix}
z\gamma g&\\
&1
\end{pmatrix}\right]\overline{\psi_0(e_{n-1}x)}\d x,$$
where $e_{n-1}$ is the vector $(0,\dots,0,1)\in \A^{n-1}$.
Conjugating, using automorphicity of $\varphi$, and changing variables we write the above as
$$|\det(z)|\int_{F^{n-1}\bs\A^{n-1}} (\varphi-\varphi_{U_n})\left[\begin{pmatrix}
z&\\
&1
\end{pmatrix}\begin{pmatrix}
g&x\\
&1
\end{pmatrix}\right]\overline{\psi_0(e_{n-1}\gamma zx)}\d x.$$
From abelian Fourier theory and Lemma \ref{approx-by-const-term} we deduce that the above integral is bounded by
$$(1+|z|)^{-N}|e_{n-1}\gamma|_F^{-M}$$
for all large $M,N$. We conclude the proof by noting that $e_{n-1}\gamma$ uniquely determines $\gamma\in P_{n-1}(F)\bs G_{n-1}(F)$.
\end{proof}

\begin{proof}[Proof of Proposition \ref{simple-truncation-cuspidal}]
Rapid decay of the cusp form $\phi$ on $[\widetilde{G}_{n-1}]$ and Lemma \ref{approx-by-const-term} yield the absolute convergence in the Lemma for sufficiently large $\Re(s)$.

To prove the equality between the integral in the proposition and $\Psi(s,W_\varphi,\overline{W_\phi})$ we first, using \eqref{Un-foureier-expansion}, write the integral as
$$\int_{[G_{n-1}]}\sum_{\gamma\in P_{n-1}(F)\bs G_{n-1}(F)} W^{U_n}_\varphi\left[\begin{pmatrix}
    \gamma g&\\&1
    \end{pmatrix}\right]\overline{\phi(g)}|\det(g)|^{s-1/2}\d g.$$
Again rapid decay of $\phi$ on $[\widetilde{G}_{n-1}]$ and Lemma \ref{bound-partial-whittaker} yield absolute convergence of the above joint integral and sum for large $\Re(s)$. This allows us to unfold and write the above as
$$\int_{P_{n-1}(F)\bs G_{n-1}(\A)}W^{U_n}_\varphi\left[\begin{pmatrix}
    g&\\&1
    \end{pmatrix}\right]\overline{\phi(g)}|\det(g)|^{s-1/2}\d g.$$
Now we insert the Fourier expansion of $\phi$ as in \eqref{Whitt-decomp} into the above equation to write the same as
$$\int_{P_{n-1}(F)\bs G_{n-1}(\A)}W^{U_n}_\varphi\left[\begin{pmatrix}
    g&\\&1
    \end{pmatrix}\right]\sum_{\gamma\in N_{n-1}(F)\backslash P_{n-1}(F)} \overline{W_\phi(\gamma g)}|\det(g)|^{s-1/2} \d g.$$
Once again, applying Lemma \ref{bound-partial-whittaker} for $W_\varphi^{U_n}$ along with Lemma \ref{local-whittaker-bound} for $W_\phi$ we deduce that the above joint sum and integral converge absolutely. This allows us to unfold once again to write the above as
$$\int_{N_{n-1}(F)\bs G_{n-1}(\A)}W^{U_n}_\varphi\left[\begin{pmatrix}
    g&\\&1
    \end{pmatrix}\right]\overline{W_\phi(g)}|\det(g)|^{s-1/2} \d g.$$
Finally, using $N_{n-1}(\A)$-equivariance of $W_\phi$ we fold the above integral as
$$\int_{N_{n-1}(\A)\bs G_{n-1}(\A)}\left( \int_{[N_{n-1}]}W^{U_n}_\varphi\left[\begin{pmatrix}
    ng&\\&1
    \end{pmatrix}\right]\overline{\psi(n)} \d n\right) \overline{W_\phi(g)}|\det(g)|^{s-1/2}\d g.$$
We conclude by noting that the inner integral evaluates to $W_\varphi\left[\begin{pmatrix}
    g&\\&1
    \end{pmatrix}\right]$. 
\end{proof}

\section{Global Set-up}\label{sec-global}

In this section we work globally and adopt the notations as described in the beginning of \S\ref{sec:global-prelim}. Let $\Pi$ and $\pi$ be cuspidal automorphic representations of $G_{n+1}(\A)$ and $G_{n-1}(\A)$, respectively, with trivial central characters. Let $\Phi\in\Pi$ and $\phi\in\pi$ be two cusp forms.

\subsection{Reciprocity as a period identity}

We first show an identity between two periods of the automorphic forms $\Phi$ and $\phi$. This identity is the point of departure for the reciprocity formula.

For $\mathbf{s}\in \C^2$ we define
\begin{equation}\label{P-period}
\P(\mathbf{s},\Phi,\phi):=\int_{[G_{n-1}]}\int_{[G_1]}\Phi\left[\begin{pmatrix} zh&&\\&z&\\&&1\end{pmatrix}
\right]
\ol{\phi(h)}
|\det(h)|^{s_1+s_2-1}|z|^{n(s_1-1/2)}\d^{\times}z\d h.
\end{equation}
As $\Phi$ is cuspidal, rapid decay of $\Phi$ ensures that the above integral converges absolutely for any $\mathbf{s}\in\C^2$.

\begin{prop}\label{abstract-reciprocity} 
For any $\mathbf{s}\in\C^2$ we have
$$\mathcal{P}(\mathbf{s},\Phi,\phi)=\mathcal{P}(\wc{\mathbf{s}},\widecheck{\Phi},\phi),$$
for $\widecheck{\Phi}:=\Pi(w^\ast)\Phi$
where $w^\ast$ and $\wc{\mathbf{s}}$ are given by \eqref{w-ast} and \eqref{s-prime}, respectively.
\end{prop}

\begin{proof}
This is a straightforward generalization of \cite[Proposition 5.1]{Nunes2020reciprocity} and the proof follows the same lines. Indeed, by making use of the automorphicity of $\Phi$ we have
$$\Phi\left[\begin{pmatrix}hz&&\\&z&\\&&1\end{pmatrix}
\right]=\widecheck{\Phi}\left[z\begin{pmatrix}
hz&&\\
&1&\\
&&1
\end{pmatrix}
\begin{pmatrix}
z^{-1}\rm{I}_n&\\
&1
\end{pmatrix}
\right]$$
We insert the above in the integral in \eqref{P-period}. Then using the central invariance of $\Phi$ and $\phi$, along with the change of variables $h\mapsto hz^{-1}$ followed by $z\mapsto z^{-1}$, we arrive at the conclusion.
\end{proof}

\subsection{Spectral decomposition of the period}\label{sec:spectral-decomposition-period}

Recall that $\Phi$ is a cusp form on $G_{n+1}(\A)$. We start by projecting $\Phi$ to the space of center invariant automorphic forms on $G_n(\A)$. For $s\in\C$, we define
$$\mc{A}_{s}\Phi(g)=|\det(g)|^{s-1/2}\int_{[G_1]}\Phi\left[\begin{pmatrix}zg&\\&1\end{pmatrix}\right]|z|^{n(s-1/2)}\d^{\times} u,\quad g\in G_n(\A).$$
Once again the above converges absolutely due to the rapid decay of $\Phi$. 

Note that $\mc{A}_{s}\Phi(g)$ is $Z_n(\A)G_n(F)$-left invariant.
We also notice that since $\Phi$ is smooth and of rapid decay, then so is $\mc{A}_{s}\Phi$ on $Z_n(\A)\bs{G}_n(\A)$. We spectrally decompose $\mc{A}_{s}\Phi$ over the standard automorphic representations of $Z_n\bs G_n$, as in \eqref{spectral-decomposition}, to obtain
$$\mc{A}_{s}\Phi(g)=\int_{\aut}\sum_{\varphi\in \widetilde{\mc{B}}(\sigma)}\frac{\IP{\mc{A}_{s}\Phi,\varphi}}{\|\varphi\|^2}\varphi(g)\d \sigma.$$
It follows directly from the definition of $\mc{A}_{s}$ that we have
$$\IP{\mc{A}_{s}\Phi,\varphi}=\int_{[G_n]}\Phi\left[\begin{pmatrix}
g&\\
&1
\end{pmatrix}\right]\overline{\varphi(g)}|\det(g)|^{s-1/2}\d g,$$
which, due to rapid decay of $\Phi$, converges absolutely for all $s\in \C$.
Using Fourier expansion of $\Phi$ as in \eqref{Whitt-decomp} we see that the above expression vanishes unless $\sigma$ is generic. This can be seen similarly as in \cite[Lemma 4.1]{Jana2020RS}. In this case it is equal to $\Psi(s,W_{\Phi},\overline{W_\varphi})$; see \S\ref{global-zeta-integral}. 
As a consequence, we may rewrite the spectral decomposition of $\mc{A}_{s}\Phi$ as
$$\mc{A}_{s}\Phi(g)=\int_{\gen}\sum_{\varphi\in \widetilde{\B}(\sigma)}\frac{\Psi(s,W_{\Phi},\overline{W_\varphi})}{\|\varphi\|^2}\varphi(g)\d\sigma.$$
The above is entire as a function of $s\in\C$.

Let $\widetilde{U}_n$ be the image of $U_n$ under the embedding $G_n\hookrightarrow G_{n+1}$ and let $\Phi_{\widetilde{U}_n}$ denote the constant term along $\widetilde{U}_n$, that is,
$$\Phi_{\widetilde{U}_n}(g):=\int_{[U_n]}\Phi\left[\begin{pmatrix}u&\\&1\end{pmatrix}g\right],\quad g\in G_{n+1}(\A).$$
Note that $\widetilde{U}_n$ is \emph{not} a unipotent radical of any parabolic of $G_{n+1}$, so the above integral need not vanish identically.
Since $[U_n]$ is compact, working as above, we spectrally decompose $\mc{A}_s\Phi_{\widetilde{U}_n}$ to obtain
$$\mc{A}_{s}\Phi(g)-\mc{A}_s\Phi_{\widetilde{U}_n}(g)=\int_{\gen}\sum_{\varphi\in \widetilde{\B}(\sigma)}\frac{\Psi(s,W_{\Phi},\overline{W_\varphi})}{\|\varphi\|^2}(\varphi(g)-\varphi_{U_n}(g))\d\sigma.$$
Once again the above converges absolutely and hence is entire as a function of $s\in\C$.

Recall that $\phi\in\pi$ is a cusp form. Let $\mathbf{s}:=(s_1,s_2)\in\C^2$ with $\Re(s_2)$ being sufficiently large. We take $g=\begin{pmatrix}h&\\&1\end{pmatrix}$ for $h\in G_{n-1}(\A)$ in the above equation and integrate the both sides against $\overline{\phi}|\det|^{s_2-1/2}$ over $h\in[G_{n-1}]$. Using Proposition \ref{simple-truncation-cuspidal} we write that
\begin{multline}\label{regularized-period}
\int_{[G_{n-1}]}\left(\mc{A}_{s_1}\Phi\left[\begin{pmatrix}h&\\&1\end{pmatrix}\right]-\mc{A}_{s_1}\Phi_{\widetilde{U}_n}\left[\begin{pmatrix}h&\\&1\end{pmatrix}\right]\right)\overline{\phi(h)}|\det(h)|^{s_2-1/2}\d h\\
=\int_{\gen}\sum_{\varphi\in \widetilde{\B}(\sigma)}\frac{\Psi(s_1,W_{\Phi},\overline{W_\varphi})\Psi(s_2,W_{\varphi},\overline{W_\phi})}{\|\varphi\|^2}\d\sigma.
\end{multline}
Both sides are absolutely convergent for any $\mathbf{s}\in\C^2$ with sufficiently large $\Re(s_2)$. Thus the expressions are entire in $s_1$ and holomorphic in $s_2$ in a right half plane.

We define the \emph{degenerate term} by
\begin{equation}\label{degenerate-term-def}
    \mc{D}(\mathbf{s},\Phi,\phi):=\int_{[G_1]}\int_{[G_{n-1}]}\Phi_{\widetilde{U}_{n}}\left[
    \begin{pmatrix}
    zh&&\\&z&\\&&1
    \end{pmatrix}\right]\ol{\phi(h)}|\det(h)|^{s_1+s_2-1}|z|^{n(s_1-1/2)}\d h\d z.
\end{equation}
Once again, rapid decay of $\Phi$ ensures that the integral in \eqref{degenerate-term-def} converges absolutely for any $\mathbf{s}\in\C^2$.

For any $\mathbf{s}\in\C^2$ we also define
\begin{equation}\label{def-moment}
    \mc{M}(\mathbf{s},\Phi,\phi):=\int_\gen \frac{\Lambda(s_1,\Pi\otimes\widetilde{\sigma})\Lambda(s_2,\sigma\otimes\widetilde{\pi})}{\mc{L}(\sigma)}H_S(\sigma; W_\Phi,W_\phi,\mathbf{s})\d\sigma,
\end{equation}
where $H_S(\sigma; W_\Phi,W_\phi,\mathbf{s})$ is as defined in \S\ref{prelim-local-factor}.
Note that the above expression is well-defined whenever $s_2$ is not a pole of $\Lambda(s_2,\sigma\otimes\tilde{\pi})$. Precisely, it is defined for $s_1\in\C$ and $\Re(s_2)\neq 1$; see \S\ref{sec:residue-term} for more details.

\begin{prop}\label{regularized-spectral-decomposition}
Let $\Phi\in\Pi$, $\phi\in\pi$, and $\mathbf{s}\in\C^2$. Recall $\mc{P}$, $\mc{D}$, and $\mathcal{L}$ from \eqref{P-period}, \eqref{degenerate-term-def}, and \eqref{harmonic-weights}, respectively. Let $S$ be a finite set of places such that both $\Phi$, $\phi$, $F$ are unramified at all $v\notin S$. Moreover, let $W_{\Phi,v}=W_{\Pi_v}$ and $W_{\phi,v}=W_{\pi_v}$ for $v\notin S$. Then
$$\mc{P}(\mathbf{s},\Phi,\phi)-\mc{D}(\mathbf{s},\Phi,\phi)=\mc{M}(\mathbf{s},\Phi,\phi),$$
for $s_1\in\C$ and sufficiently large $\Re(s_2)$.
\end{prop}

\begin{proof}
Note that $\mc{P}$ as in \eqref{P-period} can be written as
$$\int_{G_n}\mc{A}_{s_1}\Phi\left[\begin{pmatrix}h&\\&1\end{pmatrix}\right]\overline{\phi(h)}|\det(h)|^{s_2-1/2}\d h.$$
Similarly, we write \eqref{degenerate-term-def} as
$$\mc{D}(\mathbf{s},\Phi,\phi)=\int_{G_n}\mc{A}_{s_1}\Phi_{\widetilde{U}_n}\left[\begin{pmatrix}h&\\&1\end{pmatrix}\right]\overline{\phi(h)}|\det(h)|^{s_2-1/2}\d h.$$
We also choose $\widetilde{B}(\sigma)$ in \eqref{regularized-period} so that the only vectors for which $\Psi(s_1,W_{\Phi},\ol{W_{\varphi}})$ is non-vanishing satisfy $W_{\varphi,v}=W_{\sigma_v}$ for all $v\notin S$. Finally, using \eqref{harmonic-weights} and \eqref{euler-product} we conclude the proof.
\end{proof}

\section{Choice of the Vectors in the Local Factors}

\subsection{The local factor}\label{prelim-local-factor}

Let $\Pi$ and $\pi$ be cuspidal automorphic representations of $G_{n+1}(\A)$ and $G_{n-1}(\A)$, respectively. Let $\sigma$ be a generic unitary automorphic representation of $G_n(\A)$.

Let $S$ be a finite set of places of $F$. Let $v\in S$. For $W_1\in \Pi_v$ and $W_2\in\pi_v$, and $\mathbf{s}\in\C^2$ we define
\begin{equation}\label{defn-local-factor}
    H_v(\sigma_v)=H_v(\sigma_v; W_1,W_2,\mathbf{s}):=\sum_{W\in\B(\sigma_v)}\frac{\Psi_v(s_1, W_1,\ol{W})\Psi_v(s_2, W,\overline{W_2})}{L_v(s_1,\Pi_v\otimes\widetilde{\sigma}_v)L_v(s_2,\sigma_v\otimes\widetilde{\pi}_v)},
\end{equation}
where $\Psi_v$ is the $v$-adic zeta integral defined in \eqref{local-zeta-integral-def}.
Here $\B(\sigma_v)$ is an orthonormal basis of $\sigma_v$ under the unitary inner product defined in \S\ref{whittaker-kirillov-model}. The right-hand side does not depend on a choice of $\B(\sigma_v)$. We also define
$$h_v(\sigma_v) := {L_v(s_1,\Pi_v\otimes\widetilde{\sigma}_v)L_v(s_2,\sigma_v\otimes\widetilde{\pi}_v)}H_v(\sigma_v).$$
and
$$H_S(\sigma)=H_S(\sigma;\Phi,\phi,\mathbf{s}):=\prod_{v\in S}H_v(\sigma_v;W_{\Phi,v},W_{\phi,v},\mathbf{s}),$$
which are the local factor used in Theorem \ref{reciprocity}.

Lemma \ref{trivial-bound-zeta-integral} and Lemma \ref{trace-class-property} imply that the right-hand side of \eqref{defn-local-factor} is absolutely convergent for sufficiently large $\Re(\mathbf{s})$. As a function of $\mathbf{s}$ the function $H(\sigma;W_1,W_2,\mathbf{s})$ can be analytically continued to all of $\C^2$ which follows from the fact that $$\frac{\Psi_v(s,W_1,\ol{W})}{L_v(s,\Pi_v\otimes\widetilde{\sigma}_v)},\frac{\Psi_v(s,W,\ol{W_2})}{L_v(s,\sigma_v\otimes\widetilde{\pi}_v)}$$ are entire functions of $s$.

\subsection{Choices of vectors}\label{sec:choice-of-vectors}

Let $\q$ be an ideal of $\prod_{v<\infty}\o_v$ and $\p_0$ be a fixed prime such that $\q$, $\p_0$, and the discriminant $\Delta_F$ of $F$ are pairwise coprime.
We let
$$S:=\{\p_0\}\cup\{v\mid \q\}\cup\{v\mid \Delta_F\}\cup\{v\mid \infty\}.$$
We assume that $\Pi$ and $\pi$ are unramified at all  $v<\infty$.

We choose factorable cusp forms $\Phi\in\Pi$ and $\phi\in\pi$ by specifying their local Whittaker components $W_{\Phi,v}\in\Pi_v$ and $W_{\phi,v}\in\pi_v$. For all $v\notin S$ we choose $W_{\Phi,v}=W_{\Pi_v}$ and $W_{\phi,v}=W_{\pi_v}$. The choices at $v\in S$ are described below.

\subsubsection{At the places $v\mid\q$}

For any $f\in\Z_{\ge 0}$ and any prime $\p$ we define
\begin{equation}\label{choice-test-vector-original}
    W_{\Pi_\p}^{(f)}(g):=N(\p)^{-(n-1)f}\int_{F_\p^{n-1}}W_{\Pi_\p}\left(g\begin{pmatrix}\rm{I}_{n-1}&&\beta\\&1&\\&&1\end{pmatrix}\right)\mathbf{1}_{\o^{n-1}}(\beta\p^{f}) \d\beta,
\end{equation}
for $g\in G_{n+1}(F_\p)$.

Let $\p_v$ be the maximal ideal of $\o_v$. Let $\q=\prod_{v\mid\q}\p_v^{f_v}$.
For all $v\mid\q$ we choose $W_{\Phi,v} = W^{(f_v)}_{\Pi_v}$.

On the other hand, for all $v\mid\q$ we choose $W_{\phi,v}=W_{\pi_v}$.

\subsubsection{At the places $v\mid\Delta_F$}\label{subsubsec - v divide Delta}

At these places we need care due to the fact that the underlying additive character for the Whittaker model is not unramified. That is, $\psi_v(x)=\psi_{F_v}(\lambda_v x)$ for some $\lambda_v\in F_v$ such that $|\lambda_v|^{-1}=\Delta_v$ is the $v$-part of the discriminant $\Delta_F$ and $\psi_{F_v}$ is the standard \emph{unramified} additive character of $F_v$, \textit{i.e.} trivial on $\mf{o}_v$. Let
$$
a_r(\lambda_v)=\mathrm{diag}(\lambda_v^{r-1},\dots\lambda_v,1),\quad r\ge 1.
$$
Then, we take
\begin{equation}\label{choice-for-ramified-psi}
W_{\Phi,v}(g):=W_{\Pi_v}(a_{n+1}(\lambda_v)g),\quad W_{\phi,v}(g):=W_{\pi_v}(a_{n-1}(\lambda_v)g).
\end{equation}
To avoid confusion, $W_{\Pi_v}$ and $W_{\pi_v}$ are the \emph{usual} newvectors, \emph{i.e.} they are realized in the Whittaker model with respect to $\psi_{F_v}$.

\subsubsection{At the place $v=\p_0$}

Let $\tau$ be a fixed supercuspidal representation of $G_n(F_v)$ with trivial central character.  For normalization purposes we choose $\tau$ so that
$$L_v(s,\tau\otimes\widetilde{\tau})=\zeta_{v}(ns).$$
Such a $\tau$ exists; see \cite[\S2.1]{Ye2019RS}.

It is known that the normalized newvector $W_\tau\in C_c^\infty(Z_n(F_v)N_n(F_v)\bs G_n(F_v),\psi_v)$, \emph{i.e.} $W_\tau$ is compactly supported on $Z_n(F_v)N_n(F_v)\bs G_n(F_v)$; see \cite[Corollary 6.5]{casselman1980unramified-ii}.
We choose $W_{\Phi,v}$ so that
\begin{equation}\label{choice-test-vector-supercuspidal}
W_{\Phi,v}\left[\begin{pmatrix}
zh&\\&1
\end{pmatrix}\right]:=\mathbf{1}_{z\in\o_v^\times}W_{\tau}(h),\quad z\in G_1(F_v), h\in Z_n(F_v)\bs G_n(F_v).
\end{equation}
Note that the \eqref{choice-test-vector-supercuspidal} uniquely determines $W_{\Phi,v}$ due to the theory of Kirillov models; see \S\ref{whittaker-kirillov-model}.

Once again, we choose $W_{\phi,v}=W_{\pi_v}$.

\subsubsection{At the archimedean places $v\mid \infty$}\label{choice-test-vector-archimedean}

For each $v\mid\infty$, we choose $W_{\phi,v}\in\pi_v$ as a smooth vector such that $W_{\phi,v}(1)=1$. We also fix a sufficiently small $\epsilon>0$ and a ball $B\subset N_{n-1}(F_v)\backslash G_{n-1}(F_v)$ around the identity with sufficiently small radius, so that
\begin{equation}\label{choice-W-n-1-arch}
    |W_{\phi,v}(h)-1| < \epsilon,\quad \forall h\in B.
\end{equation}
Using the theory of Kirillov models, we choose $W_{\Phi,v}\in\Pi_v$ so that $W_{\Phi,v}\mid_{G_n(F_v)}$ is given by a fixed non-negative element in $C_c^\infty(N_n(F_v)\backslash G_n(F_v),\psi_v)$ such that
\begin{equation}\label{support-W-n+1-arch}
    G_1(F_v)\times N_{n-1}(F_v)\backslash G_{n-1}(F_v)\ni (z,h)\mapsto W_{\phi,v}\left[\begin{pmatrix}zh&&\\&z&\\&&1\end{pmatrix}\right]\text{ is supported on }B_0\times B,
\end{equation}
where $B_0\subset G_1(F_v)$ is a ball around $1$ with sufficiently small radius. Moreover, we normalize $W_{\Phi,v}$ by imposing that
\begin{equation}\label{normal-W-n+1-arch}
    \int_{N_{n-1}(F_v)\backslash G_{n-1}(F_v)}\int_{G_1(F_v)}W_{\phi,v}\left[\begin{pmatrix}zh&&\\&z&\\&&1\end{pmatrix}\right]\d^\times z\d h =1.
\end{equation}

\vspace{1cm}

We use the shorthand
\begin{equation}\label{arch-weight}
h_\infty(\sigma):=\prod_{v\mid\infty}h_v\left(\sigma_v;W_{\Phi,v},W_{\phi,v},\left(\frac12,\frac12\right)\right),
\end{equation}
where $h_v$ is defined in \S\ref{prelim-local-factor}, and
\begin{equation}\label{non-arch-weight}
H_\q(\sigma):=\prod_{v\mid\q}H_v\left(\sigma_v;W_{\Pi_v}^{(f_v)},W_{\pi_v},\left(\frac12,\frac12\right)\right),
\end{equation}
if $\q=\prod_{v\mid\q}\p_v^{f_v}$.

\section{Local Weight Function: Original Moment}\label{local-original}

In this section, let $v$ be a non-archimedean place and $F$ be the corresponding local field whose ring of integer is $\o$ with maximal ideal $\p$. The letters $\Pi$, $\sigma$, and $\pi$ denote irreducible generic unitary representations of $G_{n+1}(F)$, $G_n(F)$, and $G_{n-1}(F)$, respectively. Similarly, $L$, $\gamma$,... etc. denote local $L$-factors, local $\gamma$-factors... etc. We drop $F$ from the notations if there is no confusion.

We fix a $0\le\vartheta<1/2$. Further assume that $\Pi$ and $\pi$ are tempered and unramified. Also, let $\sigma$ be a $\vartheta$-tempered representation of $G_n$ with conductor exponent $c(\sigma)$.

The goal of this section is to analyse $H(\sigma)$ for the choices of the local vectors as in \S\ref{sec:choice-of-vectors} at the non-archimedean places $v\mid\q$, $v\mid\Delta_F$, and $\p_0$.

\subsection{At the places $v\mid\q$}

We devote this subsection to prove the following proposition. Recall the definition of the local factor $H_v$ from \eqref{defn-local-factor}.

\begin{prop}\label{main-local-prop-original}
Recall the choices of test vectors from \S\ref{sec:choice-of-vectors} for $v\mid\q$. We have
$$H_v(\sigma_v; W_{\Phi,v},W_{\phi,v},\mathbf{s})
\begin{dcases}
\ll_\epsilon N(\p_v)^{f_v\epsilon+c(\sigma_v)}N(\p_v)^{-n\frac{c(\sigma_v)+f_v}{2}}&\text{ if } c(\sigma_v)<f_v\\
=\frac{1}{\eta(\p_v^{f_v})\|W_{\sigma_v}\|^2}&\text{ if } c(\sigma_v)=f_v\\
=0 &\text{ if } c(\sigma_v)>f_v
\end{dcases}$$
for $\Re(\mathbf{s})=\left(\frac12,\frac12\right)$. Here $\eta(\p^{f}):=[G_n(\o):K_0(\p^f)]\asymp N(\p)^{f(n-1)}$.
\end{prop}

From \S\ref{sec:choice-of-vectors} we have
$$H_v(\sigma_v; W_{\Phi,v},W_{\phi,v},\mathbf{s})=H(\sigma;W_{\Pi_v}^{(f)},W_{\pi_v},\mathbf{s}).$$
For the rest of this subsection we suppress the notation for the place $v$.

First we explicitly write $W_\Pi^{(f)}$ defined in \eqref{choice-test-vector-original} as follows. Note that the Fourier transform of $\mathbf{1}_{\o^r}$ is itself. Thus,
\begin{equation}\label{choice-local-vector-explicit}
     W_{\Pi}^{(f)}\left[\begin{pmatrix}
     g&\\&1
     \end{pmatrix}\right]= W_{\Pi}\left[\begin{pmatrix}
     g&\\&1
     \end{pmatrix}\right]\mathbf{1}_{\o^{n-1}}(\ell(g)\p^{-f}),
\end{equation}
where $\ell(g)$ denotes the row vector constructed from the left-most $n-1$ elements of the last row of $g$.

\begin{lemma}\label{proj-onto-Kpf}
Let $W\in\sigma$. Then
$$
\Psi(s_1,W^{(f)}_{\Pi},\ol{W}) = \Psi(s_1,W^{(f)}_{\Pi},\ol{P_f(W)}),
$$
where $P_f$ is the orthogonal projection onto $\sigma^{K_0(\p^f)}$. In particular, $\Psi(s_1,W^{(f)}_{\Pi},\ol{W})$ vanishes unless $c(\sigma)\le f$
\end{lemma}

\begin{proof}
Let $\Re(s_1)$ be sufficiently large. Using \eqref{choice-local-vector-explicit} and sphericality of $W_\Pi$ we write the zeta integral in Iwasawa coordinates as
$$\int_{A_n}W_{\Pi}\left[\begin{pmatrix}a&\\&1\end{pmatrix}\right]|\det(a)|^{s_1-1/2}\int_{K_n}\mathbf{1}_{\o^{n-1}}(\ell(k)a_n\p^{-f})W(ak)\d k \frac{\d^\times a}{\delta(a)}.$$
Recall from the description of the spherical Whittaker function in \eqref{shintani} that the above vanishes unless $a_n\in\o$. Hence the inner $K_n$-integral is just a multiple of $P_e$, where $e=\max(0,f-v(a_n))$. The main formula now follows since $P_eP_f=P_e$ for $e\leq f$ and via meromorphic continuation. Finally, we conclude by recalling that $\sigma^{K_0(\p^{f})}$ is zero if $c(\sigma)> f$.
\end{proof}

\begin{lemma}\label{bound-gamma-factor}
Recall that $\sigma$ is a $\vartheta$-tempered representation with conductor exponent $c(\sigma)$. We have
$$\gamma(1/2+s,\sigma)\asymp N(\p)^{-c(\sigma)\Re(s)}.$$
for $s\in \C$ with $0\le \Re(s)<1/2-\vartheta$.
\end{lemma}

\begin{proof}
Recall that 
$$\gamma(1/2+s,\sigma)=\varepsilon(1/2+s,\sigma)\frac{L(1/2-s,\widetilde{\sigma})}{L(1/2+s,\sigma)},$$
with
$$\varepsilon(1/2+s,\sigma)=\varepsilon(1/2,\sigma)N(\p)^{-c(\sigma)s},$$
and $|\varepsilon(1/2,\sigma)|=1$ as $\sigma$ is unitary.

Also, there exist $\{\alpha_i\}_{i=1}^m, \{\beta_i\}_{i=1}^m\in\C^m$ with $m\le n$ such that $\max_i\{|\alpha_i|,|\beta_i|\}\leq N(\p)^\vartheta$ and
$$L(s,\sigma) = \prod_{i=1}^m(1-N(\p)^{-s}\alpha_i)^{-1},\quad L(s,\widetilde{\sigma}) = \prod_{i=1}^m(1-N(\p)^{-s}\beta_i)^{-1}.$$
Thus for $s$ as in the lemma we have
$$L(1/2-s,\widetilde{\sigma}),\: L(1/2+s,\sigma)\asymp 1.$$
Combining all the bounds we conclude.
\end{proof}

Note that if $\pi'$ is an unramified representation of $\PGL(r)$ with Langlands parameters $\{\nu_i\}_{i=1}^r$ then
$$\gamma(s,\pi'\otimes\sigma) = \prod_{i=1}^r\gamma(s+\nu_i,\sigma).$$
Moreover, if $\pi'$ is tempered, \emph{i.e.} $\Re(\nu_i)=0$, then Lemma \ref{bound-gamma-factor} implies that
\begin{equation}\label{bound-rs-gamma-factor}
    \gamma(1/2+s,\pi'\otimes\sigma) \asymp N(\p)^{-rc(\sigma)\Re(s)},
\end{equation}
whenever $0\le \Re(s)<1/2-\vartheta$.

\begin{lemma}\label{small-zeta-integral}
Fix $e\ge c(\sigma)$ and let $W\in\sigma$ be any $K_0(\p^e)$-invariant unit vector. Then 
$$\Psi(1/2+s_2,W,\ol{W_\pi})\ll N(\p)^{(e-c(\sigma))\epsilon},$$
for $\Re(s_2)=0$.
\end{lemma}

\begin{proof}
Let $0<\Re(s_2)<1/2-\vartheta$. 

Note that $K_0(\p^e)$-invariance of $W$ implies that if $W\left[\begin{pmatrix}h&\\&1\end{pmatrix}\right]$ is nonzero then $e_{n-1}h\in\mathfrak{o}^{n-1}$.
Thus the integral of $\Psi(1/2+s_2,W,W_2)$ can be written as the absolutely convergent integral 
$$\int_{N_{n-1}\backslash G_{n-1}}W\left[\begin{pmatrix}h&\\&1\end{pmatrix}\right]\mathbf{1}_{\o^{n-1}}(e_{n-1}h)\overline{W_\pi(h)}|\det(h)|^{s_2}\d h.$$
The absolute convergence follows from Lemma \ref{local-whittaker-bound}. By the Cauchy--Schwarz inequality and the unitarity of $W$, we conclude that the above expression has its absolute value bounded by the square root of 
$$\int_{N_{n-1}\backslash G_{n-1}}\mathbf{1}_{\o^{n-1}}(e_{n-1}h)|{W_\pi(h)}|^2|\det(h)|^{2\Re(s_2)}\d h.$$
Lemma \ref{local-whittaker-bound} confirms that the above integral converges absolutely as $\Re(s_2)>0$ and equals $L(2\Re(s_2),\pi\otimes\widetilde{\pi})$ \cite[Theorem 3.3]{Cogdell2007functions}. Thus we have $\Psi(s_2,W,\ol{W_\pi})\ll 1$.

Now we focus on $\Psi(1/2-s_2,W,W_\pi)$. We apply $\GL(n)\times\GL(n-1)$ local functional equation to obtain that $\Psi(1/2-s_2,W,W_\pi)$ equals
\begin{equation*}
   \gamma(1/2+s_2,\widetilde{\sigma}\otimes\pi)
    \int_{N_{n-1}\backslash G_{n-1}}\widetilde{W}\left[\begin{pmatrix}h&\\&1\end{pmatrix}\right]\overline{\widetilde{W_\pi}(h)}|\det(h)|^{s_2}\d h.
\end{equation*}
The above integral is absolutely convergent which follows from Lemma \ref{local-whittaker-bound}. Also from \eqref{bound-rs-gamma-factor} we obtain that the above gamma factor is bounded by $N(\p)^{-\Re(s_2)(n-1)c(\sigma)}$.

We note that $K_0({\p^e})$-invariance of $W$ implies that
$\widetilde{W}\left[\begin{pmatrix}h&\\&1\end{pmatrix}\right]$ is nonzero then $e_{n-1}h\in\p^{-e}$.
So the above integral can be written as
\begin{equation*}
\int_{N_n\backslash G_n}\widetilde{W}\left[\begin{pmatrix}h&\\&1\end{pmatrix}\right]\mathbf{1}_{\o^{n-1}}(e_{n-1}h\p^e)\overline{\widetilde{W_\pi}(h)}|\det(h)|^{s_2}\d h.
\end{equation*}
We again apply the Cauchy--Schwarz inequality on the $h$-integral. Using that $\|W\|=\|\widetilde{W}\|=1$ and $L(1,\sigma\otimes\widetilde{\sigma})\asymp 1$ we obtain that the above is bounded in absolute values by the square root of
$$\int_{N_{n-1}\backslash G_{n-1}}\mathbf{1}_{\o^{n-1}}(e_{n-1}h\p^e)|{\widetilde{W_\pi}(h)}|^2|\det(h)|^{2\Re(s_2)}\d h.$$
Changing variable $h\mapsto h\p^{-e}$ we see that the above is $N(\p)^{2(n-1)\Re(s_2)e}L(2\Re(s_2),\pi\otimes\widetilde{\pi})$. 
Thus for $\Re(s_2)>0$ we have
\begin{equation}\label{zeta-integral-negative}
    \Psi(1/2-s_2,W,W_\pi) \ll  N(\p)^{(n-1)\Re(s_2)(e-c(\sigma))}.
\end{equation}
Using the Phragm\'en--Lindel\"of convexity principle with $\Re(s_2)=\epsilon$ we may conclude. 
\end{proof}

\begin{lemma}\label{gl(n+1)-gl(n-1)-zeta-integral}
Let $\xi$ be any irreducible generic tempered unramified representation of $G_{n-1}$. Let $\chi$ be any unitary character of $F^\times$ and $l\in\Z_{\ge 0}$. Then the integral 
$$\int_{v(z)=l}\chi(z)\int_{A_{n-1}}W_\Pi\left[\begin{pmatrix}
az&&\\&z&\\&&1
\end{pmatrix}\right]\overline{W_\xi(a)}|\det(a)|^{s-1}\frac{\d^\times a}{\delta(a)}\d^\times z$$
is absolutely convergent for $\Re(s)>0$ and is $O(N(\p)^{-l(n/2-\epsilon)})$.
\end{lemma}

\begin{proof}
Note that as $\xi$ is tempered and $\Re(s)>0$ Lemma \ref{local-whittaker-bound} implies absolute convergence of the integral in the lemma.

Using Shintani's formula \eqref{shintani} we bound the above integral in absolute value by
$$N(\p)^{-nl/2}\sum_{m_1\ge\ldots\ge m_{n-1}\ge l}|\lambda_{\Pi}(m,l,0)\lambda_{\xi}(m)|N(\p)^{-\Re(s)\sum_{i}m_i}.$$
Temperedness of $\Pi$ and \eqref{spherical-whittaker-bound} implies that $\lambda_{\Pi}(m,l,0)\ll N(\p)^{\epsilon(\sum_i m_i+l)}$ and similarly, $\lambda_{\xi}(m)\ll N(\p)^{\epsilon(\sum_i m_i)}$.
Thus we obtain that the above infinite sum is convergent if $\Re(s)>0$ and is bounded.
\end{proof}

\begin{lemma}\label{gl(n)-gl(n-1)-zeta-integral}
Let $W\in\sigma^{K_0(\p^e)}$ be any unit vector for some $e\ge c(\sigma)$. Also let $\xi$ be as in Lemma \ref{gl(n+1)-gl(n-1)-zeta-integral}. Then the integral
$$\int_{A_{n-1}}{W\left[\begin{pmatrix}a&\\&1\end{pmatrix}\right]}\overline{W_\xi(a)}|\det(a)|^{-s}\frac{\d^\times a}{\delta(a)}$$
is absolutely convergent for $0<\Re(s)<1/2-\vartheta$ and is $O(N(\p)^{\Re(s)(n-1)(e-c(\sigma))})$.
\end{lemma}

\begin{proof}
This follows from the proof of Lemma \ref{small-zeta-integral}, in particular, from \eqref{zeta-integral-negative}.
\end{proof}

\begin{lemma}\label{big-zeta-integral-WP}
Let $W\in\sigma^{K_0(\p^e)}$ be any unit vector for some $e\ge c(\sigma)$. Let $l\in\Z_{\ge 0}$ and $\omega_\sigma$ be the central character of $\sigma$. Then for $\Re(s)=0$ the integral
$$\int_{v(z)=l}|z|^{ns}\overline{\omega_\sigma(z)}\int_{A_{n-1}}W_\Pi\left[\begin{pmatrix}
az&&\\&z&\\&&1
\end{pmatrix}\right]\overline{W\left[\begin{pmatrix}a&\\&1\end{pmatrix}\right]}|\det(a)|^s\frac{\d^\times a}{\delta(a)|\det(a)|}\d^\times z$$
is absolutely convergent and is $O(N(\p)^{-nl/2+\epsilon(e-c(\sigma)+l)})$.
\end{lemma}

\begin{proof}
Absolute convergence follows from Lemma \ref{local-whittaker-bound}.

We choose some $0<\eta<1/2-\vartheta$. We use non-archimedean Kontorovich--Lebedev--Whittaker transform
\cite{Guerreiro1702.08271} to write the integral as
\begin{multline*}
    \frac{1}{(n-1)!}\int_{\tau\in (S^1)^{n-1}}\int_{v(z)=l}|z|^{ns}\overline{\omega_\sigma(z)}\\
    \int_{A_{n-1}}W_\Pi\left[\begin{pmatrix}
a'z&&\\&z&\\&&1
\end{pmatrix}\right]\overline{W_{\sigma(\tau)}(a')}|\det(a')|^{\eta+s-1}\frac{\d^\times a'}{\delta(a')}\d^\times z\\
\int_{A_{n-1}}\overline{W\left[\begin{pmatrix}a''&\\&1\end{pmatrix}\right]}W_{\sigma(\tau)}(a'')|\det(a'')|^{-\eta}\frac{\d^\times a''}{\delta(a'')}
\prod_{i\neq j}(\tau_i-\tau_j)\frac{\d\tau_1}{\tau_1}\dots\frac{\d\tau_{n-1}}{\tau_{n-1}}.
\end{multline*}
Here by $\d\tau_j$ we denote the $2\pi i$-normalized Lebesgue measure on $S^1$ and $W_{\sigma(\tau)}$ is the normalized spherical vector of the representation whose Langlands parameters are given by $\tau$.

Note that the choice of $\eta$ ensures absolute convergence of all the integrals.  We use Lemma \ref{gl(n+1)-gl(n-1)-zeta-integral} to bound the outer $A_{n-1}\times F^\times$-integral by $N(\p)^{-l(n/2-\epsilon)}$ and Lemma \ref{gl(n)-gl(n-1)-zeta-integral} to bound the inner $A_{n-1}$-integral by $N(\p)^{\eta(n-1)(e-c(\sigma))}$. Bounding the $(S^1)^{n-1}$-integral trivially we conclude.
\end{proof}

\begin{lemma}\label{big-zeta-integral}
Let $W\in\sigma$ be a unit vector. Then 
$$\Psi(1/2+s_1,W_\Pi^{(f)},\ol{W})\ll {N(\p)^{c(\sigma)+\epsilon f}}{N(\p)^{-n\frac{c(\sigma)+f}{2}}},$$
for $\Re(s_1)=0$.
\end{lemma}

\begin{proof}
Using \eqref{choice-local-vector-explicit} and Iwasawa coordinates we write
\begin{multline*}
    \Psi(1/2+s_1,W_\Pi^{(f)},\ol{W})=\int_{F^\times}\overline{\omega_\sigma(z)}\int_{A_{n-1}} W_{\Pi}\left[\begin{pmatrix}
     az&&\\&z&\\&&1
     \end{pmatrix}\right]||z|^n\det(a)|^{s_1}\\
     \int_{K_n}\overline{W\left[\begin{pmatrix}
     a&\\&1
     \end{pmatrix}k\right]}\mathbf{1}_{\o^{n-1}}(\ell(k)z\p^{-f})\d k\frac{\d^\times a}{\delta(a)|\det(a)|}\d^\times z.
\end{multline*}
In the inner $K_n$-integral support condition of $\mathbf{1}_{\o^{n-1}}$ forces $l(k)\in (\p^e)^{n-1}$ where $e:=\max(0,f-v(z))$. Hence that integral can be written as
$$\int_{K_0(\p^e)}\overline{W\left[\begin{pmatrix}a&\\&1\end{pmatrix}k\right]}\d k=\vol(K_0(\p^e))\ol{P_e(W)},$$
where $P_e$ is as in Lemma \ref{proj-onto-Kpf}. Since $P_e$ is an orthogonal projection, we can write
$$\int_{K_0(\p^e)}\overline{W\left[\begin{pmatrix}a&\\&1\end{pmatrix}k\right]}\d k=c_e(W) W^{(e)}\left[\begin{pmatrix}a&\\&1\end{pmatrix}\right]$$ 
for some unit vector $W^{(e)}\in\bar{\sigma}^{K_0(\p^e)}$ and
$$c_e(W) \ll \vol(K_0(\p^e)) \asymp N(\p)^{-e(n-1)}.$$
Moreover, if $v(z)\ge f+1$ then the above $K_n$ integral vanishes unless $c(\sigma)=0$. Thus we write $\Psi(1/2+s_1,W_\Pi^{(f)},W)$ as
\begin{multline*}
    \sum_{e=c(\sigma)}^{f}c_{e}(W)\int_{v(z)=f-e}\overline{\omega_\sigma(z)}\int_{A_{n-1}} W_{\Pi}\left[\begin{pmatrix}
     az&&\\&z&\\&&1
     \end{pmatrix}\right]\\
     \overline{W^{(e)}\left[\begin{pmatrix}
     a&\\&1
     \end{pmatrix}\right]}||z|^n\det(a)|^{s_1}\frac{\d^\times a}{\delta(a)|\det(a)|}\d^\times z\\
     +\delta_{c(\sigma)=0}c_0(W)\int_{z\in\p^{f+1}}\overline{\omega_\sigma(z)}\int_{A_{n-1}} W_{\Pi}\left[\begin{pmatrix}
     az&&\\&z&\\&&1
     \end{pmatrix}\right]\\
     \overline{
     W^{(0)}
     \left[\begin{pmatrix}
     a&\\&1
     \end{pmatrix}\right]}|z|^n\det(a)|^{s_1}\frac{\d^\times a}{\delta(a)|\det(a)|}\d^\times z.
\end{multline*}
We let $\Re(s_1)=0$. Using Lemma \ref{big-zeta-integral-WP}, we bound the first summand by
$$\sum_{e=c(\sigma)}^f N(\p)^{-(n-1)e}N(\p)^{-n(f-e)/2}N(\p)^{\epsilon(f-c(\sigma))}\ll N(\p)^{-nf/2-(n/2-1)c(\sigma)+{\epsilon f}},$$
and the second summand by
$$\delta_{c(\sigma)=0}\sum_{l=f+1}^\infty N(\p)^{-l(n/2-\epsilon)} \ll \delta_{c(\sigma)=0}N(\p)^{-(f+1)(n/2-\epsilon)}.$$
Combining the above, we conclude the proof.
\end{proof}

\begin{proof}[Proof of Proposition \ref{main-local-prop-original}]
\textbf{Case $c(\sigma)<f$:} 
By Lemma \ref{proj-onto-Kpf} we can instead sum over $\B(\sigma^{K_0(\p^f)})$ in the definition of $H(\sigma)$. We now apply Lemma \ref{big-zeta-integral}, Lemma \ref{small-zeta-integral}, and the fact that for $\Re(s_1)=\Re(s_2)=1/2$ 
$$L(s_1,\Pi\otimes\widetilde{\sigma}), L(s_2,\sigma\otimes\widetilde{\pi})\asymp 1.$$
We conclude by noting that the dimension of $\sigma^{K_0(\p^f)}$ is a polynomial in $c(\sigma)$ and $f$; see \cite{Red91oldform}.

\textbf{Case $c(\sigma)=f$:} We take $\B(\sigma^{K_0(\p^f)})$ consisting of the single vector $\frac{W_\sigma}{\|W_\sigma\|}$.
Thus we obtain
$$H(\sigma;W^{(f)}_\Pi,W_\pi,\mathbf{s})=\|W_\sigma\|^{-2}\frac{\Psi(s_1,W_\Pi^{(f)},\ol{W_\sigma})}{L(s_1,\Pi\otimes\widetilde{\sigma})}\frac{\Psi(s_2,W_\sigma,\ol{W_\pi})}{L(s_2,\sigma\otimes\widetilde{\pi})}.$$
Applying \eqref{bkl-test-vector} we conclude the proof for this case.

\textbf{Case $c(\sigma)>f$:} Lemma \ref{proj-onto-Kpf} implies the case $c(\sigma)>f$ immediately.
\end{proof}

\subsection{At the places $v\mid \Delta_F$}\label{local-weight-original-places-delta}

Recall the choices of the test vectors in \eqref{choice-for-ramified-psi}. Since the test vectors differ from the spherical ones only by a multiplication on the left, it follows that $W_{\Phi,v}$ is right invariant by $\begin{pmatrix}
k&\\&1
\end{pmatrix}$, $k\in K_n$. Therefore, for a $K$-isotypic vector $W\in\sigma_v$ the zeta integral $\Psi(s_1,W_{\Phi,v},\ol{W})$ vanishes unless $W$ is spherical. Hence, $H_v(\sigma_v;W_{\Phi,v},W_{\phi,v},\mathbf{s})$ vanishes unless $\sigma_v$ is unramified in which case $H_v(\sigma_v)$ has only one summand. 
The vector corresponding to this summand has to be given, up to normalization, by a left translate of $W_{\sigma_v}$, the newvector for the unramified character.

In fact, we have
$$H(\sigma;W_{\Phi,v},W_{\phi,v},\mathbf{s})=\|W\|^{-2}\frac{\Psi_v(s_1,W_{\Phi,v},\ol{W})}{L(s_1,\Pi_v\otimes\widetilde{\sigma}_v)}\frac{\Psi_v(s_2,W,\ol{W_{\phi,v}})}{L_v(s_2,\sigma_v\otimes\widetilde{\pi}_v)},$$
where (see \S \ref{subsubsec - v divide Delta})
$$W(g):=W_{\sigma_v}(a_{n}(\lambda_v)g).$$
It follows by changing variables 
that
$$\begin{cases}
\Psi_v(s_1,W_{\Phi,v},\ol{W}) = |\lambda_v|^{\mu_1(s_1)}\Psi_v(s_1,W_{\Pi_v},\ol{W_{\sigma_v}}) = \Delta_v^{-\mu_1(s_1)} L_v(s_1,\Pi_v\otimes\widetilde{\sigma}_v),\\
\Psi_v(s_2,W,\ol{W_{\phi,v}}) = |\lambda_v|^{\mu_2(s_2)}\Psi_v(s_2,W_{\sigma_v},\ol{W_{\pi_v}}) = \Delta_v^{-\mu_2(s_2)} L_v(s_2,\sigma_v\otimes\widetilde{\pi}_v),\\
\|W\|^2 = |\lambda_v|^{\mu}\|W_{\sigma_v}\|^2=\Delta_v^{-\mu},\\
\end{cases}$$
where $\mu_1$ and $\mu_2$ are affine functions whose complex coefficients only depend on $n$; and $\mu$ is a constant only depending on $n$. Altogether, we get that

\begin{equation}\label{H-for-ramified-psi}
H_v(\sigma_v;W_{\Phi,v},W_{\phi,v},\mathbf{s})={\Delta_v^{-\mu_H(\mathbf{s})}},
\end{equation}

where $\mu_H$ is an affine functions whose complex coefficients depend only on $n$. We use the shorthand $\mu_H$ for $\mu_H\left(\frac12,\frac12\right)$.

\subsection{At the place $v=\p_0$}

Recall the choices of the local test vectors at $v=\p_0$ from \S\ref{sec:choice-of-vectors}.

\begin{prop}\label{cuspidal-projection}
We have
$$H_v(\sigma_v;W_{\Phi,v},W_{\phi,v},\mathbf{s})=
\begin{cases}
\varepsilon_v(1,\tau\otimes\widetilde{\tau})&\text{ if }\sigma_v=\tau,\\
0&\text{ otherwise};
\end{cases}$$
for $\mathbf{s}=\left(\frac12,\frac12\right)$. 
\end{prop}

The proof of this lemma is essentially contained in \cite[\S6.2]{Jana2020app}. We give a sketch of the proof for the sake of completeness.

\begin{proof}
For $W\in\B(\sigma_v)$ using \eqref{choice-test-vector-supercuspidal} we write
$$\Psi(1/2,W_{\Phi,v},\overline{W}) = \int_{Z_nN_n\bs G_n}W_\tau(g)\overline{W(g)}\d g.$$
The above integral is absolutely convergent as $W_\tau$ is compactly supported in $Z_nN_n\bs G_n$. Thus it defines a $G_n$-invariant sesquilinear pairing between $\tau$ and $\sigma_v$. By Schur's lemma, this pairing must vanish identically unless $\sigma_v=\tau$. In the latter case the pairing is proportional to the unitary inner product in $\tau$ as defined in \S\ref{whittaker-kirillov-model}.

If $\sigma_v=\tau$ we choose an orthonormal basis $\B(\sigma_v)$ containing $\frac{W_\tau}{\|W_\tau\|}$. As $\Psi(1/2,W_{\Phi,v},\overline{W})$ vanishes for $W$ orthogonal to $W_\tau$, we obtain
$$H_v(\sigma_v;W_{\Phi,v},W_{\phi,v},\mathbf{s})=\|W_\tau\|^{-2}\int_{Z_nN_n\bs G_n}|W_\tau(g)|^2\d g.$$
The above follows from \eqref{bkl-test-vector} and the fact that $L_v(s,\Pi\otimes\widetilde{\tau})=1$. We evaluate the above integral appealing to \cite[Lemma 6.2]{Jana2020app} and conclude.
\end{proof}

\section{Local Weight Function: Dual Moment}\label{local-dual}

In this section we adopt the same notations as in \S\ref{local-original}. In particular, $\sigma$ is an irreducible generic unitary $\vartheta$-tempered representation of $G_n(F)$ for some $0\le\vartheta<1/2$.

Recall $w^\ast$ from the statement of Proposition \ref{abstract-reciprocity}. 
Let $W_1\in\Pi$ and $W_2\in\pi$, and $\mathbf{s}\in\C^2$. We define the dual local factor by
\begin{equation*}
    \wc{H}(\sigma)=\widecheck{H}(\sigma;W_1,W_2,\mathbf{s}):=H(\sigma;\Pi(w^\ast)W_1,W_2,\mathbf{s})
\end{equation*}
where $H(\sigma)$ is as in \eqref{defn-local-factor}.

The goal of this section is to analyse $\wc{H}(\sigma)$, as in Proposition \ref{main-local-prop-dual}, for the choices of the local vectors as in \S\ref{sec:choice-of-vectors} at the places $v\mid\q$.

We suppress the subscript $v$ as usual.

For $W\in\sigma$ we define
\begin{equation}\label{W-check}
  \widecheck{W}^{(f)}(g)
  =N(\p)^{-(n-1)f}\int_{F^{n-1}}{W}\left[g\begin{pmatrix}\rm{I}_{n-1}&\beta\\&1\end{pmatrix}\right]\mathbf{1}_{\o^{n-1}}(\beta\p^{f}) \d\beta.
\end{equation}
Clearly, for $g\in G_{n-1}$,
\begin{equation}\label{unipotent-average}
    \widecheck{W}^{(f)}\left[\begin{pmatrix}g&\\&1\end{pmatrix}\right]=W\left[\begin{pmatrix}g&\\&1\end{pmatrix}\right]\mathbf{1}_{\o^{n-1}}(e_{n-1}g\p^{-f}),
\end{equation}
where $e_{n-1}$ is the standard row vector $(0,\dots,0,1)$.

\begin{lemma}\label{simplifying-dual-weight}
Recall the vector $W_\Pi^{(f)}$ from \eqref{choice-test-vector-original}. Then $\wc{H}(\sigma;W_\Pi^{(f)},W_\pi,\mathbf{s})$ vanishes unless $\sigma$ is unramified in which case
$$\wc{H}(\sigma;W_\Pi^{(f)},W_\pi,\mathbf{s})=\frac{\Psi(s_2,\wc{W}^{(f)}_\sigma,\ol{W_\pi})}{L(s_2,\sigma\otimes\widetilde{\pi})},$$
where $\wc{W}^{(f)}_\sigma$ is defined as in \eqref{W-check}.
\end{lemma}

\begin{proof}
As $W_\Pi$ is spherical on $G_{n+1}$ we write
\begin{align*}
    {W}_{\Pi}^{(f)}(gw^\ast)&=N(\p)^{-(n-1)f}\int_{F^{n-1}}{W}_{\Pi}\left[gw^\ast\begin{pmatrix}\rm{I}_{n-1}&&\beta\\&1&\\&&1\end{pmatrix}\right]\mathbf{1}_{\o^{n-1}}(\beta\p^{f}) \d\beta\\
    &=N(\p)^{-(n-1)f}\int_{F^{n-1}}{W}_{\Pi}\left[g\begin{pmatrix}\rm{I}_{n-1}&\beta&\\&1&\\&&1\end{pmatrix}\right]\mathbf{1}_{\o^{n-1}}(\beta\p^{f}) \d\beta.
\end{align*}
Let $\Re(s_1)$ and $\Re(s_2)$ be sufficiently large.
Changing variables we write $\Psi(s_1,\Pi(w^\ast)W_{\Pi}^{(f)},\ol{W})$ as
$$N(\p)^{-(n-1)f}\int_{F^{n-1}}\Psi(s_1,W_{\Pi},\ol{\sigma(u_{-\beta})W})\mathbf{1}_{\o^{n-1}}(\beta\p^{f}) \d\beta,$$
where
$u_{\beta}=\begin{pmatrix}\rm{I}_{n-1}&\beta\\&1\end{pmatrix}$. Let $\B(\sigma)$ be an orthonormal $K$-isotypic basis. We multiply the above integrand by $\Psi(s_2,W,\overline{W_{\pi}})$ and sum over the basis $\left\lbrace\sigma\left(\begin{pmatrix}\rm{I}_{n-1}&\beta\\&1\end{pmatrix}\right)W\right\rbrace_{W\in\B(\sigma)}$.

By linearity of the zeta integral, this yields 
\begin{equation*}
    \wc{H}(\sigma)=\sum_{W\in\B(\sigma)}\frac{\Psi(s_1,W_\Pi,\ol{W})}{L(s_1,\Pi\otimes\widetilde{\sigma})}\frac{\Psi(s_2,\wc{W}^{(f)},\ol{W_\pi)}}{L(s_2,\sigma\otimes\widetilde{\pi})}.
\end{equation*}
By Lemma \ref{proj-onto-Kpf}, $\Psi(s_1,W_\Pi,\ol{W})$ vanishes unless $\sigma$ is unramified in which case we may replace the above sum by a sum over $W\in\B(\sigma^{K_n})$, which we may take to consist only of the vector ${W_\sigma}$ (recall \eqref{L2-whittaker}). Applying \eqref{bkl-test-vector} and meromorphic continuation we conclude the proof.
\end{proof}

\begin{lemma}\label{bound-zeta-integral-dual-side}
Let $\sigma$ be unramified. We have
$$\Psi(1/2+s_2,\wc{W}^{(f)}_\sigma,\ol{W_\pi})\ll N(\p)^{-f(n-1)(1/2-\vartheta)},$$
for $\Re(s_2)=0$.
\end{lemma}

\begin{proof}
Using Lemma \ref{local-whittaker-bound} we write the zeta integral as the absolutely convergent integral
$$\int_{N_{n-1}\backslash G_{n-1}}\wc{W}^{(f)}_\sigma\left[\begin{pmatrix}g&\\&1\end{pmatrix}\right]\overline{W_\pi(g)}|\det(g)|^{s_2}\d g.$$
Using \eqref{unipotent-average} and sphericality of the vectors $W_\sigma$ and $W_\pi$ we write the above as
$$\int_{A_{n-1}}W_\sigma\left[\begin{pmatrix}a&\\&1\end{pmatrix}\right]\overline{W_\pi(a)}|\det(a)|^{s_2}\int_{K_{n-1}}\mathbf{1}_{\o^{n-1}}(e_{n-1}ka_{n-1}\p^{-f})\d k\frac{\d^\times a}{\delta(a)}.$$
The inner $K_{n-1}$-integral vanishes unless $v(a_{n-1})\ge f$ in which case the integral evaluates to $1$.
Using Shintani's formula \eqref{shintani} for the spherical vectors we may write the above as
$$\sum_{m_1\ge\dots\ge{m_{n-1}\ge f}}\lambda_\sigma(m,0)\overline{\lambda_\pi(m)}N(\p)^{-(1/2+s_2)\sum_i m_i}.$$
Using \eqref{spherical-whittaker-bound} we obtain
$$\lambda_\sigma(m,0)\ll N(\p)^{(\vartheta+\epsilon)\sum_i m_i},$$
and temperedness of $\pi$ implies that $\lambda_\pi(m)\ll N(\p)^{\epsilon\sum_i m_i}$. Thus we get the absolute value of $\Psi(1/2+s_2,\wc{W}^{(f)}_\sigma,\overline{W_\pi})$ is bounded by
$$\sum_{m_1\ge\dots\ge{m_{n-1}\ge f}}N(\p)^{-(1/2-\vartheta-\epsilon)\sum_i m_i}\ll N(\p)^{-f(n-1)(1/2-\vartheta-\epsilon)},$$
which follows as $\vartheta<1/2$.
\end{proof}

\begin{prop}\label{main-local-prop-dual}
Fix $0\le \vartheta<1/2$ and let $\sigma$ be a $\vartheta$-tempered representation. Then $\wc{H}(\sigma;W_\Pi^{(f)},W_\pi,\mathbf{s})$ vanishes unless $\sigma$ is unramified in which case
$$\wc{H}(\sigma;W_\Pi^{(f)},W_\pi,\mathbf{s})\ll N(\p)^{-f(n-1)(1/2-\vartheta-\epsilon)},$$
for $\Re(\mathbf{s})=\left(\frac12,\frac12\right)$.
\end{prop}

\begin{proof}
As $\pi$ is tempered and $\sigma$ is $\vartheta$-tempered with $\vartheta<1/2$ we have
$L(s_2,\sigma\otimes\widetilde{\pi})\asymp 1$
for $\Re(s_2)=1/2$. 
We conclude using Lemma \ref{simplifying-dual-weight} and Lemma \ref{bound-zeta-integral-dual-side}.
\end{proof}

\section{The Residue Term: Dual Side}\label{sec:residue-term}

Recall $\mc{M}(\mathbf{s},\Phi,\phi)$ from \eqref{def-moment} which is originally defined for sufficiently large $\Re(\mathbf{s})$. The goal of this section is to meromorphically continue $\mc{M}(\mathbf{s},\Phi,\phi)$ on the right of $\Re(\mathbf{s})=\left(\frac12,\frac12\right)$. 

We use the shorthand ``$\Re(s_i)>c$ for $i=1,2$'' as ``$\Re(\mathbf{s})>c$'' and similarly for the symbols $\geq,\,<,\,\leq$ and $=$.
Our aim is to prove the following proposition.

\begin{prop}\label{analytically-continued-spectral-decomposition}
Recall $\mc{P}$, $\mc{D}$, and $\mc{M}$ from \eqref{P-period}, \eqref{degenerate-term-def}, and \eqref{def-moment}, respectively. Then we have 
$$\mc{P}-\mc{D} = \mc{M} +\mc{R},$$
where $\mc{R}$ is defined in \eqref{def-R}. All of the above are evaluated at $(\mathbf{s},\Phi,\phi)$ with $1/2-\epsilon\le\Re(\mathbf{s})<1$ and are holomorphic in this region.
Moreover, if the local components of $\Phi$ and $\phi$ are chosen as in \S\ref{sec:choice-of-vectors} then
$$\mc{R}\left(\left(\frac12,\frac12\right),\wc{\Phi},\phi\right) \ll_{\Pi,\pi} N(\q)^{-f((n-2)/2-\epsilon)},$$
where $\wc{\Phi}$ is as in Proposition \ref{abstract-reciprocity}.
\end{prop}

\subsection{Analytic continuation}\label{Analytic-continuation-sec}

We recall the definition of $\mc{M}$ below:
$$\mc{M}(\mathbf{s},\Phi,\phi)=\int_{\gen}\frac{\Lambda(s_1,\Pi\otimes \widetilde{\sigma})\Lambda(s_2,\sigma\otimes\widetilde{\pi})}{\mathcal{L}(\sigma)}H_S(\sigma;W_{\Phi},W_{\phi},\mathbf{s})\d\sigma.$$
Using the parametrization of the generic spectrum via cuspidal data (\emph{cf}.\ \S \ref{subsec:global-classification}) we decompose $\M=\M_0+\M_1$,
where $\M_1$ corresponds to the contribution coming from the parabolic $Q$ attached the partition $n=(n-1)+1$ and $\M_0:=\M-\M_1$.
we further decompose $\M_1=\M_2+\M_\pi$ where $\M_\pi$ corresponds to the the Eisenstein spectrum attached to 
\begin{equation}\label{problematic-reps}
\sigma(\pi,z):=\mathcal{I}\left(M,\pi\cdot|\det|^{z}\otimes |\cdot|^{-(n-1)z}\right),\quad z\in i\R,
\end{equation}
where $M\cong G_{n-1}\times G_1$ is the Levi of $Q$ and $\M_2:=\M_1-\M_\pi$.
In other words,
\begin{equation}\label{M-pi-def}
\M_\pi(\mathbf{s},\Phi,\phi):=c_Q\int_{\Re(z)=0}\frac{\Lambda(s_1,\Pi\otimes \widetilde{\sigma(\pi,z)})\Lambda(s_2,\sigma(\pi,z)\otimes\widetilde{\pi})}{\mathcal{L}(\sigma(\pi,z))}H_S(\sigma(\pi,z))\d z,
\end{equation}
where $c_Q$ is a positive constant depending only on $Q$ (see \cite[Main Theorem]{arthur1979eisenstein}) and $\d z$ is the $2\pi i$-normalized Lebesgue measure on $i\R$.

Note that $\Lambda(s,\Pi\otimes\widetilde{\sigma})$ is an entire function of $s$ and $\Lambda(s,\sigma\otimes\widetilde{\pi})$ is an entire function of $s$ if $\sigma$ is not of the form \eqref{problematic-reps}, which follows from \cite[Theorem 4.2]{Cogdell2007functions}. Applying positivity of $\mathcal{L}(\sigma)$ for unitary $\sigma$ (see \S\ref{sec:harmonic-weights}) and entireness of $H_S(\sigma)$ as a function of $\mathbf{s}$ (see \S\ref{prelim-local-factor}) we obtain that
$\M_0$ and $\M_2$ are entire functions of $\mathbf{s}$. 

The above argument also implies that $\mc{M}$, as defined in \eqref{def-moment}, is also holomorphic for $s_1\in\C$ and $1/2\le\Re(s_2)<1$.

We now perform the meromorphic continuation of $\mc{M}_\pi(\mathbf{s},\Phi,\phi)$ which is originally defined for large $\Re(\mathbf{s})$ to $\Re(\mathbf{s})\ge 1/2$ (actually, for all $s_1\in\C$). The argument is a generalization of \cite[Proposition 9.1]{Nunes2020reciprocity}.

\begin{lemma}\label{holomorphicity-local-weight}
Let any $v\in S$ and $W_1\in\Pi_v$ and $W_2\in\pi_v$. Then the local factor $H_v(\sigma(\pi_v,z);W_1,W_2,\mathbf{s})$, originally defined for $z\in i\R$ and large $\Re(\mathbf{s})$, can be meromorphically continued as a function of $z,s_1,s_2$ to all of $\C^3$ such that it is holomorphic for $\mathbf{s}\in\C^2$ and sufficiently small $|\Re(z)|$.
\end{lemma}

The proof can be extracted from \cite{FLO2012representations}. We prove it here for completeness and readers' convenience.

\begin{proof}
From the definition of the local weight in \eqref{defn-local-factor} we know that $H_v(\sigma_v)$ does not depend on the choice of basis $\B(\sigma_v)$. We construct a basis $\B(\sigma(\pi_v,z))$ for $z\in i\R$ via \emph{flat sections}, as follows. 

Let $\tau^{(K\cap M)}_{{\pi_v},z}$ be the restriction of $\tau_{{\pi_v},z}:={\pi_v}\otimes|\det|^{z}\bigotimes |\cdot|^{-(n-1)z}$ to $K\cap M$. Thus $\tau^{(K\cap M)}_{{\pi_v},z}$ is independent of $z$ and so we call it $\tau^{(K\cap M)}_{\pi_v}$. We construct $\tau_{\pi_v}^{(K)}:=\mathrm{Ind}^K_{K\cap M}\tau_{\pi_v}^{(K\cap M)}$ which is, consequently, independent of $z$. We choose an orthonormal basis $\B(\tau_{\pi_v}^{(K)})$ of $\tau_{\pi_v}^{(K)}$. For all $\eta\in\B(\tau_{\pi_v}^{(K)})$ we choose a function $\xi_{{\pi_v},z,\eta}$ on $G(F_v)$ which satisfies
$$\xi_{{\pi_v},z,\eta}(nmk)=\delta_Q^{1/2+z}(m)\tau_{{\pi_v},z}(m)\eta(k),n\in N_Q,m\in M,k\in K,$$
where $N_Q$ is the unipotent radical attached to $Q$ and $\delta_Q$ is the modular character attached to $Q$.

Let $W^{\mathrm{Jac}}$ be Jacquet's functional from the induced model to the Whittaker model; see \cite[\S1.4]{FLO2012representations} and \cite{Jacquet04IntRep}. Then it is known from \cite[Proposition A.2]{FLO2012representations} that
$$\zeta \langle W^{\mathrm{Jac}}_{\xi_{{\pi_v},z,\eta_1}},W^{\mathrm{Jac}}_{\xi_{{\pi_v},z,\eta_2}}\rangle_0=\langle {\xi_{\pi_v,z,\eta_1}},{\xi_{\pi_v,z,\eta_2}}\rangle = \langle \eta_1,\eta_2\rangle,$$
where $\zeta$ is an absolute constant. In particular, from \eqref{inner-product-normalization} we conclude that $\|W^{\mathrm{Jac}}_{\xi_{{\pi}_v,z,{\eta}}}\|^2$ does not vanish when $z$ varies in a small enough neighbourhood of the imaginary axis.

We rewrite the definition \eqref{defn-local-factor} of $H_v(\sigma(\pi_v,z);W_1,W_2,\mathbf{s})$ as
$$\sum_{\eta\in\B(\tau^{(K)}_{\pi_v})}\frac{\Psi_v\left(s_1,W_1,W^{\mathrm{Jac}}_{\xi_{\widetilde{\pi}_v,-z,\ol{\eta}}}\right)\Psi_v\left(s_2,W^{\mathrm{Jac}}_{\xi_{{\pi}_v,z,{\eta}}},\ol{W_2}\right)}{\|W^{\mathrm{Jac}}_{\xi_{{\pi}_v,z,{\eta}}}\|^2L_v(s_1,\Pi_v\otimes\sigma(\widetilde{\pi}_v,-z))L_v(s_2,\sigma(\pi_v,z)\otimes\widetilde{\pi})}.$$
From \cite[Corollaire 3.5]{jacquet1967fonctions} (also see \cite[Theorem 4]{Jacquet04IntRep}) it is known that $W^{\mathrm{Jac}}_{\xi_{\pi_v,z,\eta}}$ is entire in $z$. We work as in \S\ref{prelim-local-factor} to conclude the proof noting the meromorphic properties of $L_v(1,\sigma(\pi_v,z)\otimes\sigma(\widetilde{\pi}_v,-z))$.
\end{proof}

We define the \emph{residue term} $\mc{R}(\mathbf{s},\Phi,\phi)$ to be
\begin{equation}\label{def-R}
r_{F,Q}\frac{\Lambda(s_1+s_2-1,\Pi\otimes\widetilde{\pi})\Lambda(s_1+s_2-1+n(1-s_2),\Pi)}{\Lambda(1+n(1-s_2),\pi)}
H_S(\sigma(\pi_v,1-s_2))\mathcal{L}_\infty(\sigma(\pi,1-s_2)),
\end{equation}
where 
$$\mathcal{L}_\infty(\sigma(\pi,z)):=\prod_{v\mid\infty}L_v(1,\pi_v,\Ad)L_v(1+nz,{\pi}_v)L_v(1-nz,\widetilde{\pi}_v)$$
and $r_{F,Q}$ is a certain positive constant depending only on $F$ and $Q$ that can be explicitly computed in terms of the residues of the Dedekind zeta function and the automorphic Plancherel density (only depends on $Q,n$).

\begin{lemma}\label{analytic-continuation-prop}
The function $\mc{M}_\pi(\mathbf{s},\Phi,\phi)$, initially defined for large $\Re(\mathbf{s})$, admits a meromorphic continuation to $\Re(\mathbf{s})> \frac12-\epsilon$ and is given by the sum of $\mc{R}(\mathbf{s},\Phi,\phi)$ and the right-hand side of \eqref{M-pi-def}, where
$\mc{R}(\mathbf{s},\Phi,\phi)$ is given by \eqref{def-R}. In particular, both the summands are holomorphic on $\Re(\mathbf{s})=1/2$.
\end{lemma}

\begin{proof}
Using \eqref{problematic-reps} we rewrite \eqref{M-pi-def}
as
\begin{multline*}
\mathcal{M}_\pi(\mathbf{s},\Phi,\phi)=c_Q\int_{\Re(z)=0}\Lambda(s_1-z,\Pi\otimes \widetilde{\pi})\Lambda(s_2+z, \pi\otimes \widetilde{\pi})\\
\times\frac{\Lambda(s_1+(n-1)z,\Pi)\Lambda(s_2-(n-1)z,\widetilde{\pi})}{\mathcal{L}(\sigma(\pi,z))}
{H}_S(\sigma(\pi,z)) \d z.
\end{multline*}
Let $\delta>0$ be sufficiently small in terms of all parameters. We use the positivity of $\mathcal{L}(\sigma(\pi,z))$ on $z\in i\R$ to define a continuous even function $\kappa:\Rr\mapsto (0,\delta)$ so that $\mathcal{L}(\sigma(\pi,{z}))$ does not vanish for $-2\kappa(\Im(z))<\Re(z)<0$. 

We notice that we can \textit{analytically} continue $\mc{M}_{\pi}(\mathbf{s},\Phi,\phi)$ to $\Re(\mathbf{s})>1$, since in that region, the integrand is holomorphic in $\mathbf{s}$. Now, suppose that
$$1<\Re(\mathbf{s})<1+\kappa(\Im(\mathbf{s})).$$
We shift the contour of the integral defining $\mc{M}_{\pi}(\mathbf{s},\Phi,\phi)$ to $\Re(z) = -\kappa(\Im(z))$. We pick up a simple pole in this process at $z=1-s_2$ which is from $\Lambda(s_2+z,\pi\otimes\widetilde{\pi})$. This is due to the fact that $\Lambda(s_1-z,\Pi\otimes\widetilde{\pi})$, $\Lambda(s_1+(n-1)z,\Pi)$, and $\Lambda(s_2-(n-1)z,\widetilde{\pi})$ are entire in $\mathbf{s}$ and $z$. Moreover, Lemma \ref{holomorphicity-local-weight} and our choice of $\delta$ ensures that $H_S(\sigma(\pi,z))/\mathcal{L}^S(\sigma(\pi,z))$ is holomorphic as a function of $z$ in the $\delta$-neighbourhood of $0$. 

We call the residue at $z=1-s_2$ to be $\mc{R}(\mathbf{s},\Phi,\phi)$ as in \eqref{def-R}. The expression in \eqref{def-R} follows from 
$$
\mathcal{L}(\sigma(\pi,z))=d_{F,M}L(1,\pi,\Ad)L(1+nz,{\pi})L(1-nz,\widetilde{\pi}),
$$
which follows Lemma \ref{harmonic-weight-computation}. 

Moreover, we observe that in view of our choice of $\kappa$, the shifted integral defines a holomorphic function in the region
$$1-\kappa(\Im(\mathbf{s}))<\Re(\mathbf{s})<1+\kappa(\Im(\mathbf{s})).$$
We now take $\mathbf{s}$ satisfying $1-\kappa(\Im(\mathbf{s}))<\Re(\mathbf{s})<1$. We may shift the contour back to the line $\Re(z)=0$, crossing no poles in the process. This proves the desired formula for $$1-\kappa(\Im(\mathbf{s}))<\Re(\mathbf{s})<1.$$
We note from \eqref{M-pi-def} that $\mc{M}$ is holomorphic in $1/2-\epsilon<\Re(\mathbf{s})<1$. This follows from the holomorphicity of the Rankin--Selberg $L$-functions and the same in $z$ of the factor $H_S(\sigma(\pi,z))/\mathcal{L}(\sigma(\pi,z))$ on $\Re(z)=0$.

Similarly, meromorphicity (and holomorphicity on $\Re(\mathbf{s})=1/2$) of $\mc{R}$, as in \eqref{def-R}, in the same region follows from that of the local $L$-factors and the proof of Lemma \ref{holomorphicity-local-weight}.
\end{proof}

\subsection{Local computations for the residue term at $v\mid \q$}

Recall the choice of the test vectors at $v\mid\q$ from \S\ref{sec:choice-of-vectors}.
In this subsection, we analyse the $v$-adic local components in the residue term in \eqref{def-R} for $\mathbf{s}=\left(\frac12,\frac12\right)$ and $\Phi$ replaced by $\wc{\Phi}$ which is defined in Proposition \ref{abstract-reciprocity}. In the rest of this subsection we will only work $v$-adically and suppress the subscript $v$.

Once again, we meromorphically continue $\wc{H}(\sigma(\pi,z))$ for these choices from $z\in i\R$ to $z\in \C$. However, the method is different than that in the proof of Lemma \ref{holomorphicity-local-weight}.

\begin{lemma}\label{analytic-continue-hp}
We have
\begin{equation*}
\wc{H}(\sigma(\pi,1/2);{W}^{(f)}_{\Pi},W_{\pi},\mathbf{s})=\frac{\Psi(s,\widecheck{W}^{(f)}_{\sigma(\pi,1/2)},\ol{W_{\pi}})}{L(s,\sigma(\pi,1/2)\otimes \widetilde{\pi})},
\end{equation*}
where $\wc{W}^{(f)}$ is given by \eqref{W-check}.
\end{lemma}

\begin{proof}
Let $z\in i\R$ so that $\sigma(\pi,z)$ is unitary. We apply Lemma \ref{simplifying-dual-weight} to obtain
$$\wc{H}(\sigma(\pi,z);{W}^{(f)}_{\Pi},W_{\pi},\mathbf{s})=\frac{\Psi(s,\widecheck{W}^{(f)}_{\sigma(\pi,z)},\ol{W_{\pi}})}{L(s,\sigma(\pi,z)\otimes \widetilde{\pi})}.$$
Using holomorphicity of $\widecheck{W}^{(f)}_{\sigma(\pi,z)}$ and $L(s,\sigma(\pi,z)\otimes\widetilde{\pi})^{-1}$ in $z$, and holomorphicity of the above ratio in $s$ we analytically continue the above ratio. Hence we conclude upon taking $z=1/2$.
\end{proof}

We need the following strengthened version of Lemma \ref{bound-zeta-integral-dual-side}. 

\begin{lemma}\label{main-technical-lemma-residue}
Let $\sigma$ be any unramified representation (not necessarily unitary) of $G_n(F)$ with trivial central character and $\widecheck{W}^{(f)}$ be given as in \eqref{W-check}. Then we have
$$\frac{\Psi(s,\widecheck{W}^{(f)}_{\sigma},\ol{W_{\pi}})}{L(s,\sigma\otimes\widetilde{\pi})}=\sum_{k=0}^{\min(f,n-1)}\frac{(-1)^{k}e_k(A_{\widetilde{\pi}})\lambda_{\widetilde{\sigma}}(f-k,0,\ldots,0)}{N(\p)^{(f(n-1)+k)s}},$$
for any $s\in\C$. Here $e_k(A_{\pi})$ denotes the $k$-th elementary symmetric polynomial evaluated on the Satake parameters $A_\pi$ of $\pi$.
\end{lemma}

\begin{proof}
By entireness of the ratio in the left-hand side above  in $s\in\C$ it is enough to prove the formula for large $\Re(s)$.

Using \eqref{unipotent-average} we have
$$\Psi(s,\widecheck{W}^{(f)}_{\sigma},\ol{W_{\pi}})=\int_{\substack{a\in A_{n-1}\\v(a_{n-1})\ge f}} W_{\sigma}\left[\begin{pmatrix}
a&\\
&1
\end{pmatrix}\right] \ol{W_{\pi}(a)}|\det(a)|^{s-1/2}\frac{\d^\times a}{\delta(a)}.$$
We now apply Shintani's formula \eqref{shintani}, leading to
\begin{equation}\label{A-sum-at-last!}
\Psi(s,\widecheck{W}^{(f)}_{\sigma},\ol{W_{\pi}})=\sum_{m_1\ge\ldots\ge m_{n-1}\geq f}\frac{\lambda_{\sigma}(m,0)\lambda_{\widetilde{\pi}}(m)}{N(\p)^{s\sum_i m_i}}.
\end{equation}
Let $A_{\sigma}$ be the Satake parameters of $\sigma$ and fix an ordered \(n\)-tuple $\alpha=(\alpha_1,\ldots,\alpha_n)\in\C^n$ such that $A_{\sigma}=\{\alpha_j\}_{j=1}^n$ as a multiset. We assume that $\alpha$ is \emph{regular}, that is, $\alpha_i\neq \alpha_j$ for $i\neq j$ and prove the lemma. The general result will follow by continuity.

Expanding the determinant in the numerator of $\lambda_{\sigma}$ along the bottom row and using the well-known formula for the Vandermonde determinant
$$V_n(\alpha):=\prod_{1\le i<j\le n}(\alpha_i-\alpha_j),$$
we deduce that
\begin{equation*}
\lambda_{\sigma}(m,0)=\sum_{j=1}^{n}\frac{(-1)^{n+j}\alpha_j^{-1}\lambda_{\sigma^{(j)}}(m)V_{n-1}(\alpha^{(j)})}{V_n(\alpha)},
\end{equation*}
where by $\alpha^{(j)}$ we denote the \((n-1)\)-tuple obtained from $\alpha$ by removing $\alpha_j$ and by $\sigma^{(j)}$ the local representation whose Satake parameters is the multiset corresponding to $\alpha^{(j)}$.
Moreover, notice that 
$$\lambda_{\sigma^{(j)}}(m)=\alpha_j^{-f}\lambda_{\sigma^{(j)}}(m_1-f,\ldots,m_{n-1}-f).$$
Applying this and the equality
$$\sum_{m_1\ge \ldots \ge m_{n-1}\ge 0}\frac{\lambda_{\sigma^{(j)}}(m) \lambda_{\widetilde{\pi}}(m)}{N(\p)^{s\sum_i m_i}}=L(s,\sigma^{(j)}\otimes \widetilde{\pi})$$
to \eqref{A-sum-at-last!} gives
$$\Psi(s,\widecheck{W}^{(f)}_{\sigma},\ol{W_{\pi}})=\sum_{j=1}^{n}\frac{(-1)^{n+j}\alpha_j^{-f-1}V_{n-1}(\alpha^{(j)})L(s,\sigma^{(j)}\otimes \widetilde{\pi})}{N(\p)^{(n-1)f s}V_n(\alpha)}.$$
Since $\sigma$ is unramified, we have the factorization
$$L(s,\sigma\otimes\widetilde{\pi})=L(s,\sigma^{(j)}\otimes\widetilde{\pi})L(s,\widetilde{\pi}\otimes\chi_j),$$
where $\chi_j$ is the unramified character for which $\chi_j(\p)=\alpha_j$. Therefore,
$$\Psi(s,\widecheck{W}^{(f)}_{\sigma},\ol{W_{\pi}})=L(s,\sigma\otimes \widetilde{\pi})\sum_{j=1}^{n}\frac{(-1)^{n+j}\alpha_j^{-f-1}V_{n-1}(\alpha^{(j)})L(s,\widetilde{\pi}\otimes \chi_j)^{-1}}{N(\p)^{f(n-1)s}V_n(\alpha)}.$$
Expanding $L(s,\widetilde{\pi}\otimes\chi_j)^{-1}$ and changing the order of summation, we deduce that
$$\Psi(s,\widecheck{W}^{(f)}_{\sigma},\ol{W_{\pi}})=L(s,\sigma\otimes \widetilde{\pi})\sum_{k=0}^{n-1}\frac{(-1)^ke_k(A_{\widetilde{\pi}})}{N(\p)^{(f(n-1)+k)s}}\sum_{j=1}^{n}\frac{(-1)^{n+j}\alpha_j^{k-f-1}V_{n-1}(\alpha^{(j)})}{V_n(\alpha)}.$$
Let $\alpha^{-1}$ denote the \(n\)-tuple $\{\alpha_j^{-1}\}_{j=1}^n$. Notice that the inner-most summand can be written as
$$(-1)^{j+1}\frac{\left(\alpha_j^{-1}\right)^{f-k+n-1}V_{n-1}((\alpha^{-1})^{(j)})}{V_n(\alpha^{-1})}.$$
Thus, for $f\geq k$, we write the inner-most sum as
$$\lambda_{\widetilde{\sigma}}(f-k,0,\ldots,0)$$
by expanding the numerator of the Schur polynomial along the top row. 

Otherwise, we have $0\leq f-k+n-1< n-1$ and using the very same idea we may relate this sum to a quotient of two determinants and where the numerator has two equal rows and thus it vanishes.

Altogether this leads to
$$\Psi(s,\widecheck{W}^{(f)}_{\sigma},\ol{W_{\pi}})=L(s,\sigma\otimes \widetilde{\pi})\sum_{k=0}^{\min(f,n-1)}\frac{(-1)^{k}e_{k}(A_{\widetilde{\pi}})\lambda_{\widetilde{\sigma}}(f-k,0,\ldots,0)}{N(\p)^{(f(n-1)+k)s}},$$
which concludes the proof.
\end{proof}

\begin{lemma}\label{bound-residue-p-factor}
We have
$$\wc{H}(\sigma(\pi,1/2);W_\Pi^{(f)},W_\pi,\mathbf{s})\ll N(\p)^{-f((n-2)/2-\epsilon)}$$
for $\mathbf{s}=\left(\frac12,\frac12\right)$.
\end{lemma}

\begin{proof}
Using Lemma \ref{analytic-continue-hp} and Lemma \ref{main-technical-lemma-residue} for $\sigma=\sigma(\pi,1/2)$ and the given $\mathbf{s}$ we obtain
$$\wc{H}(\sigma(\pi,1/2);W_\Pi^{(f)},W_\pi,\mathbf{s})=\sum_{k=0}^{\min(f,n-1)}\frac{(-1)^{k}e_k(A_{\widetilde{\pi}})\lambda_{\widetilde{\sigma(\pi,1/2)}}(f-k,0,\ldots,0)}{N(\p)^{(f(n-1)+k)/2}}.$$
Now it follows from the equality of $L$-functions
$$L(s,\widetilde{\sigma(\pi,1/2)})=L(s-1/2,\widetilde{\pi})\zeta(s+(n-1)/2)$$
that
$$\lambda_{\widetilde{\sigma(\pi,1/2)}}(f-k,0,\ldots,0)=\sum_{j=0}^{f-k}\lambda_{\widetilde{\pi}}(j,0,\dots,0)|\p^{j}|^{-1/2}|\p^{f-k-j}|^{(n-1)/2}.$$
The right-hand side above is $\ll N(\p)^{(f-k)(1/2+\epsilon)}$ as $\pi$ is tempered. Hence, we conclude.
\end{proof}

\subsection{Proof of Proposition \ref{analytically-continued-spectral-decomposition}}

The first assertion follows from Proposition \ref{regularized-spectral-decomposition} and Lemma \ref{analytic-continuation-prop} along with the discussion preceding it.

From the definitions \eqref{P-period} and \eqref{degenerate-term-def}, respectively, it is clear that $\P$ and $\D$ are entire in $\mathbf{s}$ as $\Phi$ and $\phi$ are cuspidal. The argument preceding Lemma \ref{analytic-continuation-prop} implies that $\mc{M}$ is entire in $s_1$ and is holomorphic on the region $1/2\le \Re(s_2)<1$. This implies that $\mc{R}$ is holomorphic on the same region.

The final assertion on $\mc{R}$ follows from the definition of $\mc{R}$ in \eqref{def-R} upon applying Lemma \ref{bound-residue-p-factor} and that $\mathcal{L}_\infty(\sigma(\pi,1/2))\ll 1$ along with holomorphicity of the local weights on $\Re(\mathbf{s})=1/2$, as in Lemma \ref{holomorphicity-local-weight}, and of the Rankin--Selberg $L$-functions.

\section{The Degenerate Term: Dual Side}\label{sec:degenerate-term}

In this section we use the same notations as in the beginning of \S\ref{sec-global}. The goal of this section is to study the degenerate term defined in \eqref{degenerate-term-def}, as follows. The following proposition can be thought as a generalization of \cite[Proposition 9.1]{Nunes2020reciprocity}.

\begin{prop}\label{degenerate-term}
Let $\Phi$ and $\phi$ be given as in \S\ref{sec:choice-of-vectors}. Recall $\wc{\Phi}$ from Proposition \ref{abstract-reciprocity}. We have
$$\D\left(\left(\frac12,\frac12\right),\wc{\Phi},\phi\right) =\Delta_F^{-\mu_{\Omega}}\D_\infty\frac{L^{\p_0}(1,\Pi\otimes \widetilde{\pi})L^{\p_0}(n/2,\widetilde{\Pi})}{L^{\p_0}(1+n/2,\widetilde{\pi})},$$
where $\mu_{\Omega}$ is a constant depending only on $n$ and $\D_\infty$ is as defined in \eqref{archimedean-factor-degenerate}.

Moreover, we choose the local test vectors at the archimedean places in such a way so that $\D_\infty \asymp 1$.
\end{prop}

First, we define a local factor which will be served as the local component of the degenerate term. For
$W_1\in \Pi_v$ and $W_2\in\pi_v$ we define $\Omega_v(\mathbf{s},W_1,W_2)$ by
\begin{equation}\label{degenerate-local-def}
\int_{N_{n-1}(F_v)\backslash G_{n-1}(F_v)}\int_{G_1(F_v)}W_{1}\left[\begin{pmatrix}zh&&\\&z&\\&&1\end{pmatrix}
\right]\ol{W_{2}(h)}|\det(h)|^{s_1+s_2-1}|z|^{n(s_1-1/2)}\d^{\times}z\d h.
\end{equation}
It follows from Lemma \ref{local-whittaker-bound} that the integral in \eqref{degenerate-local-def} converges absolutely for $\Re(\mathbf{s})\ge1/2$.

We also define
\begin{equation}\label{archimedean-factor-degenerate}
    \D_\infty(\mathbf{s})=\D_\infty\left(\mathbf{s},\otimes_{v\mid\infty}W_{\Phi,v},\otimes_{v\mid\infty}W_{\phi,v}\right):=\prod_{v\mid\infty}\Omega_v(\wc{\mathbf{s}},W_{\Phi,v},W_{\phi,v}).
\end{equation}
We abbreviate $\D_\infty\left(\frac12,\frac12\right)$ as $\D_\infty$.

\begin{lemma}\label{degenerate-eulerian}
Recall $\wc{\mathbf{s}}$ from \eqref{s-prime}. Then for any $\Phi\in \Pi$ and $\phi\in\pi$ we have
$$\D(\mathbf{s},\Phi,\phi)=\prod_v\Omega_v(\wc{\mathbf{s}},W_{\wc{\Phi}, v},W_{\phi,v})$$
for sufficiently large $\Re(\wc{\mathbf{s}})$.
\end{lemma}

\begin{proof}
Working as in the proof of Proposition \ref{abstract-reciprocity} we may deduce that
\begin{equation}\label{D-period}
\mathcal{D}(\mathbf{s},\Phi,\phi)=\int_{[G_{n-1}]}\int_{[G_1]}\wc{\Phi}^\hash\left[\begin{pmatrix}zh&&\\&z&\\&&1\end{pmatrix}
\right]
\ol{\phi(h)}
|\det(h)|^{\wc{s}_1+\wc{s}_2-1}|z|^{n(\wc{s}_1-1/2)}\d^{\times}z\d h.
\end{equation}
where $\wc{\mathbf{s}}$ is as in \eqref{s-prime} and
$${\Phi}^\hash(g):=\int_{[\widecheck{U}_n]}\Phi(ug)\d u$$
such that
$\widecheck{U}_n:=w^{\ast-1}\widetilde{U}_nw^\ast$ which consists of matrices of the form
$$\begin{pmatrix}
\rm{I}_{n-1}&&x\\
&1&\\
&&1
\end{pmatrix},
\quad x\in M_{(n-1)\times 1}.$$
We first compute the Fourier--Whittaker expansion of $\wc{\Phi}^\hash$. Using \eqref{Whitt-decomp} we write
$$\widecheck{\Phi}^\hash(g)=\sum_{\gamma \in N_n(F)\backslash G_n(F)}\int_{[\wc{U}_n]}W_{\widecheck{\Phi}}\left[\begin{pmatrix}
\gamma&\\
&1
\end{pmatrix}
ug\right]\d u.$$
We notice that for $\gamma \in G_n(F)$,
$$W_{\widecheck{\Phi}}\left[\begin{pmatrix}
\gamma &\\&1 
\end{pmatrix}
\begin{pmatrix}
\rm{I}_{n-1}&&x\\&1&\\&&1
\end{pmatrix}g
\right] =\psi_0(\ell(\gamma)x) W_{\widecheck{\Phi}}\left[\begin{pmatrix}
\gamma &\\&1 
\end{pmatrix}
g
\right],$$
where $\ell(\gamma)$ denote the row matrix formed by the left most $n-1$ entries of the last row of $\gamma$.
Integrating both sides above over $x$ in $(F\bs \Aa)^{n-1}$ we conclude in the above Fourier--Whittaker expansion of $\wc{\Phi}^\hash$ we must have $\ell(\gamma)=0$; equivalently, $\gamma\in N_n(F)\bs G_1(F)P_n(F)$. Using the isomorphism $N_n\backslash G_1P_n\cong N_{n-1}\backslash G_{n-1} \times G_1$ we write
$$\wc{\Phi}^\hash(g)=\sum_{\gamma\in N_{n-1}(F)\bs G_{n-1}(F)}\sum_{q\in G_1(F)}W_{\wc{\Phi}}\left[\begin{pmatrix}
q\gamma&&\\&q&\\&&1
\end{pmatrix}g\right].$$
Inserting the above expansion in \eqref{D-period} and executing an unfolding-folding we obtain that
$\mathcal{D}({\mathbf{s}},\Phi,\phi)$ can be written as
$$\int_{N_{n-1}(\mathbb{A})\backslash G_{n-1}(\mathbb{A})}\int_{G_1(\mathbb{A})}W_{\widecheck{\Phi}}\left[\begin{pmatrix}zh&&\\&z&\\&&1\end{pmatrix}
\right]\ol{W_{\phi}(h)}|\det(h)|^{\wc{s}_1+\wc{s}_2-1}|z|^{n(\wc{s}_1-1/2)}\d^{\times}z\d h.$$
The above expression converges absolutely for sufficiently large $\Re(\wc{\mathbf{s}})$. Also the above is Eulerian and can be factored as $\prod_v\Omega_v(\wc{\mathbf{s}},W_{\wc{\Phi}_v},W_{\phi_v})$. 
\end{proof}

In the next subsections we will analyse the local integrals $\Omega_v$ at various $v$ and for the test vectors chosen in \S\ref{sec:choice-of-vectors}.

\subsection{At the places $v\mid\q$}

In this subsection, we
work at the place $v$ and unless otherwise stated the objects below are $v$-adic.

\begin{lemma}\label{degenerate-p-place}
Let $f\in\Z_{\ge 0}$. Recall $\Omega$ from \eqref{degenerate-local-def} and $W_\Pi^{(f)}$ from \eqref{choice-test-vector-original}. We have
$$\Omega(\mathbf{s},W_\Pi^{(f)},W_\pi)=\frac{L(s_1+s_2,\Pi\otimes\widetilde{\pi})L(ns_1,\widetilde{\Pi})}{L\left((n+1)s_1 + s_2,\widetilde{\pi}\right)}$$
for $\mathbf{s}\in\C^2$ with large $\Re(\mathbf{s})$.
\end{lemma}

\begin{proof}
First, it follows from \eqref{choice-local-vector-explicit} that for $h\in G_{n-1}$ and $z\in G_1$, one has
$$W^{(f)}_{\Pi}\left[\begin{pmatrix}zh&&\\&z&\\&&1\end{pmatrix}\right]=W_{\Pi}\left[\begin{pmatrix}zh&&\\&z&\\&&1\end{pmatrix}\right].$$
Therefore, it suffices to take $f=0$.

We use \eqref{shintani} and see that
\begin{equation}\label{OmegaAfterShintani}
\Omega(\mathbf{s},W_{\Pi},W_{\pi})=\sum_{\substack{m_1\ge \ldots\ge m_{n-1}\ge 0\\l\ge 0}}\frac{\lambda_{\Pi}(m_1+l,\ldots,m_{n-1}+l,l,0)\lambda_{\widetilde{\pi}}(m)}{N(\p)^{(s_1+s_2)\sum_i m_i+nl s_1}}.
\end{equation}
We now proceed as in the proof of Lemma \ref{main-technical-lemma-residue} and follow the notations there.

Let $\alpha=(\alpha_1,\ldots,\alpha_{n+1})$ be an ordered \((n+1)\)-tuple of complex numbers such that $\{\alpha_j\}_{j=1}^{n+1}$ is the multiset of Satake parameters of $\Pi$. Here, as in Lemma \ref{main-technical-lemma-residue}, we first suppose that $\alpha$ is regular and deduce the general result by continuity.

Using the representation of $\lambda_{\Pi}$ in terms of Schur polynomials we obtain that the summand above is
$$\lambda_{\Pi}(m_1+l,\ldots,m_{n-1}+l,l,0)=\sum_{j=1}^{n+1}(-1)^{n+j+1}\frac{\alpha_j^{-l-1}\lambda_{\Pi^{(j)}}(m,0)V_n(\alpha^{(j)})}{V_{n+1}(\alpha)}.$$
Using this in \eqref{OmegaAfterShintani}, changing the order of summation and performing the sum over $m$, we obtain
$$\Omega(\mathbf{s},W_{\Pi},W_{\pi})=\sum_{j=1}^{n+1}(-1)^{n+j+1}\frac{L(s_1+s_2,\Pi^{(j)}\otimes \widetilde{\pi})V_n(\alpha^{(j)})}{V_{n+1}(\alpha)}\sum_{l\ge 0}\frac{\alpha_j^{-l-1}}{N(\p)^{lns_1}}.$$
Let $\chi_j$ be the unramified character of $F^{\times}$ such that $\chi_j(\p)=\alpha_j$. Since $\Pi$ is unramified, we have the factorization
$$L(s,\Pi\otimes\widetilde{\pi})=L(s,\Pi^{(j)}\otimes\widetilde{\pi})L(s,\widetilde{\pi}\otimes\chi_j).$$
Hence, we have
$$\Omega(\mathbf{s},W_{\Pi},W_{\pi})=L(s_1+s_2,\Pi\otimes\widetilde{\pi})\sum_{j=1}^{n+1}(-1)^{n+j+1}\frac{L(s_1+s_2, \widetilde{\pi}\otimes \chi_j)^{-1}V_n(\alpha^{(j)})}{V_{n+1}(\alpha)}\sum_{l\ge 0}\frac{\alpha_j^{-l-1}}{N(\p)^{lns_1}}.$$
Expanding the term $L(s_1+s_2,\widetilde{\pi}\otimes \chi_j)^{-1}$ and changing the order of summation the above can be written as
$$L(s_1+s_2,\Pi\otimes\widetilde{\pi})\sum_{k=0}^{n-1}(-1)^ke_k(A_{\widetilde{\pi}})N(\p)^{-k(s_1+s_2)}\sum_{l\geq 0}\frac{1}{N(\p)^{ln s_1}}
\sum_{j=1}^{n+1}(-1)^{n+j+1}\frac{\alpha_j^{k-l-1}V_n(\alpha^{(j)})}{V_{n+1}(\alpha)},$$
where $e_k(A_{\pi})$ denotes the $k$-th elementary symmetric polynomial on the Satake parameters $A_\pi$ of $\pi$. 

Working as in Lemma \ref{main-technical-lemma-residue} we deduce that the inner-most sum vanishes unless $l\ge k$ in which case the sum evaluates to $$\lambda_{\widetilde{\Pi}}(l-k,0,\ldots,0).$$
Thus we obtain that
\begin{equation*}
\Omega(\mathbf{s},W_{\Pi},W_{\pi})=L(s_1+s_2,\Pi\otimes\widetilde{\pi})\sum_{k=0}^{n-1}(-1)^ke_k(A_{\widetilde{\pi}})N(\p)^{-k((n+1)s_1+s_2)}\sum_{l\ge 0}\frac{\lambda_{\widetilde{\Pi}}(l,0\dots,0)}{N(\p)^{lns_1}}
\end{equation*}
Noting that
$$L\left((n+1)s_1 + s_2,\widetilde{\pi}\right)^{-1}=\sum_{k=0}^{n-1}(-1)^ke_k(A_{\widetilde{\pi}})N(\p)^{-k((n+1)s_1+s_2)}$$
and
$$L(ns_1,\widetilde{\Pi})=\sum_{l\ge 0}\frac{\lambda_{\widetilde{\Pi}}(l,0\dots,0)}{N(\p)^{lns_1}},$$
we conclude the proof.
\end{proof}

\subsection{At the places $v\mid \Delta_F$}

Recall from \S\ref{sec:choice-of-vectors} that in this case our choice of test vectors is given by \eqref{choice-for-ramified-psi}, so that in this case $\Omega_v(\mathbf{s},W_{\Phi,v},W_{\phi,v})$ is given by
$$\int_{A_{n-1}}\int_{G_1}W_{\Pi_v}\left[a_{n+1}(\lambda_v)\begin{pmatrix}za&&\\&z&\\&&1\end{pmatrix}
\right]\ol{W_{\pi_v}(a_{n-1}(\lambda_v)a)}|\det(a)|^{s_1+s_2-1}|z|^{n(s_1-1/2)}\d^{\times}z\frac{\d^\times a}{\delta(a)}.$$
Now it follows from the change of variables $\lambda_v a_{n-1}(\lambda_v)a\mapsto a$ and $\lambda_v z\mapsto z$ that the above equals
$$|\lambda_v|^{\mu_{\Omega}(\mathbf{s})}\Omega_v(\mathbf{s},W_{\Pi_v},W_{\pi_v}),$$
where $\mu_{\Omega}(\mathbf{s})$ is an affine function whose coefficients are complex numbers depending only on $n$. Recall that $|\lambda_v|^{-1}=\Delta_v$ is the $v$ part of the discriminant. Also notice that Lemma \ref{degenerate-p-place}
also includes the unramified computation as a special case by simply taking $f=0$. This implies that
\begin{equation}\label{degenerate-discrimant-places}
\Omega_v(s,W_{\Phi,v},W_{\phi,v})=\Delta_v^{-\mu_{\Omega}(\mathbf{s})}\frac{L_v(s_1+s_2,\Pi_v\otimes\widetilde{\pi}_v)L_v(ns_1,\widetilde{\Pi}_v)}{L_v\left((n+1)s_1 + s_2,\widetilde{\pi}_v\right)}.
\end{equation}

\subsection{A combined result}

By putting together Lemma \ref{degenerate-p-place} and \eqref{degenerate-discrimant-places}, and taking into consideration that both formul{\ae} apply for the unramified places $v\not\in S$, we have the following lemma.

\begin{lemma}\label{degenerate-unramified}
For large $\Re(\mathbf{s})$ we have
$$\Omega_v(\mathbf{s},W_{\Phi,v},W_{\phi,v})=\Delta_v^{-\mu_{\Omega}(\mathbf{s})}\frac{L_v(s_1+s_2,\Pi_v\otimes\widetilde{\pi}_v)L_v(ns_1,\widetilde{\Pi}_v)}{L_v\left((n+1)s_1 + s_2,\widetilde{\pi}_v\right)}$$
for $v<\infty$ and $v\neq \p_0$ where the local vectors are as in \S\ref{sec:choice-of-vectors}.
\end{lemma}

\subsection{At the place $\p_0$}

Recall the choices of the test vectors at $v$ in \S\ref{sec:choice-of-vectors}.

\begin{lemma}\label{degenerate-supercuspidal}
We have
$$\Omega_v(\mathbf{s},W_{\Phi,v},W_{\phi,v})=1$$
for any $\mathbf{s}\in\C^2$.
\end{lemma}

\begin{proof}
From \eqref{choice-test-vector-supercuspidal} we have
$$\Omega_v(\mathbf{s},W_{\Phi},W_{\phi})=\int_{N_{n-1}(F_v)\bs G_{n-1}(F_v)}W_{\tau}\left[\begin{pmatrix}
h&\\&1
\end{pmatrix}\right]\ol{W_{\pi_v}(h)}|\det(h)|^{s_1+s_2-1}\d h.$$
The above is absolutely convergent for all $\mathbf{s}$ which follows from compact support of $W_\tau$ in the domain of integration.
From \eqref{bkl-test-vector} we conclude that the above integral is $L_v(s_1+s_2-1/2,\tau\otimes\widetilde{\pi}_v)$. As $\tau$ is supercuspidal and $\pi_v$ is unramified this local $L$-factor is equal to $1$, which follows from \cite{Miyauchi2014Whittaker}.
\end{proof}

\subsection{At the archimedean places}

Let $v\mid\infty$. Recall the choices of the test vectors at $v$ in \S\ref{sec:choice-of-vectors}.

\begin{lemma}\label{degenrate-archimedean}
Let $\mathbf{s}=\left(\frac{1}{2},\frac{1}{2}\right)$. Then for $W_{\Phi,v}$ and $W_{\phi,v}$ as in \S\ref{choice-test-vector-archimedean} we have
$$\Omega_v(\mathbf{s},W_{\Phi,v},W_{\phi,v})\asymp 1,$$
where the implied constants depend at most on the archimedean components of $\Pi$ and $\pi$.
\end{lemma}

\begin{proof}
Note that 
$$\Omega_v(\mathbf{s},W_{\Phi,v},W_{\phi,v}=\int_{N_{n-1}(F_v)\backslash G_{n-1}(F_v)}\int_{G_1(F_v)}W_{\Phi,v}\left[\begin{pmatrix}zh&&\\&z&\\&&1\end{pmatrix}
\right]\ol{W_{\phi,v}(h)}\d^{\times}z\d h.$$
Using the normalizations and supports of $W_{\Phi,v}$ and $W_{\phi,v}$ specified in \S\ref{choice-test-vector-archimedean} we write the above as
$$1+\int_{B}\int_{B_0}W_{\Phi,v}\left[\begin{pmatrix}zh&&\\&z&\\&&1\end{pmatrix}
\right]\left(\ol{W_{\phi,v}(h)}-1\right)\d^{\times}z\d h,$$
and the second summand in the last expression can be bounded by $\epsilon$. Making $\epsilon$ sufficiently small we conclude.
\end{proof}

\subsection{Proof of Proposition \ref{degenerate-term}}

Let $\Re(\mathbf{s})$ be large enough. Using Lemma \ref{degenerate-eulerian}, Lemma \ref{degenerate-p-place}, and Lemma \ref{degenerate-supercuspidal} we obtain
$$\D(\wc{\mathbf{s}},\wc{\Phi},\phi)= \Delta_F^{-\mu_\Omega(\mathbf{s})}\D_\infty(\mathbf{s})
\frac{L^{\p_0}(s_1+s_2,\Pi\otimes\widetilde{\pi})L^{\p_0}(ns_1,\widetilde{\Pi})}{L^{\p_0}\left((n+1)s_1 + s_2,\widetilde{\pi}\right)},$$
where  $\D_\infty(\mathbf{s})$ is defined in \eqref{archimedean-factor-degenerate}.
Both sides above are holomorphic in $\Re(\mathbf{s})\ge\left(\frac12,\frac12\right)$ which follows from holomorphicity and non-vanishing of $L$-functions in the region of absolute convergence and Lemma \ref{holomorphicity-local-weight}. We conclude by inserting $\mathbf{s}=\left(\frac12,\frac12\right)$, defining $\mu_\Omega:=\mu_\Omega\left(\frac12,\frac12\right)$, and appealing to Lemma \ref{degenrate-archimedean}.

\section{Proofs of the Main Results}\label{sec:proofs}

In this section we put together all of our previous results and finally prove the main theorems. We use the shorthand $\M$ for $\M(\mathbf{s},\Phi,\phi)$ and $\widecheck{\M}$ for $\M(\widecheck{\Phi},\phi,\wc{\mathbf{s}})$, where $\widecheck{\Phi}$ and $\wc{\mathbf{s}}$ are given in Proposition \ref{abstract-reciprocity}. Similar notations are used for the functions $\mc{P}$, $\D$ and $\mc{R}$.

\subsection{Proof of Theorem \ref{reciprocity}}

From Proposition \ref{analytically-continued-spectral-decomposition}
we have
$$\P=\M+\mc{R}+\D$$
in the region $1/2\le \Re(\mathbf{s})< 1$.
Now, Proposition \ref{abstract-reciprocity} tells us that
$$\P=\widecheck{\P}.$$
We apply Proposition \ref{analytically-continued-spectral-decomposition} to $\mc{\wc{P}}$ with $1/2\le\Re(\mathbf{s})<1$.
Thus we obtain
$$\M=\wc{\M}+\widecheck{\mc{R}}+\wc{\D}-\mc{R}-\D,$$
which proves Theorem \ref{reciprocity}.

\subsection{A few lemmas}

\begin{lemma}\label{rapid-decay-local-weight}
Let $v$ be any place. Then for any $W_1\in\Pi_v$ and $W_2\in\pi_v$ we have
$$H_v(\sigma_v; W_1,W_2,\mathbf{s})\ll_{W_1,W_2,N} C(\sigma_v)^{-N}$$
for any $\mathbf{s}\in\C^2$.
\end{lemma}

This lemma allows us to truncate the Plancherel integral so the archimedean parameters have \emph{essentially} bounded size.

\begin{proof}
We use Lemma \ref{trivial-bound-zeta-integral} in the definition \eqref{defn-local-factor}. We see that it suffices to show that for any $N$ and $d$ there exists an $M$ such that
$$\sum_{W\in\B(\sigma_v)}S_{-M}(W)S_d(W)\ll C(\sigma_v)^{-N}.$$
If $v$ is non-archimedean the proof of the above is contained in the proof of Lemma \ref{trace-class-property}. If $v$ archimedean then we argue as in the proof of \cite[Lemma 3.3]{JaNe2019anv} and conclude via \cite[Lemma 2.6.6]{MV2010subconvexity}.
\end{proof}

From now on let $F$ be totally real. Recall that $\sigma_{\mathrm{f}}:=\otimes_{v<\infty}\sigma_v$, $C(\sigma_{\mathrm{f}}):=\prod_{v<\infty}C(\sigma_v)$.

\begin{lemma}\label{weak-weyl-law}
Let $\q$ be an integral ideal of the finite adeles of $F$ with norm $N(\q)$ and let $\mathfrak{X}=(X_v)_{v\mid \infty}$ where $X_v>1$ for $v\mid\infty$. We define
\begin{equation*}
\mc{F}_{\q,\mathfrak{X}}:=\left\{
\begin{aligned}
& \sigma\text{ unitary generic automorphic representation}\\
& \text{of $G_n(\A)$ with trivial central character }
\end{aligned}
\middle|
\begin{aligned}
C(\sigma_\mathrm{f})\mid N(\q);\, C(\sigma_v)<X_v,v\mid\infty
\end{aligned}
\right\}.
\end{equation*}
Then
$$\int_{\mc{F}_{\q,\mathfrak{X}}}\frac{1}{\mc{L}(\sigma)}\d\sigma \ll_\epsilon (N(\q)X)^{n-1+\epsilon},$$
where $X:=\prod_{v\mid\infty}X_v$.
\end{lemma}

The proof is essentially the same as that of \cite[Theorem 9]{JaNe2019anv}; see also \cite[Theorem 7]{Jana2020app}. We give a sketch of the proof for the sake of completeness.

\begin{proof}
Let $\q:=\prod_{v<\infty}\p_v^{f_v}$ with $f_v\ge 0$. For each $v<\infty$ we construct a test function $\alpha_{v}$ on $G_n(F_v)$ which is an $L^1$-normalized characteristic function of $K_0(\p_v^{f_v})$.

For each $v\mid \infty$ we fix sufficiently small $\tau_v>0$ and construct an approximate archimedean congruence subset $K_0(X_v,\tau_v)\subset G_n(F_v)$ as in \cite[(1.4)]{JaNe2019anv}. We also construct a test function $\alpha_{v}$ on $G_n(F_v)$ which is an $L^1$-normalized smoothened characteristic function of $K_0(X_v,\tau_v)$.

Let $\alpha^{\ast}_{v}$ be the self-convolution of $\alpha_{v}$, that is,
$$\alpha_v^\ast(g):=\int_{G_n(F_v)}\alpha_v(h)\alpha_v(gh^{-1})\d h.$$ We define
$$J_{\sigma_v}(\alpha^{\ast}_{v}):=\sum_{W\in \B(\sigma_v)}\sigma_v(\alpha^\ast_{v})W(1)\ol{W(1)}$$
where $\B(\sigma_v)$ is an orthonormal basis of $\sigma_v$. Applying the definition of $\alpha_v^\ast$ and changing the basis to $\{\sigma_v(h)W\}_{W\in\B(\sigma_v)}$ we also see that
$$J_{\sigma_v}(\alpha^{\ast}_{v})=\sum_{W\in \B(\sigma_v)}|\sigma_v(\alpha_v)W(1)|^2\ge 0.$$
Now for each $v<\infty$ we fix $\B(\sigma_v)\ni \frac{W_{\sigma_v}}{\|W\|_{\sigma_v}}$.
Classical non-archimedean newvector theory \cite{JPSS1981conducteur} and the normalization of the newvector, $W_{\sigma_v}(1)=1$, imply that
$$\sigma_v(\alpha_v)W_{\sigma_v}(1)=\vol(K_0(\p_v^{f_v}))^{-1}\int_{K_0(\p_v^{f_v})}W_{\sigma_v}(h)\d h
=\delta_{c(\sigma_v)\le f_v}.$$
Thus, applying \eqref{L2-ramified-whittaker} we have
$$J_{\sigma_v}(\alpha^{\ast}_{v})\ge \frac{|\sigma_v(\alpha_v)W_{\sigma_v}(1)|^2}{\|W_{\sigma_v}\|^2}\gg 1,\quad\text{if $c(\sigma_v)< f_v$ for $v<\infty$},$$
where the implied constant is uniform of $v$.
A similar statement holds when $v\mid\infty$, as can be observed in \cite[Proof of Theorem 9]{JaNe2019anv}. That is, we have
$$J_{\sigma_v}(\alpha^{\ast}_{v})\gg 1
\quad\text{if $C(\sigma_v)< X_v$ for $v\mid\infty$}.$$
Thus, combining the estimates at each place and applying the well-known bound $K^{\#\{v\mid\q\}}\gg_{K,\epsilon} N(\q)^{-\epsilon}$ for a fixed positive constant $K$, we obtain
\begin{equation}\label{prod-J-more-than-N-eps}
\prod_v J_{\sigma_v}(\alpha_v^\ast)\gg_\epsilon N(\q)^{-\epsilon}.
\end{equation}
Now we work as in the \cite[Proof of Theorem    9]{JaNe2019anv}. Consider the function
$$G_n(\A)\ni x_2\mapsto\sum_{\gamma\in G_n(F)}\left(\prod_v\alpha_v^\ast\right)(x_1^{-1}\gamma x_2).$$
If $x_i\in [N_n]$ then the support condition of $\alpha_v^\ast$ implies that the $\gamma$ sum above can be restricted to $N_n(F)$.
Now spectrally decomposing the above and integrating against $\psi_v(x_1^{-1}x_2)$ over $x_i\in [N_n]$ we obtain
$$\sum_{\gamma\in N_n(F)}\int_{[N_n]^2}\left(\prod_v\alpha_v^\ast\right)(x_1^{-1}\gamma x_2)\d x_1\d x_2=\int_{\sigma}\frac{\prod_{v}J_{\sigma_v}(\alpha^{\ast}_{v})}{\mathcal{L}(\sigma)}\d\sigma.$$
By \eqref{prod-J-more-than-N-eps}, we obtain
$$\int_{\mathcal{F}_{\q,\mathcal{X}}}\frac{1}{\mathcal{L}(\sigma)}\d\sigma\ll_\epsilon N(\q)^\epsilon\int_{\sigma}\frac{\prod_{v}J_{\sigma_v}(\alpha^{\ast}_{v})}{\mathcal{L}(\sigma)}\d\sigma = N(\q)^\epsilon\prod_{v}\int_{N_n(F_v)}\alpha^{\ast}_{v}(n_v)\overline{\psi_v(n_v)}\d n_v.$$
The last product of integrals can be bounded by $$\ll \prod_{v\mid\q}\vol(K_0(\p_v^{f_v}))^{-1}\prod_{v\mid\infty}\vol(K_0(X_v,\tau_v))^{-1}\ll \prod_{v\mid\q}N(\p_v)^{f_v(n-1)}\prod_{v\mid\infty}X_v^{n-1}.$$
The right-hand side above equals $(N(\q)X)^{n-1}$.
\end{proof}

\subsection{Proof of Theorem \ref{asymptotics}}

As before, $\q=\prod_{v\mid\q}\p_v^{f_v}$ is an ideal of $\o_F$ of norm $N(\q)$. We fix a set of places $S$, as in \S\ref{sec:choice-of-vectors},
$$S=\{\p_0\}\cup\{v\mid\q\}\cup\{v\mid \Delta_F\}\cup\{v\mid\infty\}.$$
Our point of departure is Theorem \ref{reciprocity} with $S$ as above, $\mathbf{s}=\left(\frac12,\frac12\right)$, and cusp forms $\Phi$ and $\phi$ with local components as in \S\ref{sec:choice-of-vectors}. Once again, we suppress $\Phi,\phi$ and $\mathbf{s}$ from the notations.

Recall that 
$$\sigma^{(\p_0)}_{\mathrm{f}}:=\bigotimes_{\p_0\neq v<\infty}\sigma_v,\quad
\mf{c}{(\sigma^{(\p_0)}_{\mathrm{f}})}:=\prod_{\p_0\neq v<\infty}\p^{c(\sigma_{v})},\quad C{(\sigma^{(\p_0)}_{\mathrm{f}})}:=\prod_{\p_0\neq v<\infty}C(\sigma_{v}).$$
Finally, recall the notations $H_\q(\sigma)$ and $h_\infty(\sigma)$ from \eqref{non-arch-weight} and \eqref{arch-weight}, respectively.

\begin{lemma}\label{original-moment}
Recall $\mc{M}$ from \eqref{def-moment} and $\tau$ from \S\ref{sec:choice-of-vectors}. 
For all $0<\delta\le 1/2$ there exists $\delta'>0$ such that
$$\mc{M}=\kappa_{F,n}\!\!\!\!\!\!\!\!\!\!\sum_{\substack{\sigma\text{ cuspidal; }\sigma_{\mathfrak{p}_0}=\tau;\\ \mf{c}(\sigma^{(\p_0)}_{\mathrm{f}})\mid\q,\,C(\sigma^{(\p_0)}_{\mathrm{f}})\ge N(\q)^{1/2-\delta} }}\frac{L(1/2,\Pi\otimes\widetilde{\sigma})L(1/2,\sigma\otimes\widetilde{\pi})}{L(1,\sigma,\Ad)}\varepsilon_{\p_0}(1,\tau\otimes\widetilde{\tau})H_\q(\sigma)h_\infty(\sigma)+O_{\Pi,\pi}(N(\q)^{-\delta'}),$$
for some positive constant $\kappa_{F,n}$ depending on the number field $F$ and $n$.
\end{lemma}

\begin{proof}
Applying Proposition \ref{cuspidal-projection} and \eqref{H-for-ramified-psi} to the definition of $\mc{M}$ in \eqref{def-moment} we obtain
$$\mc{M}=\sum_{\sigma\text{ cuspidal: }\sigma_{\mathfrak{p}_0}=\tau}\frac{L(1/2,\Pi\otimes\widetilde{\sigma})L(1/2,\sigma\otimes\widetilde{\pi})}{\mathcal{L}(\sigma)}H_\q(\sigma)\varepsilon_{\p_0}(1,\tau\otimes\widetilde{\tau})\Delta_F^{-\mu_H}h_\infty(\sigma).$$
Using Proposition \ref{main-local-prop-original}, convexity bound \eqref{convexity-bound}, and Lemma \ref{rapid-decay-local-weight} we truncate the above sum at $c(\sigma_v)\le f_v$ for all $v\mid\q$, $C(\sigma_v)=0$ for all $v\notin S$ or $v\mid \Delta_F$, and $C(\sigma_v)\ll N(\q)^{\epsilon}$ for all $v\mid\infty$ with an error of $O(N(\q)^{-A})$ for any large $A$. 

Let us denote by $C(\q,\tau)$ the set of cuspidal representations satisfying the conditions above. Applying \eqref{convexity-bound} and Proposition \ref{main-local-prop-original} along with the fact that $K^{\#\{v\mid\q\}}\ll N(\q)^\epsilon$ for any fixed constant $K$, we deduce the bound
\begin{equation*}
    \sum_{\substack{\sigma\in C(\q,\tau)\\C(\sigma^{{(\p_0)}}_{\mathrm{f}})\leq N(\q)^{1/2-\delta}}} \frac{L(1/2,\Pi\otimes\widetilde{\sigma})L(1/2,\sigma\otimes\widetilde{\pi})}{\mathcal{L}(\sigma)}H_\q(\sigma)h_\infty(\sigma) \ll_{\epsilon,\tau,\Pi,\pi} N(\q)^{-n/2+\epsilon}\sum_{\substack{\sigma\in C(\q,\tau)\\ C(\sigma^{{(\p_0)}}_{\mathrm{f}}) \leq N(\q)^{1/2-\delta}}} \frac{C(\sigma_{\mathrm{f}})}{\mathcal{L}(\sigma)}.
\end{equation*}

We rewrite the above quantity as
$$N(\q)^{-n/2+\epsilon}\sum_{\substack{\mathfrak{b}\mid\q \\|\mathfrak{b}|\le N(\q)^{1/2-\delta}}}\sum_{\substack{\sigma\text{ cuspidal}\\ C(\sigma_{\mathrm{f}})=|\mathfrak{b}|C(\tau)\\C(\sigma_v)\ll N(\q)^\epsilon,\,\,v\mid\infty }}\frac{C(\sigma_{\mathrm{f}})}{\mathcal{L}(\sigma)}.$$
We apply Lemma \ref{weak-weyl-law} to bound the inner sum above by $O(N(\q)^{\epsilon}|\mathfrak{b}|^{n})$. Using that the number of divisors of $\q$ is bounded by $N(\q)^{\epsilon}$ we see that the above display is bounded by $N(\q)^{-n\delta+\epsilon}$ which is $O(N(\q)^{-\delta'})$ after taking $\epsilon$ sufficiently small.

Finally, we conclude the proof using \eqref{harmonic-weight-cuspidal}.
\end{proof}

\begin{lemma}\label{dual-moment}
There exists some $\delta''>0$ such that
$$\wc{\mc{M}}\ll_{\Pi,\pi} N(\q)^{-\delta''},$$
where the implied constant in the error term depends on $\Pi$ and $\pi$.
\end{lemma}

\begin{proof}
Applying Proposition \ref{main-local-prop-dual} to the definition of $\wc{\mc{M}}$ as in \eqref{def-moment} we obtain
$$\wc{\mc{M}}=\kappa'_{F,n}\int_{\substack{\gen\\c(\sigma_v)=0,{\p_0\neq v<\infty}}}\frac{L(1/2,\Pi\otimes\widetilde{\sigma})L(1/2,\sigma\otimes\widetilde{\pi})}{\mathcal{L}(\sigma)}\wc{H}_\q(\sigma)\wc{H}_{\p_0}(\sigma_{\p_0})\wc{h}_\infty(\sigma)\d\sigma$$
for some constant $\kappa'_{F,n}$ depending on $F$ and $n$.
Once again, we use Lemma \ref{rapid-decay-local-weight} to truncate the integral, and apply Proposition \ref{main-local-prop-dual} and \eqref{convexity-bound} to deduce the bound
$$\wc{\mc{M}} \ll N(\q)^{\epsilon-(n-1)\eta}\int_{\substack{\gen;\,C(\sigma_v)\ll N(\q)^{\epsilon},v\mid\infty\\c(\sigma_v)=0, \p_0\neq v<\infty\\C(\sigma_{\p_0})\ll 1}}\frac{1}{\mathcal{L}(\sigma)}\d\sigma,$$
where $\eta$ is a uniform bound towards the generalized Ramanujan conjecture for $G_n$. For instance, we can take $\eta=(n^2+1)^{-1}$; see \cite{LRS1999ramanujan}.
Finally, we conclude by Lemma \ref{weak-weyl-law} and taking $\epsilon$ small enough.
\end{proof}

\begin{lemma}\label{original-residue-degenerate-zero}
We have $\mc{R}=0=\mc{D}$. Consequently, $\mc{P}=\mc{M}$.
\end{lemma}

\begin{proof}
Repeating the argument in \S\ref{sec:spectral-decomposition-period} and of the proof of Proposition \ref{cuspidal-projection} we write
$$\mathcal{A}_{s_1}\Phi(g) = \sum_{\substack{\sigma\text{ cuspidal}\\\sigma_{\p_0}=\tau}}\sum_{\varphi\in \widetilde{\B}(\sigma)}\frac{\Psi(s_1,W_{\Phi},\overline{W_\varphi})}{\|\varphi\|^2}\varphi(g).$$
As before, the right-hand side is entire in $s_1$ and absolutely convergent. Moreover, as $\varphi$ are cusp forms we now integrate the right-hand side term by term against $\phi|\det|^{s_2-1/2}$. Thus we obtain
$$\mc{P}(\mathbf{s},\Phi,\phi)=\sum_{\sigma\text{ cuspidal: }\sigma_{\mathfrak{p}_0}=\tau}\frac{\Lambda(s_1,\Pi\otimes\widetilde{\sigma})\Lambda(s_2,\sigma\otimes\widetilde{\pi})}{\mathcal{L}(\sigma)}H_S(\sigma).$$
The right-hand side is again entire in $\mathbf{s}$ which follows from the analytic properties of Rankin--Selberg zeta integrals. Plugging-in $\mathbf{s}=\left(\frac12,\frac12\right)$ above and looking at the expression of $\mc{M}$ in the proof of Lemma \ref{original-moment} we obtain $\mc{P}=\mc{M}$.

On the other hand, from \eqref{def-R} we see that $\mc{R}=0$ as $H_{\p_0}(\sigma(\pi_{\p_0},z))=0$ for any $z\in\C$ which follows from Lemma \ref{holomorphicity-local-weight} and Proposition \ref{cuspidal-projection} for $z\in i\R$. Consequently, Proposition \ref{analytically-continued-spectral-decomposition} implies that $\D=0$.
\end{proof}

\begin{proof}[Proof of Theorem \ref{asymptotics}]
Theorem \ref{reciprocity} and Lemma \ref{original-residue-degenerate-zero} imply that
$\mc{M}=\wc{\mc{M}}+\wc{\mc{D}}+\wc{\mc{R}}.$
We apply Lemma \ref{original-moment}, Lemma \ref{dual-moment}, bound of $\wc{\mc{R}}$ from Proposition \ref{analytically-continued-spectral-decomposition}, expression of $\mathcal{L}$ from \eqref{harmonic-weight-cuspidal}, and expression of $\wc{\mc{D}}$ from Proposition \ref{degenerate-term}.
\end{proof}

\begin{proof}[Proof of Corollary \ref{non-vaninshing}]
The main term in Theorem \ref{asymptotics} is non-zero. This follows from Shahidi's result \cite[Proposition 7.2.4]{Shahidi2010Eisenstein} that the Rankin--Selberg $L$-values at $1$ are non-zero, and Proposition \ref{degenerate-term}. Hence we conclude by taking $N(\q)$ sufficiently large.
\end{proof}

\bibliographystyle{abbrv}
\bibliography{references}

\end{document}